\documentclass[11pt]{article}
\usepackage{amssymb}
\usepackage{amsmath}
\usepackage{amsfonts}
\usepackage{color}
\usepackage{fancyhdr}
\usepackage{makeidx}
\usepackage{graphicx}
\usepackage{cases}
\usepackage{geometry}

\newtheorem{proposition}{Proposition}[section]
\newtheorem{theorem}[proposition]{Theorem}

\newtheorem{lemma}[proposition]{Lemma}

\newtheorem{definition}[proposition]{Definition}
\newtheorem{remark}[proposition]{Remark}

\newtheorem{condition}[proposition]{Condition}
\numberwithin{equation}{section}
\numberwithin{proposition}{section}
\newenvironment{proof}[1][Proof]{\noindent\textbf{#1.} }{\ 
\rule{0.5em}{0.5em}}

\makeindex
\parskip        \baselineskip
\begin{document}

\title{Moderate Deviation Principles for Stochastic Differential Equations
with Jumps}
\author{Amarjit Budhiraja\thanks{%
DEPARTMENT OF STATISTICS AND OPERATIONS RESEARCH, UNIVERSITY OF NORTH CAROLINA, CHAPEL HILL, email: {\tt budhiraj@email.unc.edu}. \newline
Research supported in part by the National Science Foundation (DMS-1004418,
DMS-1016441, DMS-1305120) and the Army Research Office (W911NF-10-1-0158).},
Paul Dupuis\thanks{%
DIVISION OF APPLIED MATHEMATICS, BROWN UNIVERSITY,  email: {\tt dupuis@dam.brown.edu}. \newline
Research supported in part by the National Science Foundation (DMS-1317199),
and the Air Force Office of Scientific Research (FA9550-12-1-0399), and the
Army Research Office (W911NF-12-1-0222).}, Arnab Ganguly\thanks{%
DEPARTMENT OF MATHEMATICS, UNIVERSITY OF LOUISVILLE, email: {\tt a0gang02@louisville.edu}. \newline 
Research supported in part by the Air Force Office of Scientific Research
(FA9550-12-1-0399).} }
\maketitle

\begin{abstract}
Moderate deviation principles for stochastic differential equations driven
by a Poisson random measure (PRM) in finite and infinite dimensions are
obtained. Proofs are based on a variational representation for expected
values of positive functionals of a PRM.  \newline
\ \newline

\noindent\textbf{MSC 2010 subject classifications:} \newline
60F10, 60H15, 60J75, 60J25 \newline

\end{abstract}

\setcounter{equation}{0}

\section{Introduction}

Large deviation principles for small noise diffusion equations have been
extensively studied in the literature. Since the original work of Freidlin
and Wentzell \cite{FV70, FW84}, model assumptions have been significantly
relaxed and many extensions have been studied in both finite-dimensional and
infinite-dimensional settings. In \cite{BD00, BDM08} a general approach for
studying large deviation problems in such settings has been introduced that
is based on a variational representation for expectations of positive
functionals of an infinite
dimensional Brownian motion. This approach has now been adopted for the
study of large deviation problems for a broad range of stochastic partial
differential equation based models, particularly those arising in stochastic
fluid dynamics, and also for settings where the coefficients in the model
have little regularity. We refer the reader to \cite{BDM09} for a partial
list of references. Large deviation problems for finite dimensional
diffusions with jumps have been studied by several authors (see for example 
\cite{LP92, DE97}). In contrast, it is only recently that the analogous
problems for infinite dimensional stochastic differential equations (SDE)
have received attention \cite{RZ07,SZ11}. In \cite{BDM09} a variational
representation for expected values of positive functionals of a general
Poisson random measure (or more generally, functions that depend both on a
Poisson random measure and an infinite dimensional Brownian motion) was
derived. As in the Brownian motion case, the representation is motivated in
part by applications to large deviation problems, and \cite{BDM09}
illustrates how the representation can be applied in a simple finite
dimensional setting. In \cite{BCD} the representation was used to study
large deviation properties of a family of infinite dimensional SDE driven by
a Poisson random measure (PRM).

The goal of the current work is to study moderate deviation problems for
stochastic dynamical systems. In such a study one is concerned with
probabilities of deviations of a smaller order than in large deviation
theory. Consider for example an independent and identically distributed
(iid) sequence $\{Y_{i}\}_{i\geq 1}$ of $\mathbb{R}^{d}$-valued zero mean
random variables with common probability law $\rho $. A large deviation
principle (LDP) for $S_{n}=\sum_{i=1}^{n}Y_{i}$ will formally say that for $%
c>0$ 
\begin{equation*}
\mathbb{P}(|S_{n}|>nc)\approx \exp \{-n\inf \{I(y):|y|\geq c\}\},
\end{equation*}%
where for $y\in \mathbb{R}^{d}$, $I(y)=\sup_{\alpha \in \mathbb{R}%
^{d}}\{\langle \alpha ,y\rangle -\log \int_{\mathbb{R}^{d}}\exp \langle
\alpha ,y\rangle \rho (dy)\}$. Now let $\{a_{n}\}$ be a positive sequence
such that $a_{n}\uparrow \infty $ and $n^{-1/2}a_{n}\rightarrow 0$ as $%
n\rightarrow \infty $ (e.g. $a_{n}=n^{1/4}$). Then a moderate deviation
principle (MDP) for $S_{n}$ will say that 
\begin{equation*}
\mathbb{P}(|S_{n}|>n^{1/2}a_{n}c)\approx \exp \{-a_{n}^{2}\inf
\{I^{0}(y):|y|\geq c\}\},
\end{equation*}%
where $I^{0}(y)=\frac{1}{2}\left\langle y,\Sigma ^{-1}y\right\rangle $ and $%
\Sigma =\mbox{Cov}(Y)$. Thus the moderate deviation principle gives
estimates on probabilities of deviations of order $n^{1/2}a_{n}$ which is of
lower order than $n$ and with a rate function that is a quadratic form.
Since $a_n \to \infty$ as slowly as desired, moderate deviations bridge the gap between a central limit approximation
and a large deviations approximation.
Moderate deviation principles have been extensively studied in Mathematical
Statistics. Early research considered the setting of iid sequences and
arrays (see \cite{Os72,Mi76,Am79, Am80, Am82,RS65,Pe99,Fr05}). Empirical
processes in general topological spaces have been studied in \cite{BM78,
BM80, DB83, Ch91, Le92, dAc92, Ar03}. The setting of weakly dependent
sequences was covered in \cite{Gh74, Gh75, Ch97, GB77, BS78, Ga96, De96,
Gr97, Dj02, GH06,DZ97}, and MDPs for occupation measures of Markov chains
and general additive functionals of Markov chains were considered in \cite%
{Ga95, Lim95, CdA97, CdA98, Ch01, Gu00,CG04,GZ98}.

Moderate deviation principles for continuous time stochastic dynamical
systems are less well studied. The paper \cite{LS99} considers a
finite-dimensional two scale diffusion model under stochastic averaging.
Additional results involving moderate deviations and the averaging principle
were obtained in \cite{Gu01, Gu03, CL09}. The paper \cite{HS04} considered a
certain diffusion process with Brownian potentials and derived moderate
deviation estimates for its longtime behavior. Moderate deviation results in
the context of statistical inference for finite dimensional diffusions have
been considered in \cite{DGW99, Hu10, GJ09}. None of the above results
consider stochastic dynamical systems with jumps or infinite dimensional
models.

In this paper we study moderate deviation principles for finite and infinite
dimensional SDE with jumps. For simplicity we consider only settings where
the noise is given in terms of a PRM and there is no Brownian component.
However, as noted in Remark \ref{rem:rem1150}, the more general case where
both Poisson and Brownian noises are present can be treated similarly. In
finite dimensions, the basic stochastic dynamical system we study takes the
form 
\begin{equation*}
X^{\varepsilon }(t)=x_{0}+\int_{0}^{t}b(X^{\varepsilon }(s))ds+\int_{\mathbb{%
X}\times \lbrack 0,t]}\varepsilon G(X^{\varepsilon }(s-),y)N^{\varepsilon
^{-1}}(dy,ds).
\end{equation*}%
Here $b:\mathbb{R}^{d}\rightarrow \mathbb{R}^{d}$ and $G:\mathbb{R}%
^{d}\times \mathbb{X}\rightarrow \mathbb{R}^{d}$ are suitable coefficients
and $N^{\varepsilon ^{-1}}$ is a Poisson random measure on $\mathbb{X}_{T}=%
\mathbb{X}\times \lbrack 0,T]$ with intensity measure $\varepsilon ^{-1}\nu
_{T}=\varepsilon ^{-1}\nu \otimes \lambda _{T}$, where $\mathbb{X}$ is a
locally compact Polish space, $\nu $ is a locally finite measure on $\mathbb{%
X}$, $\lambda _{T}$ is the Lebesgue measure on $[0,T]$ and $\varepsilon >0$
is the scaling parameter. Under conditions $X^{\varepsilon }$ will converge
in probability (in a suitable path space) to $X^{0}$ given as the solution
of the ODE 
\begin{equation*}
\dot{X}^{0}(t)=b(X^{0}(t))+\int_{\mathbb{X}}G(X^{0}(t),y)\nu
(dy),\;X^{0}(0)=x_{0}.
\end{equation*}%
The moderate deviations problem for $\{X^{\varepsilon }\}_{\varepsilon >0}$
corresponds to studying asymptotics of $$(\varepsilon /a^{2}(\varepsilon
))\log \mathbb{P}(Y^{\varepsilon }\in \cdot ),$$ where $Y^{\varepsilon
}=(X^{\varepsilon }-X^{0})/a(\varepsilon )$ and $a(\varepsilon )\rightarrow
0 $, $\varepsilon /a^{2}(\varepsilon )\rightarrow 0$ as $\varepsilon
\rightarrow 0$. In this paper we establish a moderate deviations principle
under suitable conditions on $b$ and $G$. We
in fact give a rather general sufficient condition for a moderate deviation
principle to hold for systems driven by Poisson random measures (see Theorem %
\ref{LDPpoi}). This sufficient condition covers many finite and infinite
dimensional models of interest. A typical infinite dimensional model
corresponds to the SPDE 
\begin{align}
dX^{\varepsilon }(u,t)& =(LX^{\varepsilon }(u,t)+\beta (X^{\varepsilon
}(u,t)))dt+\varepsilon \int_{\mathbb{X}}G(X^{\varepsilon
}(u,t-),u,y)N^{\varepsilon ^{-1}}(ds,dy)  \label{SPDE} \\
X^{\varepsilon }(u,0)& =x(u),\quad u\in O\subset \mathbb{R}^{d}.  \notag
\end{align}%
where $L$ is a suitable differential operator, $O$ is a bounded domain in $%
\mathbb{R}^{d}$ and the equation is considered with a suitable boundary
condition on $\partial O$. Here $N^{\varepsilon ^{-1}}$ is a PRM as above.
The solution of such a SPDE has to be interpreted carefully, since typically
solutions for which $LX^{\varepsilon }(u,t)$ can be defined classically do
not exist. We follow the framework of \cite{KaXi95}, where the solution
space is described as the space of RCLL trajectories with values in the dual
of a suitable nuclear space (see Section \ref{sec:sec2.4} for precise
definitions). Roughly speaking, a nuclear space is given as an intersection
of a countable collection of Hilbert spaces, where the different spaces may
be viewed as \textquotedblleft function spaces\textquotedblright\ with
varying degree of regularity. Since the action of the differential operator $%
L$ on a function will typically produce a function with lesser regularity,
this framework of nested Hilbert spaces enables one to efficiently
investigate existence and uniqueness of solutions of SPDE of the form %
\eqref{SPDE}. Another common approach for studying equations of the form %
\eqref{SPDE} is through a mild solution formulation \cite{PZ07}. Although
not investigated here we expect that analogous results can be established
using such a formulation.

Large and moderate deviation approximations can provide qualitative and
quantitative information regarding complex stochastic models such as (\ref%
{SPDE}). For example, an equation studied in some detail at the end of this
paper models the concentration of pollutants in a waterway. Depending on the
event of interest either the large and moderate deviation approximation
could be appropriate, in which case one could use the rate function to
identify the most likely interactions between the pollution source and the
dynamics of the waterway that lead to a particular outcome, such as
exceeding an allowed concentration. However, the rate function only gives an asymptotic approximation for probabilities of such outcomes and the resulting error 
due to the use of this
approximation cannot be eliminated.

An alternative is to use numerical schemes such as Monte Carlo, which have
the property that if a large enough number of good quality samples can be
generated, then an arbitrary level of accuracy can be achieved. While this
may be true in principle, it is in practice difficult when considering
events of small probability, since many samples are required for errors that
are small relative to the quantity being computed. The issue is especially
relevant for a problem modeled by an equation as complex as (\ref{SPDE}),
since the generation of even a single sample could be relatively expensive.
Hence an interesting potential use of the results of the present paper are
to importance sampling and related accelerated Monte Carlo methods \cite%
{asmgly,liu}. If in fact the moderate deviation approximation is relevant,
the relatively simple form of the corresponding rate function suggests that
many of the constructions needed to implement an effective importance
sampling scheme \cite{dupwan} would be simpler than in the corresponding
large deviation context.

We now make some comments on the technique of proof. As in \cite{BCD}, the
starting point is the variational representation for expectations of
positive functionals of a PRM from \cite{BDM09}. The usefulness of
variational representation in proving large deviation or moderate deviation
type results lies in the fact that it allows one to bypass the traditional
route of approximating the original sequence of solutions by
discretizations; the latter approach is particularly cumbersome for SPDEs
and more so for SPDEs driven by Poisson random measure. Moreover the
variational representation approach does not require proving exponential
tightness and other exponential probability estimates that are frequently
some of the most technical parts of a traditional large deviations argument.
A key step in our approach is to prove the tightness for controlled versions
of the state processes given that the costs for controls are suitably
bounded. For example, to prove a moderate deviation principle for SPDEs of
the form \eqref{SPDE}, the tightness of the sequence of controlled processes 
$\bar{Y}^{\varepsilon ,\varphi ^{\varepsilon }}$ needs to be established,
where 
\begin{equation}
\bar{Y}^{\varepsilon ,\varphi ^{\varepsilon }}=\frac{1}{a(\varepsilon )}(%
\bar{X}^{\varepsilon ,\varphi ^{\varepsilon }}-X^{0})  \label{eq:centered}
\end{equation}%
\begin{equation*}
d\bar{X}^{\varepsilon ,\varphi ^{\varepsilon }}(u,t)=\left( L\bar{X}%
^{\varepsilon ,\varphi ^{\varepsilon }}(u,t)+b(\bar{X}^{\varepsilon ,\varphi
}(u,t))\right) dt+\int_{\mathbb{X}_{T}}\varepsilon G(\bar{X}^{\varepsilon
,\varphi ^{\varepsilon }}(u,t-),y)\ N^{\varepsilon ^{-1}\varphi
^{\varepsilon }}(dt,dy)
\end{equation*}%
and the controls $\varphi ^{\varepsilon }:\mathbb{X}\times \lbrack 0,T] \to [0,\infty)$ are
predictable processes satisfying $L_{T}(\varphi ^{\varepsilon })\leq
Ma^{2}(\varepsilon )$ for some constant $M$. Here $L_{T}$ denotes the large
deviation rate function associated with Poisson random measures (see %
\eqref{Ltdef}) and $N^{\varepsilon ^{-1}\varphi ^{\varepsilon }}$ is a
controlled Poisson random measure, namely a counting process with intensity $%
\varepsilon ^{-1}\varphi ^{\varepsilon }(x,s)\nu _{T}(dx,ds)$ (see 
\eqref{Eqn:
control} for a precise definition). In comparison (cf. \cite{BCD}), to prove
a large deviation principle for $X^{\varepsilon }$, the key step is proving
the tightness of the controlled processes $\bar{X}^{\varepsilon ,\varphi
^{\varepsilon }}$ with the controls $\varphi ^{\varepsilon }$ satisfying $%
L_{T}(\varphi ^{\varepsilon })\leq M$ for some constant $M$. The proof of
this tightness property relies on the fact that the estimate $L_{T}(\varphi
^{\varepsilon })\leq M$ implies tightness of $\varphi ^{\varepsilon }$ in a
suitable space. Although in the moderate deviations problem one has the
stronger bound $L_{T}(\varphi ^{\varepsilon })\leq Ma^{2}(\varepsilon )$ on
the cost of controls, the mere tightness of $\varphi ^{\varepsilon }$ does
not imply the tightness of $\bar{Y}^{\varepsilon ,\varphi ^{\varepsilon }}$.
Instead one needs to study tightness properties of $\psi ^{\varepsilon
}=(\varphi ^{\varepsilon }-1)/a(\varepsilon )$. In general $\psi
^{\varepsilon }$ may not be in $L^{2}(\nu _{T})$ and one of the challenges
is to identify a space where suitable tightness properties of the centered
and normalized controls $\{\psi ^{\varepsilon }\}$ can be established. The
key idea is to split $\psi ^{\varepsilon }$ into two terms, one of which
lies in a closed ball in $L^{2}(\nu _{T})$ (independent of $\varepsilon $)
and the other approaches $0$ in a suitable manner. Estimates on each of the
two terms (see Lemma \ref{sineqs2}) are key ingredients in the proof and are
used many times in this work, in particular to obtain uniform in $%
\varepsilon $ moment estimates on centered and scaled processes of the form (%
\ref{eq:centered}).

The rest of the paper is organized as follows. Section \ref{Sec:PRM}
contains some background on PRMs and the variational representation from 
\cite{BDM09}. In Section \ref{Sec:thm} we present a general moderate
deviation principle for measurable functionals of a PRM. Although this
result concerns a large deviations principle with a certain speed, we refer
to it as a MDP since its typical application is to the proof of moderate
deviation type results. This general result covers many stochastic dynamical
system models in finite and infinite dimensions. Indeed, by using the
general theorem from Section \ref{Sec:thm}, a moderate deviation principle
for finite dimensional SDE driven by PRM is established in Section \ref%
{sec_MDP} and an infinite-dimensional model is considered in Section \ref%
{sec:sec2.4}. Sections \ref{sec:proofthm1} to \ref{sec:sec748} are devoted
to proofs. The result for the infinite dimensional setting requires many
assumptions on the model. In Section \ref{sec:secexample} we show that these
assumptions are satisfied for an SPDE that has been proposed as a model for
the spread of a pollutant with Poissonian point sources in a waterway.

\medskip \noindent \textbf{Notation:} \noindent

The following notation is used. For a topological space $\mathcal{E}$,
denote the corresponding Borel $\sigma $-field by $\mathcal{B}(\mathcal{E})$%
. We use the symbol \textquotedblleft $\Rightarrow $\textquotedblright\ to
denote convergence in distribution. For a Polish space $\mathbb{X}$, denote
by $C([0,T]:\mathbb{X})$ and $D([0,T]:\mathbb{X})$ the space of continuous
functions and right continuous functions with left limits from $[0,T]$ to $%
\mathbb{X}$, endowed with the uniform and Skorokhod topology, respectively.
For a metric space $\mathcal{E}$, denote by $M_{b}(\mathcal{E})$ and $C_{b}(%
\mathcal{E})$ the space of real bounded $\mathcal{B}(\mathcal{E})/\mathcal{B}%
(\mathbb{R})$-measurable functions and real bounded and continuous functions
respectively. For Banach spaces $B_{1},B_{2}$, $L(B_{1},B_{2})$ will denote
the space of bounded linear operators from $B_{1}$ to $B_{2}$. For a measure 
$\nu $ on $\mathcal{E}$ and a Hilbert space $H$, let $L^{2}(\mathcal{E},\nu
,H)$ denote the space of measurable functions $f$ from $\mathcal{E}$ to $H$
such that $\int_{\mathcal{E}}\Vert f(v)\Vert ^{2}\nu (dv)<\infty $, where $%
\Vert \cdot \Vert $ is the norm on $H$. When $H=\mathbb{R}$ and $\mathcal{E}$
is clear from the context we write $L^{2}(\nu )$.

For a function $x:[0,T]\rightarrow\mathcal{E}$, we use the notation $x_{t}$
and $x(t)$ interchangeably for the evaluation of $x$ at $t\in\lbrack0,T]$. A
similar convention will be followed for stochastic processes. We say a
collection $\{X^{\varepsilon}\}$ of $\mathcal{E}$-valued random variables is
tight if the distributions of $X^{\varepsilon}$ are tight in $\mathcal{P}(%
\mathcal{E})$ (the space of probability measures on $\mathcal{E}$).

A function $I:\mathcal{E}\rightarrow\lbrack0,\infty]$ is called a rate
function on $\mathcal{E}$ if for each $M<\infty$, the level set $\{x\in 
\mathcal{E}:I(x)\leq M\}$ is a compact subset of $\mathcal{E}$.

Given a collection $\{b(\varepsilon )\}_{\varepsilon >0}$ of positive reals,
a collection $\{X^{\varepsilon }\}_{\varepsilon >0}$ of $\mathcal{E}$-valued
random variables is said to satisfy the Laplace principle upper bound
(respectively, lower bound) on $\mathcal{E}$ with speed $b(\varepsilon )$
and rate function $I$ if for all $h\in C_{b}(\mathcal{E})$ 
\begin{equation*}
\limsup_{\varepsilon \rightarrow 0}b(\varepsilon )\log \mathbb{E}\left\{
\exp \left[ -\frac{1}{b(\varepsilon )}h(X^{\varepsilon })\right] \right\}
\leq -\inf_{x\in \mathcal{E}}\{h(x)+I(x)\},
\end{equation*}%
and, respectively, 
\begin{equation*}
\liminf_{\varepsilon \rightarrow 0}b(\varepsilon )\log \mathbb{E}\left\{
\exp \left[ -\frac{1}{b(\varepsilon )}h(X^{\varepsilon })\right] \right\}
\geq -\inf_{x\in \mathcal{E}}\{h(x)+I(x)\}.
\end{equation*}%
The Laplace principle is said to hold for $\{X^{\varepsilon }\}$ with speed $%
b(\varepsilon )$ and rate function $I$ if both the Laplace upper and lower
bounds hold. It is well known that when $\mathcal{E}$ is a Polish space, the
family $\{X^{\varepsilon }\}$ satisfies the Laplace principle upper
(respectively lower) bound with a rate function $I$ on $\mathcal{E}$ if and
only if $\{X^{\varepsilon }\}$ satisfies the large deviation upper
(respectively lower) bound for all closed sets (respectively open sets) with
the rate function $I$. For a proof of this statement we refer to Section 1.2
of \cite{DE97}.

\section{Preliminaries and Main Results}

\label{prelim}

\subsection{Poisson Random Measure and a Variational Representation}

\label{Sec:PRM}

Let $\mathbb{X}$ be a locally compact Polish space and let $\mathcal{M}_{FC}(%
\mathbb{X})$ be the space of all measures $\nu $ on $(\mathbb{X},\mathcal{B}(%
\mathbb{X}))$ such that $\nu (K)<\infty $ for every compact $K\subset 
\mathbb{X}$. Endow $\mathcal{M}_{FC}(\mathbb{X})$ with the usual vague
topology. This topology can be metrized such that $\mathcal{M}_{FC}(\mathbb{X%
})$ is a Polish space \cite{BDM09}. Fix $T\in (0,\infty )$ and let $\mathbb{X%
}_{T}=\mathbb{X}\times \lbrack 0,T]$. Fix a measure $\nu \in \mathcal{M}%
_{FC}(\mathbb{X})$, and let $\nu _{T}=\nu \otimes \lambda _{T}$, where $%
\lambda _{T}$ is Lebesgue measure on $[0,T]$.

A Poisson random measure $\mathbf{n}$ on $\mathbb{X}_{T}$ with mean measure
(or intensity measure) $\nu_{T}$ is a $\mathcal{M}_{FC}(\mathbb{X}_{T})$%
-valued random variable such that for each $B\in\mathcal{B}(\mathbb{X}_{T})$
with $\nu_{T}(B)<\infty$, $\mathbf{n}(B)$ is Poisson distributed with mean $%
\nu_{T}(B)$ and for disjoint $B_{1},...,B_{k}\in\mathcal{B}(\mathbb{X}_{T})$%
, $\mathbf{n}(B_{1}),...,\mathbf{n}(B_{k})$ are mutually independent random
variables (cf. \cite{Ikeda}). Denote by $\mathbb{P}$ the measure induced by $%
\mathbf{n}$ on $(\mathcal{M}_{FC}(\mathbb{X}_{T}),\mathcal{B}(\mathcal{M}%
_{FC}(\mathbb{X}_{T})))$. Then letting $\mathbb{M}=\mathcal{M}_{FC}(\mathbb{X%
}_{T})$, $\mathbb{P}$ is the unique probability measure on $(\mathbb{M},%
\mathcal{B}(\mathbb{M}))$ under which the canonical map, $N:\mathbb{M}%
\rightarrow\mathbb{M},N(m)\doteq m$, is a Poisson random measure with
intensity measure $\nu_{T}$. Also, for $\theta>0$, $\mathbb{P}_{\theta}$
will denote a probability measure on $(\mathbb{M},\mathcal{B}(\mathbb{M}))$
under which $N$ is a Poisson random measure with intensity $\theta\nu_{T}$ .
The corresponding expectation operators will be denoted by $\mathbb{E}$ and $%
\mathbb{E}_{\theta}$, respectively.

Let $F\in M_{b}(\mathbb{M})$. We now present a variational representation
from \cite{BDM09} for $-\log\mathbb{E}_{\theta}(\exp[-F(N)])$, in terms of a
Poisson random measure constructed on a larger space. Let $\mathbb{Y}=%
\mathbb{X}\times\lbrack0,\infty)$ and $\mathbb{Y}_{T}=\mathbb{Y}\times
\lbrack0,T]$. Let $\bar{\mathbb{M}}=\mathcal{M}_{FC}(\mathbb{Y}_{T})$ and
let $\mathbb{\bar{P}}$ be the unique probability measure on $(\bar{\mathbb{M}%
},\mathcal{B}(\bar{\mathbb{M}}))$ under which the canonical map, $\bar{N}:%
\bar{\mathbb{M}}\rightarrow\bar{\mathbb{M}},\bar{N}(m)\doteq m$, is a
Poisson random measure with intensity measure $\bar{\nu}_{T}=\nu\otimes%
\lambda _{\infty}\otimes\lambda_{T}$, where $\lambda_{\infty}$ is Lebesgue
measure on $[0,\infty)$. The corresponding expectation operator will be
denoted by $\mathbb{\bar{E}}$. Let $\mathcal{F}_{t}\doteq\sigma\{\bar{N}%
((0,s]\times A):0\leq s\leq t,A\in\mathcal{B}(\mathbb{Y})\}$ be the $\sigma$%
-algebra generated by $\bar{N}$, and let $\bar{\mathcal{F}}_{t}$ denote the
completion under $\mathbb{\bar{P}}$. We denote by $\bar{\mathcal{P}}$ the
predictable $\sigma$-field on $[0,T]\times\bar{\mathbb{M}}$ with the
filtration $\{\bar{\mathcal{F}}_{t}:0\leq t\leq T\}$ on $(\bar{\mathbb{M}},%
\mathcal{B}(\bar{\mathbb{M}}))$. Let $\bar{\mathcal{A}}_{+}$ [resp. $\bar{%
\mathcal{A}}$] be the class of all $(\mathcal{B}(\mathbb{X})\otimes\bar{%
\mathcal{P}})/\mathcal{B}[0,\infty)$ [resp. $(\mathcal{B}(\mathbb{X})\otimes 
\bar{\mathcal{P}})/\mathcal{B}(\mathbb{R})$]-measurable maps from $\mathbb{X}%
_{T}\times\bar{\mathbb{M}}$ to $[0,\infty)$ [resp. $\mathbb{R}$]. For $%
\varphi\in\bar{\mathcal{A}}_{+}$, define a counting process $N^{\varphi}$ on 
$\mathbb{X}_{T}$ by 
\begin{equation}
N^{\varphi}((0,t]\times U)=\int_{U\times\lbrack0,\infty)\times [
0,t]}1_{[0,\varphi(x,s)]}(r)\bar{N}(dx\,dr\,ds),\quad t\in\lbrack0,T],U\in 
\mathcal{B}(\mathbb{X}).  \label{Eqn: control}
\end{equation}
We think of $N^{\varphi}$ as a controlled random measure, with $\varphi$
selecting the intensity for the points at location $x$ and time $s$, in a
possibly random but non-anticipating way. When $\varphi(x,s,\bar{m}%
)\equiv\theta\in(0,\infty)$, we write $N^{\varphi}=N^{\theta}$. Note that $%
N^{\theta}$ has the same distribution with respect to $\mathbb{\bar{P}}$ as $%
N$ has with respect to $\mathbb{P}_{\theta}$.

Define $\ell :[0,\infty )\rightarrow \lbrack 0,\infty )$ by 
\begin{equation}
\ell (r)=r\log r-r+1,\quad r\in \lbrack 0,\infty ).  \label{eq:eq121}
\end{equation}%
For any $\varphi \in \bar{\mathcal{A}}_{+}$ and $t\in \lbrack 0,T]$ the
quantity 
\begin{equation}
L_{t}(\varphi )=\int_{\mathbb{X}\times \lbrack 0,t]}\ell (\varphi
(x,s,\omega ))\nu _{T}(dx\,ds)  \label{Ltdef}
\end{equation}%
is well defined as a $[0,\infty ]$-valued random variable. Let $\left\{
K_{n}\subset \mathbb{X},n=1,2,\ldots \right\} $ be an increasing sequence of
compact sets such that $\cup _{n=1}^{\infty }K_{n}=\mathbb{X}$. For each $n$
let 
\begin{align*}
\bar{\mathcal{A}}_{b,n}& \doteq \left\{ \varphi \in \bar{\mathcal{A}}_{+}:%
\text{ for all }(t,\omega )\in \lbrack 0,T]\times \mathbb{\bar{M}}\text{, }%
n\geq \varphi (x,t,\omega )\geq 1/n\right. \\
& \hspace{1in}\left. \text{ if }x\in K_{n}\text{ and }\varphi (x,t,\omega )=1%
\text{ if }x\in K_{n}^{c}\right\}
\end{align*}%
and let $\bar{\mathcal{A}}_{b}=\cup _{n=1}^{\infty }\bar{\mathcal{A}}_{b,n}$.

The following is a representation formula proved in \cite{BDM09}. For the
second equality in the theorem see the proof of Theorem 2.4 in \cite{BCD}.

\begin{theorem}
\label{VR: PRM} Let $F\in M_{b}(\mathbb{M})$. Then for $\theta>0$%
\begin{align*}
-\log\mathbb{E}_{\theta}(e^{-F(N)})=-\log\mathbb{\bar{E}}(e^{-
F(N^{\theta})})= & \inf_{\varphi\in\bar{\mathcal{A}}_{+}}\mathbb{\bar{E}}%
\left[ \theta L_{T}(\varphi)+F(N^{\theta\varphi})\right] \\
= & \inf_{\varphi\in\bar{\mathcal{A}}_{b}}\mathbb{\bar{E}}\left[ \theta
L_{T}(\varphi)+F(N^{\theta\varphi})\right] .
\end{align*}
\end{theorem}

\subsection{A General Moderate Deviation Result}

\label{Sec:thm} For $\varepsilon>0$, let $\mathcal{G}^{\varepsilon}$ be a
measurable map from $\mathbb{M}$ to $\mathbb{U}$, where $\mathbb{U}$ is some
Polish space. Let $a:\mathbb{R}_{+}\rightarrow\mathbb{R}_{+}$ be such that
as $\varepsilon\rightarrow0$ 
\begin{equation}
a(\varepsilon)\rightarrow0\mbox{ and }b(\varepsilon)=\frac{\varepsilon}{%
a^{2}(\varepsilon)}\rightarrow0.  \label{eq:eq1138}
\end{equation}
In this section we will formulate a general sufficient condition for the
collection $\mathcal{G}^{\varepsilon}(\varepsilon N^{\varepsilon^{-1}})$ to
satisfy a large deviation principle with speed $b(\varepsilon)$ and a rate
function that is given through a suitable quadratic form.

For $\varepsilon>0$ and $M<\infty$, consider the spaces 
\begin{align}
\mathcal{S}_{+,\varepsilon}^{M} & \doteq\{\varphi:\mathbb{X}\times
\lbrack0,T]\rightarrow\mathbb{R}_{+}|\ L_{T}(\varphi)\leq
Ma^{2}(\varepsilon)\}  \label{ctrlclass} \\
\mathcal{S}_{\varepsilon}^{M} & \doteq\{\psi:\mathbb{X}\times\lbrack
0,T]\rightarrow\mathbb{R}|\ \psi=(\varphi-1)/a(\varepsilon),\varphi \in%
\mathcal{S}_{+,\varepsilon}^{M}\}.  \notag
\end{align}
We also let 
\begin{align}
{\mathcal{U}}_{+,\varepsilon}^{M} & \doteq\left\{ \phi\in\bar{\mathcal{A}}%
_{b}:\phi(\cdot,\cdot,\omega)\in\mathcal{S}_{+,\varepsilon}^{M},\text{ }%
\mathbb{\bar{P}}\text{-a.s.}\right\}  \label{eq:def_of_Us} \\
{\mathcal{U}}_{\varepsilon}^{M} & \doteq\left\{ \phi\in\bar{\mathcal{A}}%
:\phi(\cdot,\cdot,\omega)\in\mathcal{S}_{\varepsilon}^{M},\text{ }\mathbb{%
\bar{P}}\text{-a.s.}\right\} .  \notag
\end{align}

The norm in the Hilbert space $L^{2}(\nu_{T})$ will be denoted by $\|
\cdot\|_{2}$ and $B_{2}(r)$ denotes the ball of radius $r$ in $%
L^{2}(\nu_{T}) $. Given a map $\mathcal{G}_{0}:L^{2}(\nu_{T})\rightarrow%
\mathbb{U}$ and $\eta\in\mathbb{U}$, let 
\begin{equation*}
\mathbb{S}_{\eta}\equiv\mathbb{S}_{\eta}[\mathcal{G}_{0}]=\{\psi\in
L^{2}(\nu_{T}):\eta=\mathcal{G}_{0}(\psi)\}
\end{equation*}
and define $I$ by 
\begin{equation}
I(\eta)=\inf_{\psi\in\mathbb{S}_{\eta}}\left[ \frac{1}{2}\|\psi\|_{2}^{2}%
\right] .  \label{eq:first_rate}
\end{equation}
Here we follow the convention that the infimum over an empty set is $+\infty$%
.

We now introduce a sufficient condition that ensures that $I$ is a rate
function and the collection $Y^{\varepsilon}\equiv\mathcal{G}^{\varepsilon
}(\varepsilon N^{\varepsilon^{-1}})$ satisfies a LDP with speed $%
b(\varepsilon )$ and rate function $I$. A set $\{\psi^{\varepsilon}\}\subset%
\bar {\mathcal{A}}$ with the property that $\sup_{\varepsilon>0}\|\psi
^{\varepsilon}\|_{2}\leq M$ a.s. for some $M<\infty$ will be regarded as a
collection of $B_{2}(M)$-valued random variables, where $B_{2}(M)$ is
equipped with the weak topology on the Hilbert space $L^{2}(\nu_{T})$. Since 
$B_{2}(M)$ is weakly compact, such a collection of random variables is
automatically tight. Throughout this paper $B_{2}(M)$ will be regarded as
the compact metric space obtained by equipping it with the weak topology on $%
L^{2}(\nu_{T})$.

Suppose $\varphi \in \mathcal{S}_{+,\varepsilon }^{M}$, which we recall
implies $L_{T}(\varphi )\leq Ma(\varepsilon )^{2}$. Then as shown in Lemma %
\ref{sineqs2} below, there exists $\kappa _{2}(1)\in (0,\infty )$ that is
independent of $\varepsilon $ and such that $\psi 1_{\{|\psi
|<1/a(\varepsilon )\}}\in B_{2}((M\kappa _{2}(1))^{1/2})$, where $\psi
=(\varphi -1)/a(\varepsilon )$.

\begin{condition}
\label{gencond} For some measurable map $\mathcal{G}_{0}:L^{2}(\nu
_{T})\rightarrow\mathbb{U}$, the following two conditions hold.\label{cond1}

\begin{enumerate}
\item[(a)] Given $M\in (0,\infty )$, suppose that $g^{\varepsilon },g\in
B_{2}(M)$ and $g^{\varepsilon }\rightarrow g$. Then 
\begin{equation*}
\mathcal{G}_{0}(g^{\varepsilon })\rightarrow \mathcal{G}_{0}(g).
\end{equation*}

\item[(b)] Given $M\in (0,\infty )$, let $\{\varphi ^{\varepsilon
}\}_{\varepsilon >0}$ be such that for every $\varepsilon >0$, $\varphi
^{\varepsilon }\in \mathcal{U}_{+,\varepsilon }^{M}$ and for some $\beta \in
(0,1]$, $\psi ^{\varepsilon }1_{\{|\psi ^{\varepsilon }|\leq \beta
/a(\varepsilon )\}}\Rightarrow \psi $ in $B_{2}((M\kappa _{2}(1))^{1/2})$
where $\psi ^{\varepsilon }=(\varphi ^{\varepsilon }-1)/a(\varepsilon )$.
Then 
\begin{equation*}
\mathcal{G}^{\varepsilon }(\varepsilon N^{\varepsilon ^{-1}\varphi
^{\varepsilon }})\Rightarrow \mathcal{G}_{0}(\psi ).
\end{equation*}%
\label{cond2}
\end{enumerate}
\end{condition}

\begin{theorem}
\label{LDPpoi} Suppose that the functions $\mathcal{G}^{\varepsilon }$ and $%
\mathcal{G}_{0}$ satisfy Condition \ref{gencond}. Then $I$ defined by (\ref%
{eq:first_rate}) is a rate function and $\{Y^{\varepsilon }\equiv \mathcal{G}%
^{\varepsilon }(\varepsilon N^{\varepsilon ^{-1}})\}$ satisfies a large
deviation principle with speed $b(\varepsilon )$ and rate function $I$.
\end{theorem}

In the next two sections we will present two applications. The first is to a
general family of finite dimensional SDE driven by Poisson noise and the
second is to certain SPDE models with Poisson noise.

\subsection{Finite Dimensional SDEs}

\label{sec_MDP} In this section we study SDEs of the form 
\begin{equation}
X^{\varepsilon}(t)=x_{0}+\int_{0}^{t}b(X^{\varepsilon}(s))ds+\int _{\mathbb{X%
}\times\lbrack0,t]}\varepsilon
G(X^{\varepsilon}(s-),y)N^{\varepsilon^{-1}}(dy,ds),  \label{eq:eq949}
\end{equation}
where the coefficients $b$ and $G$ satisfy the following condition.

\begin{condition}
\label{cndns} The functions $b:\mathbb{R}^{d}\rightarrow\mathbb{R}^{d}$ and $%
G:\mathbb{R}^{d}\times\mathbb{X}\rightarrow\mathbb{R}^{d}$ are measurable
and satisfy

\begin{enumerate}
\item[(a)] for some $L_{b}\in(0,\infty)$ 
\begin{equation*}
|b(x)-b(x^{\prime})|\leq L_{b}|x-x^{\prime}|,\ x,x^{\prime}\in\mathbb{R}^{d},
\end{equation*}

\item[(b)] for some $L_{G}\in L^{1}(\nu)\cap L^{2}(\nu)$ 
\begin{equation*}
|G(x,y)-G(x^{\prime},y)|\leq L_{G}(y)|x-x^{\prime}|,\ x,x^{\prime}\in\mathbb{%
R}^{d},\ y\in\mathbb{X},
\end{equation*}

\item[(c)] for some $M_{G}\in L^{1}(\nu)\cap L^{2}(\nu)$ 
\begin{equation*}
|G(x,y)|\leq M_{G}(y)(1+|x|),\ x\in\mathbb{R}^{d},\ y\in\mathbb{X}.
\end{equation*}
\end{enumerate}
\end{condition}

The following result follows by standard arguments (see Theorem IV.9.1 of 
\cite{Ikeda}).

\begin{theorem}
\label{pathuniq} Fix $x_{0}\in\mathbb{R}^{d}$, and assume Condition \ref%
{cndns}. The following conclusions hold.

\begin{enumerate}
\item[(a)] For each $\varepsilon >0$ there is a measurable map $\bar{%
\mathcal{G}}^{\varepsilon }:\mathbb{M}\rightarrow D([0,T]:\mathbb{R}^{d})$
such that for any probability space $(\tilde{\Omega},\tilde{\mathcal{F}},%
\tilde{P})$ on which is given a Poisson random measure ${\mathbf{n}}%
_{\varepsilon }$ on $\mathbb{X}_{T}$ with intensity measure $\varepsilon
^{-1}\nu _{T}$, $\tilde{X}^{\varepsilon }=\bar{\mathcal{G}}^{\varepsilon
}(\varepsilon {\mathbf{n}}_{\varepsilon })$ is a $\tilde{\mathcal{F}}%
_{t}=\sigma \{{\mathbf{n}}_{\varepsilon }(B\times \lbrack 0,s]),s\leq t,B\in 
\mathcal{B}(\mathbb{X})\}$ adapted process that is the unique solution of
the stochastic integral equation 
\begin{equation}
\tilde{X}^{\varepsilon }(t)=x_{0}+\int_{0}^{t}b(\tilde{X}^{\varepsilon
}(s))ds+\int_{\mathbb{X}\times \lbrack 0,t]}\varepsilon G(\tilde{X}%
^{\varepsilon }(s-),y){\mathbf{n}}_{\varepsilon }(dy,ds),\,t\in \lbrack 0,T].
\label{eq:eq837}
\end{equation}%
In particular $X^{\varepsilon }=\bar{\mathcal{G}}^{\varepsilon }(\varepsilon
N^{\varepsilon ^{-1}})$ is the unique solution of \eqref{eq:eq949}.

\item[(b)] \noindent There is a unique $X^{0}$ in $C([0,T]:\mathbb{R}^{d})$
that solves the equation 
\begin{equation}
X^{0}(t)=x_{0}+\int_{0}^{t}b(X^{0}(s))ds+\int_{\mathbb{X}\times\lbrack
0,t]}G(X^{0}(s),y)\nu(dy)ds,\,t\in\lbrack0,T].  \label{eq:eq840}
\end{equation}
\end{enumerate}
\end{theorem}

We now state a LDP for $\{Y^{\varepsilon}\}$, where 
\begin{equation}
Y^{\varepsilon}\equiv\frac{1}{a(\varepsilon)}(X^{\varepsilon}-X^{0}),
\label{eq:eq113}
\end{equation}
and $a(\varepsilon)$ is as in \eqref{eq:eq1138}. For this we will need the
following additional condition on the coefficients. Let 
\begin{equation}
m_{T}=\sup_{0\leq t\leq T}|X^{0}(t)|.  \label{eq:eq1024}
\end{equation}
For a differentiable function $f:\mathbb{R}^{d}\rightarrow\mathbb{R}^{d}$
let $Df(x)=\left( \partial f_{i}(x)/\partial x_{j}\right) _{i,j}$, $x\in%
\mathbb{R}^{d}$, and let $|Df|_{op}$ denote the operator norm of the matrix $%
Df$. We define a class of functions by%
\begin{equation}
\mathcal{H}\doteq\left\{ h:\mathbb{X}\rightarrow\mathbb{R}:\exists \
\delta>0,\mbox{ s.t.   }\forall\Gamma\mbox{ with }\nu(\Gamma)<\infty
,\int_{\Gamma}\exp(\delta h^{2}(y))\nu(dy)<\infty\right\} .
\label{eq:defofH}
\end{equation}

\begin{condition}
\label{cndnsnew}

\begin{description}
\item 
\begin{enumerate}
\item[(a)] The functions $L_{G}$ and $M_{G}$ are in the class $\mathcal{H}$.

\item[(b)] For every $y\in \mathbb{X}$, the maps $x\mapsto b(x)$ and $%
x\mapsto G(x,y)$ are differentiable. For some $L_{Db}\in (0,\infty )$ 
\begin{equation*}
|Db(x)-Db(x^{\prime })|_{op}\leq L_{Db}|x-x^{\prime }|,\,x,x^{\prime }\in 
\mathbb{R}^{d}
\end{equation*}%
and for some $L_{DG}\in L^{1}(\nu )$ 
\begin{equation*}
|D_{x}G(x,y)-D_{x}G(x^{\prime },y)|_{op}\leq L_{DG}(y)|x-x^{\prime
}|,\,x,x^{\prime }\in \mathbb{R}^{d},\,y\in \mathbb{X}.
\end{equation*}%
With $m_{T}<\infty $ as in \eqref{eq:eq1024} 
\begin{equation*}
\sup_{\{x\in \mathbb{R}^{d}:|x|\leq m_{T}\}}\int_{\mathbb{X}%
}|D_{x}G(x,y)|_{op}\nu (dy)<\infty .
\end{equation*}
\end{enumerate}
\end{description}
\end{condition}

We note that all locally bounded and measurable real functions on $\mathbb{X}
$ are in $\mathcal{H}$.  

The following result gives a Moderate Deviations Principle for finite dimensional SDE.
\begin{theorem}
\label{LDPpoifin} Suppose that Conditions \ref{cndns} and \ref{cndnsnew}
hold. Then $\{Y^{\varepsilon }\}$ satisfies a large deviation principle in $%
D([0,T]:\mathbb{R}^{d})$ with speed $b(\varepsilon )$ and the rate function
given by 
\begin{equation*}
\bar{I}(\eta )=\inf_{\psi }\left\{ \frac{1}{2}\Vert \psi \Vert
_{2}^{2}\right\} ,
\end{equation*}%
where the infimum is taken over all $\psi \in L^{2}(\nu _{T})$ such that 
\begin{align}
\eta (t)& =\int_{0}^{t}[Db(X^{0}(s))]\eta (s)ds+\int_{\mathbb{X}\times
\lbrack 0,t]}[D_{x}G(X^{0}(s),y)]\eta (s)\nu (dy)ds  \notag \\
& \hspace{0.3cm}+\int_{\mathbb{X}\times \lbrack 0,t]}\psi
(y,s)G(X^{0}(s),y)\nu (dy)ds,\;\hspace{1.5cm}t\in \lbrack 0,T].
\label{eq:finite_dim_limit}
\end{align}
\end{theorem}


Note that (\ref{eq:finite_dim_limit}) has a unique solution $\eta\in C([0,T]:%
\mathbb{R}^{d})$. In particular, $\bar{I}(\eta)=\infty$ for all $\eta\in
D([0,T]:\mathbb{R}^{d})\setminus C([0,T]:\mathbb{R}^{d})$.

The following theorem gives an alternative expression for the rate function.
From Condition \ref{cndns}(c) it follows that $y\mapsto G_{i}(X^{0}(s),y)$
is in $L^{2}(\nu )$ for all $s\in \lbrack 0,T]$ and $i=1,\ldots ,d$, where $%
G=(G_{1},\ldots ,G_{d})^{\prime }$. For $i=1,\ldots ,d$, let $e_{i}:\mathbb{X%
}\times \lbrack 0,T]\rightarrow \mathbb{R}$ be measurable functions such
that for each $s\in \lbrack 0,T]$, $\{e_{i}(\cdot ,s)\}_{i=1}^{d}$ is an
orthonormal collection in $L^{2}(\nu )$ and the linear span of the
collection is same as that of $\{G_{i}(X^{0}(s),\cdot )\}_{i=1}^{d}$. Define 
$A:[0,T]\rightarrow \mathbb{R}^{d\times d}$ such that, for each $s\in
\lbrack 0,T]$, 
\begin{equation*}
A_{ij}(s)=\langle G_{i}(X^{0}(s),\cdot ),e_{j}(s,\cdot )\rangle _{L^{2}(\nu
)},\;i,j=1,\ldots ,d.
\end{equation*}

For $\eta \in D([0,T]:\mathbb{R}^{d})$ let%
\begin{equation*}
I(\eta )=\inf_{u}\frac{1}{2}\int_{0}^{T}|u(s)|^{2}ds
\end{equation*}%
where the infimum is taken over all $u\in L^{2}([0,T]:\mathbb{R}^{d})$ such
that 
\begin{equation}
\eta (t)=\int_{0}^{t}[Db(X^{0}(s))+G_{1}(X^{0}(s))]\eta
(s)ds+\int_{0}^{t}A(s)u(s)ds,\;t\in \lbrack 0,T]
\label{eq:finite_dim_limit_u}
\end{equation}%
and $G_{1}(x)=\int_{\mathbb{X}}D_{x}G(x,y)\nu (dy).$

\begin{theorem}
\label{th:thequivratefn} Under the conditions of Theorem \ref{LDPpoifin}, $I
= \bar{I}$.
\end{theorem}

\begin{remark}
\label{rem:rem1150}

\begin{description}
\item {(1)} Theorem \ref{th:thequivratefn} in particular says that the rate
function for $\{Y^{\varepsilon}\}$ is the same as that associated with the
large deviation principle with speed $\varepsilon$ for the Gaussian process 
\begin{equation*}
dZ^{\varepsilon}(t)=A_{1}(t)Z^{\varepsilon}(t)dt+\sqrt{\varepsilon }%
A(t)dW(t),\;Z^{\varepsilon}(0)=x_{0},
\end{equation*}
where $W$ is a standard $d$-dimensional Brownian motion and $%
A_{1}(t)=Db(X^{0}(t))+G_{1}(X^{0}(t))$.

\item {(2)} One can similarly establish moderate deviations results for
systems that have both Poisson and Brownian noise. In particular the
following result holds. Suppose $\sigma :\mathbb{R}^{d}\rightarrow \mathbb{R}%
^{d\times d}$ is a Lipschitz continuous function and $X^{\varepsilon }$
solves the integral equation 
\begin{equation*}
X^{\varepsilon }(t)=x_{0}+\int_{0}^{t}b(X^{\varepsilon }(s))ds+\sqrt{%
\varepsilon }\int_{[0,t]}\sigma (X^{\varepsilon }(s))dW(s)+\int_{\mathbb{X}%
\times \lbrack 0,t]}\varepsilon G(X^{\varepsilon }(s-),y)dN^{\varepsilon
^{-1}}.
\end{equation*}%
Then under Conditions \ref{cndns} and \ref{cndnsnew}, $\{Y^{\varepsilon }\}$
defined as in \eqref{eq:eq113} satisfies a large deviation principle in $%
D([0,T]:\mathbb{R}^{d})$ with speed $b(\varepsilon )$ and the rate function
given by 
\begin{equation*}
\bar{I}(\eta )=\inf_{\psi ,u}\frac{1}{2}\left\{ \Vert \psi \Vert
_{2}^{2}+\int_{[0,T]}|u(s)|^{2}ds\right\} ,
\end{equation*}%
where the infimum is taken over all $(\psi ,u)\in L^{2}(\nu _{T})\times
L^{2}([0,T]:\mathbb{R}^{d})$ such that 
\begin{align*}
\eta (t)& =\int_{0}^{t}[Db(X^{0}(s))]\eta (s)ds+\int_{0}^{t}\sigma
(X^{0}(s))u(s)ds \\
& \hspace{0.3cm}+\int_{\mathbb{X}\times \lbrack
0,t]}[D_{x}G(X^{0}(s),y)]\eta (s)\nu (dy)ds+\int_{\mathbb{X}\times \lbrack
0,t]}\psi (y,s)G(X^{0}(s),y)\nu (dy)ds.
\end{align*}%
Also, the rate function can be simplified as in Theorem \ref%
{th:thequivratefn}.
\end{description}
\end{remark}

\subsection{Infinite Dimensional SDE}

\label{sec:sec2.4} The equation considered here has been studied in \cite%
{KaXi95} where general sufficient conditions for strong existence and
pathwise uniqueness of solutions are identified. The solutions in general
will be distribution valued and a precise formulation of the solution space
is given in terms of Countable Hilbertian Nuclear Spaces (cf. \cite{KaXi95}%
). Recall that a separable Fr\'{e}chet space $\Phi $ is called a \textbf{%
countable Hilbertian nuclear space }(CHNS) if its topology is given by an
increasing sequence $\Vert \cdot \Vert _{n}$, $n\in \mathbb{N}_{0}$, of
compatible Hilbertian norms, and if for each $n\in \mathbb{N}_{0}$ there
exists $m>n$ such that the canonical injection from $\Phi _{m}$ into $\Phi
_{n}$ is Hilbert-Schmidt. Here $\Phi _{k}$, for each $k\in \mathbb{N}_{0}$,
is the completion of $\Phi $ with respect to $\Vert \cdot \Vert _{k}$.

Identify $\Phi_{0}^{\prime}$ with $\Phi_{0}$ using Riesz's representation
theorem and denote the space of bounded linear functionals on $\Phi_{n}$ by $%
\Phi_{-n}$. This space has a natural inner product [and norm] which we
denote by $\langle \cdot,\cdot\rangle_{-n}$ [resp. $\|\cdot\|_{-n}$], $n\in 
\mathbb{N}_{0}$, such that $\{\Phi_{-n}\}_{n\in\mathbb{N}_{0}}$ is a
sequence of increasing Hilbert spaces and the topological dual of $\Phi$,
denoted as $\Phi^{\prime}$ equals $\cup_{n=0}^{\infty}\Phi_{-n}$ (see
Theorem 1.3.1 of \cite{KaXi95}). Solutions of the SPDE considered in this
section will have sample paths in $D([0,T]:\Phi_{-n})$ for some finite value
of $n$.

We will assume that there is a sequence $\{\phi _{j}\}_{j\in \mathbb{N}%
}\subset \Phi $ such that $\{\phi _{j}\}$ is a complete orthonormal system
(CONS) in $\Phi _{0}$ and is a complete orthogonal system (COS) in each $%
\Phi _{n},n\in \mathbb{Z}$. Then $\{\phi _{j}^{n}\}=\{\phi _{j}\Vert \phi
_{j}\Vert _{n}^{-1}\}$ is a CONS in $\Phi _{n}$ for each $n\in \mathbb{Z}$.
It is easily seen that, for each $r>0$, $\eta \in \Phi _{-r}$ and $\phi \in
\Phi _{r}$, $\eta \lbrack \phi ]$ can be expressed as 
\begin{equation}
\eta \lbrack \phi ]=\sum_{j=1}^{\infty }\langle \eta ,\phi _{j}\rangle
_{-r}\langle \phi ,\phi _{j}\rangle _{r}.  \label{eqn:defsqbrkt}
\end{equation}%
We refer the reader to Example 1.3.2 of \cite{KaXi95} for a canonical
example of such a countable Hilbertian nuclear space (CHNS) defined using a
closed densely defined self-adjoint operator on $\Phi _{0}$. A similar
example is considered in Section \ref{sec:secexample}. The SPDE we consider
takes the form 
\begin{equation}
X^{\varepsilon }(t)=x_{0}+\int_{0}^{t}b(X^{\varepsilon }(s))ds+\int_{\mathbb{%
X}\times \lbrack 0,t]}\varepsilon G(X^{\varepsilon }(s-),y)N^{\varepsilon
^{-1}}(ds,dy)  \label{eq:sdeg}
\end{equation}%
where the coefficients $b$ and $G$ satisfy Condition \ref{bGcond_nucl} below
(cf. Chapter 6, \cite{kallianpur1994}). A precise definition of a solution
to (\ref{eq:sdeg}) is as follows.

\begin{definition}
\label{def:sde} Let $(\tilde{\Omega},\tilde{\mathcal{F}},\tilde{P})$ be a
probability space on which is given a Poisson random measure ${\mathbf{n}}%
_{\varepsilon}$ on $\mathbb{X}_{T}$ with intensity measure $%
\varepsilon^{-1}\nu_{T}$. Fix $r\in\mathbb{N}_{0}$ and suppose that $%
x_{0}\in\Phi_{-r}$. A stochastic process $\{X_{t}^{\varepsilon}\}_{t\in%
\lbrack 0,T]}$ defined on $\tilde{\Omega}$ is said to be a $\Phi_{-r}$%
-valued strong solution to the SDE \eqref{eq:sdeg} with $N^{%
\varepsilon^{-1}} $ replaced with ${\mathbf{n}}_{\varepsilon}$ and initial
value $x_{0}$, if the following hold.

\begin{enumerate}
\item[(a)] $X_{t}^{\varepsilon }$ is a $\Phi _{-r}$-valued $\tilde{\mathcal{F%
}}_{t}$-measurable random variable for all $t\in \lbrack 0,T]$, where $%
\tilde{\mathcal{F}}_{t}=\sigma \{{\mathbf{n}}_{\varepsilon }(B\times \lbrack
0,s]),s\leq t,B\in \mathcal{B}(\mathbb{X})\}$.

\item[(b)] $X^{\varepsilon }\in D([0,T]:\Phi _{-r})$ a.s.

\item[(c)] The map $(s,\omega )\mapsto b(X_{s}^{\varepsilon }(\omega ))$ is
measurable from $[0,T]\times \Omega $ to $\Phi _{-r}$ and the map $(s,\omega
,y)\mapsto G(s,X_{s-}^{\varepsilon }(\omega ),y)$ is $(\tilde{\mathcal{P}}%
\times \mathcal{B}(\mathbb{X}))/\mathcal{B}(\Phi _{-r})$ measurable, where $%
\tilde{\mathcal{P}}$ is the predictable $\sigma $-field corresponding to the
filtration $\{\tilde{\mathcal{F}}_{t}\}$. Furthermore, 
\begin{equation*}
\tilde{E}\int_{0}^{T}\int_{\mathbb{X}}\Vert G(s,X_{s}^{\varepsilon },v)\Vert
_{-r}^{2}\nu (dv)ds<\infty
\end{equation*}%
and 
\begin{equation*}
\tilde{E}\int_{0}^{T}\Vert b(X_{s}^{\varepsilon })\Vert _{-r}^{2}ds<\infty .
\end{equation*}

\item[(d)] For all $t\in \lbrack 0,T]$, almost all $\omega \in \tilde{\Omega}
$, and all $\phi \in \Phi $ 
\begin{equation}
X_{t}^{\varepsilon }[\phi ]=x_{0}[\phi ]+\int_{0}^{t}b(X_{s}^{\varepsilon
})[\phi ]ds+\varepsilon \int_{\mathbb{X}\times \lbrack
0,t]}G(s,X_{s-}^{\varepsilon },y)[\phi ]{\mathbf{n}}_{\varepsilon }(dy,ds).
\label{eq:sdeg2}
\end{equation}
\end{enumerate}
\end{definition}

We now present a condition from \cite{KaXi95} that ensures unique
solvability of \eqref{eq:sdeg}. Let $\theta_{p}:\Phi_{-
p}\rightarrow\Phi_{p} $ be the isometry such that for all $j\in\mathbb{N}$, $%
\theta_{p}(\phi_{j}^{-p})=\phi_{j}^{p}$. It is easy to check that for all $%
p\in\mathbb{N}$, $\theta_{p}(\Phi)\subseteq\Phi$ (see Remark 6.1.1 of \cite%
{KaXi95}).

\begin{condition}
\label{bGcond_nucl} For some $p,q\in\mathbb{N}$ with $q>p$ for which the
embedding of $\Phi_{-p}$ to $\Phi_{-q}$ is Hilbert -Schmidt, the following
hold.

\begin{enumerate}
\item[(a)] \label{bGcont_nucl} (Continuity) $b:\Phi ^{\prime }\rightarrow
\Phi ^{\prime }$ is such that it is a continuous function from $\Phi _{-p}$
to $\Phi _{-q}$. $G$ is a map from $\Phi ^{\prime }\times \mathbb{X}$ to $%
\Phi ^{\prime }$ such that for each $u\in \Phi _{-p}$, $G(u,\cdot )\in L^{2}(%
\mathbb{X},\nu ,\Phi _{-p})$ and the mapping $\Phi _{-p}\ni u\mapsto
G(u,\cdot )\in L^{2}(\mathbb{X},\nu ,\Phi _{-p})$ is continuous.

\item[(b)] \label{bGgrowth_nucl} There exist $M_{b}\in (0,\infty )$ and $%
M_{G}\in L^{1}(\nu )\cap L^{2}(\nu )$ such that 
\begin{equation*}
\Vert b(u)\Vert _{-q}\leq M_{b}(1+\Vert u\Vert _{-p}),\quad \Vert
G(u,y)\Vert _{-p}\leq M_{G}(y)(1+\Vert u\Vert _{-p}),\quad u\in \Phi
_{-p},y\in \mathbb{X}.
\end{equation*}

\item[(c)] For some $C_{b}\in (0,\infty )$ and all $\phi \in \Phi $ 
\begin{equation*}
2b(\phi )[\theta _{p}\phi ]\leq C_{b}(1+\Vert \phi \Vert _{-p}^{2}).
\end{equation*}

\item[(d)] \label{bLip_nucl} For some $L_{b}\in (0,\infty )$ 
\begin{equation*}
\left\langle u-u^{\prime },b(u)-b(u^{\prime })\right\rangle _{-q}\leq
L_{b}\Vert u-u^{\prime }\Vert _{-q}^{2},\ u,u^{\prime }\in \Phi _{-p}.
\end{equation*}

\item[(e)] For some $L_{G}\in L^{1}(\nu )\cap L^{2}(\nu )$ 
\begin{equation*}
\Vert G(u,y)-G(u^{\prime },y)\Vert _{-q}\leq L_{G}(y)\Vert u-u^{\prime
}\Vert _{-q},\ u,u^{\prime }\in \Phi _{-p},\ y\in \mathbb{X}.
\end{equation*}
\end{enumerate}
\end{condition}

The following unique solvability result follows from \cite{KaXi95}. For part
(b) see Theorem 3.7 in \cite{BCD}.

\begin{theorem}
\label{Thm: strsol} Fix $x_{0}\in\Phi_{-p}$, and assume Condition \ref%
{bGcond_nucl}. The following conclusions hold.

\begin{enumerate}
\item[(a)] For each $\varepsilon >0$, there is a measurable map $\bar{%
\mathcal{G}}^{\varepsilon }:\mathbb{M}\rightarrow D([0,T]:\Phi _{-q})$ such
that for any probability space $(\tilde{\Omega},\tilde{\mathcal{F}},\tilde{P}%
)$ and a Poisson random measure ${\mathbf{n}}_{\varepsilon }$ as in
Definition \ref{def:sde}, $\tilde{X}^{\varepsilon }=\bar{\mathcal{G}}%
^{\varepsilon }(\varepsilon {\mathbf{n}}_{\varepsilon })$ is the unique $%
\Phi _{-q}$-valued strong solution of \eqref{eq:sdeg} with $N^{\varepsilon
^{-1}}$ replaced with ${\mathbf{n}}_{\varepsilon }$. Furthermore, for every $%
t\in \lbrack 0,T]$, $\tilde{X}_{t}^{\varepsilon }\in \Phi _{-p}$ and $\tilde{%
E}\sup_{0\leq t\leq T}\Vert \tilde{X}_{t}^{\varepsilon }\Vert
_{-p}^{2}<\infty $, In particular $X^{\varepsilon }=\bar{\mathcal{G}}%
^{\varepsilon }(\varepsilon N^{\varepsilon ^{-1}})$ satisfies, for every $%
\phi \in \Phi $, 
\begin{equation}
X_{t}^{\varepsilon }[\phi ]=X_{0}[\phi ]+\int_{0}^{t}b(X_{s}^{\varepsilon
})[\phi ]ds+\varepsilon \int_{\lbrack 0,t]\times \mathbb{X}%
}G(X_{s-}^{\varepsilon },y)[\phi ]N^{\varepsilon ^{-1}}(dy,ds).
\label{eq:sdeg2b}
\end{equation}

\item[(b)] The integral equation 
\begin{equation}
X^{0}(t)=x_{0}+\int_{0}^{t}b(X^{0}(s))ds+\int_{E\times \lbrack
0,t)}G(X^{0}(s),y)\nu (dy)ds.  \label{eq:eq846}
\end{equation}%
has a unique $\Phi _{-q}$-valued continuous solution. That is, there is a
unique $X^{0}\in C([0,T],\Phi _{-q})$ such that for all $t\in \lbrack 0,T]$
and all $\phi \in \Phi $ 
\begin{equation}
X_{t}^{0}[\phi ]=X_{0}[\phi ]+\int_{0}^{t}b(X_{s}^{0})[\phi
]ds+\int_{[0,t]\times \mathbb{X}}G(X_{s}^{0},y)[\phi ]\nu (dy)ds.
\label{eq:sdeg4}
\end{equation}%
Furthermore $X_{t}^{0}\in \Phi _{-p}$ for all $t\in \lbrack 0,T]$ and 
\begin{equation}
m_{T}=\sup_{0\leq t\leq T}\Vert X_{t}^{0}\Vert _{-p}<\infty .
\label{eq:eq336}
\end{equation}
\end{enumerate}
\end{theorem}

As before, we are interested in a LDP for $\{Y^{\varepsilon}\}$, where 
\begin{equation*}
Y^{\varepsilon}\equiv\frac{1}{a(\varepsilon)}(X^{\varepsilon}-X^{0}),
\end{equation*}
and $a(\varepsilon)$ is as in \eqref{eq:eq1138}. For that we will need some
additional conditions on the coefficients. Recall the definition of Fr\'{e}%
chet derivative of a real valued function on a Hilbert space (see Chapter
II.5 of \cite{DuSc}), which characterizes the derivative as a bounded linear
functional on the Hilbert space. For the remainder of this section we
considered a fixed $p$ and $q$ that satisfy Condition \ref{bGcond_nucl}.

\begin{condition}
\label{bGcond_nuclNew} There exists a positive integer $q_{1}>q$ such that
the canonical mapping of $\Phi _{-q}$ to $\Phi _{-q_{1}}$ is
Hilbert-Schmidt, and the following hold.

\begin{enumerate}
\item[(a)] For every $\phi \in \Phi $, the Fr\'{e}chet derivative of the map 
$\Phi _{-q}\ni v\mapsto b(v)[\phi ]$ from $\Phi _{-q}$ to $\mathbb{R}$
exists and is denoted by $D(b(\cdot )[\phi ])$. For each $\phi \in \Phi $,
there exists $L_{Db}(\phi )\in (0,\infty )$ such that 
\begin{equation*}
\Vert D(b(u)[\phi ])-D(b(u^{\prime })[\phi ])\Vert _{op}\leq L_{Db}(\phi
)\Vert u-u^{\prime }\Vert _{-q},\ u,u^{\prime }\in \Phi _{-p},
\end{equation*}%
Here $\Vert \cdot \Vert _{op}$ is the operator norm in $L(\Phi _{-q},\mathbb{%
R})$.

\item[(b)] Recall that $\phi _{k}^{q_{1}}\doteq \phi _{k}\Vert \phi
_{k}\Vert _{q_{1}}^{-1}$. Then for every $\eta \in \Phi _{-q}$ 
\begin{equation*}
\sup_{\{v\in \Phi _{-p}:\Vert v\Vert _{-p}\leq m_{T}\}}\sum_{k=1}^{\infty
}\left\vert D\left( b(v)[\phi _{k}^{q_{1}}]\right) [\eta ]\right\vert
^{2}\equiv M_{2}(\eta )<\infty .
\end{equation*}%
This means that $A_{v}(\eta ):\Phi \rightarrow \mathbb{R}$ defined by $%
A_{v}(\eta )[\phi ]=D\left( b(v)[\phi ]\right) [\eta ]$ extends to a bounded
linear map from $\Phi _{q_{1}}$ to $\mathbb{R}$ (i.e. an element of $\Phi
_{-q_{1}}$). For all $v\in \Phi _{-p}$ such that $\Vert v\Vert _{-p}\leq
m_{T}$, $\eta \mapsto A_{v}(\eta )$ is a continuous map from $\Phi _{-q}$ to 
$\Phi _{-q_{1}}$ and there exist $M_{A},L_{A},C_{A}\in (0,\infty )$ such
that 
\begin{equation}
\sup_{\{v\in \Phi _{-p}:\Vert v\Vert _{-p}\leq m_{T}\}}\Vert A_{v}(\eta
)\Vert _{-q_{1}}\leq M_{A}(1+\Vert \eta \Vert _{-q}),\mbox{ for all }\eta
\in \Phi _{-q}.  \label{eq: eq429a}
\end{equation}%
\begin{equation}
\sup_{\{v\in \Phi _{-p}:\Vert v\Vert _{-p}\leq m_{T}\}}\langle \eta ,A_{v}(\eta )\rangle _{-q_{1}}\leq
L_{A}\Vert \eta\Vert _{-q_{1}}^{2},\mbox{ for all }\eta
\in \Phi _{-q}.  \label{eq: eq429b}
\end{equation}%
\begin{equation}
\sup_{\{v\in \Phi _{-p}:\Vert v\Vert _{-p}\leq m_{T}\}}2A_{v}(\phi +\zeta
)[\theta _{q}\phi ]\leq C_{A}(\Vert \zeta \Vert _{-p}+\Vert \phi \Vert
_{-q})\Vert \phi \Vert _{-q},\mbox{ for
all }\phi \in \Phi ,\zeta \in \Phi _{-p},  \label{eq: eq429c}
\end{equation}%
where $\theta _{q}$ was defined just before Condition \ref{bGcond_nucl}.

\item[(c)] For every $\phi \in \Phi _{q_{1}},y\in \mathbb{X}$, the Fr\'{e}%
chet derivative of $G(\cdot ,y)[\phi ]:\Phi _{-q_{1}}\rightarrow \mathbb{R}$%
, denoted as $D_{x}(G(\cdot ,y)[\phi ])$, exists. The map $\Phi _{-p}\ni
u\rightarrow D_{x}(G(u,y)[\phi ])\in L(\Phi _{-q_{1}},\mathbb{R})$ is
Lipschitz continuous: for each $\phi \in \Phi _{q_{1}}$ there exists $%
L_{DG}(\phi ,\cdot )\in L^{1}(\nu )$ such that 
\begin{equation*}
\Vert D_{x}G(u,y)[\phi ]-D_{x}G(u^{\prime },y)[\phi ]\Vert _{op,-q_{1}}\leq
L_{DG}(\phi ,y)\Vert u-u^{\prime }\Vert _{-q},\ u,u^{\prime }\in \Phi
_{-p},y\in \mathbb{X}.
\end{equation*}%
There exists $M_{DG}:\Phi _{-p}\times \mathbb{X}\rightarrow \mathbb{R}_{+}$
such that 
\begin{equation}
\Vert D_{x}(G(u,y)[\phi ])\Vert _{op,-q_{1}}\leq M_{DG}(u,y)\Vert \phi \Vert
_{q_{1}},\;u\in \Phi _{-p},\phi \in \Phi _{q_{1}},y\in \mathbb{X}.
\label{eq:225a}
\end{equation}%
\begin{equation*}
\Vert D_{x}(G(u,y)[\phi ])\Vert _{op,-q}\leq M_{DG}(u,y)\Vert \phi \Vert
_{q},\;u\in \Phi _{-p},\phi \in \Phi _{q},y\in \mathbb{X},
\end{equation*}%
and 
\begin{equation*}
M_{DG}^{\ast }\doteq \sup_{\{u\in \Phi _{-p}:\Vert u\Vert _{-p}\leq
m_{T}\}}\int_{\mathbb{X}}\max \{M_{DG}(u,y),M_{DG}^{2}(u,y)\}\nu (dy)<\infty
.
\end{equation*}%
Here $\Vert \cdot \Vert _{op,-q_{1}}$ (resp. $\Vert \cdot \Vert _{op,-q}$)
is the operator norm in $L(\Phi _{-q_{1}},\mathbb{R})$ (resp. $L(\Phi _{-q},%
\mathbb{R})$).

\item[(d)] The functions $M_{G}$ and $L_{G}$ in Condition \ref{bGcond_nucl}
are in $\mathcal{H}$ defined by (\ref{eq:defofH}).
\end{enumerate}
\end{condition}

Theorem \ref{etaeqexist} shows that under Conditions \ref{bGcond_nucl} and %
\ref{bGcond_nuclNew}, for every $\psi\in L^{2}(\nu_{T})$ there is a unique $%
\eta\in C([0,T],\Phi_{-q_{1}})$ that solves 
\begin{align}
\eta(t) & =\int_{0}^{t}Db(X^{0}(s))\eta(s)ds+\int_{\mathbb{X}\times
\lbrack0,t]}D_{x}G(X^{0}(s),y)\eta(s)\nu(dy)ds  \notag \\
& \hspace{0.3cm}+\int_{\mathbb{X}\times\lbrack0,t]}G(X^{0}(s),y)\psi
(y,s)\nu(dy)ds,  \label{nueq_nucl}
\end{align}
in the sense that for every $\phi\in\Phi$ 
\begin{align}
\eta(t)[\phi] & =\int_{0}^{t}D(b(X^{0}(s))[\phi])\eta(s)ds+\int _{\mathbb{X}%
\times\lbrack0,t]}D_{x}(G(X^{0}(s),y)[\phi])\eta(s)\nu (dy)ds  \notag \\
& \hspace{0.3cm}+\int_{\mathbb{X}\times\lbrack0,t]}G(X^{0}(s),y)[\phi
]\psi(y,s)\nu(dy)ds.  \label{limpt_poi_inf}
\end{align}
The following is the main result of this section.

\begin{theorem}
\label{LDPpoi_inf} Suppose Conditions \ref{bGcond_nucl} and \ref%
{bGcond_nuclNew} hold. Then $\{Y^{\varepsilon}\}_{\varepsilon>0}$ satisfies
a large deviation principle in $D([0,T],\Phi_{-q_{1}})$ with speed $%
b(\varepsilon)$ and rate function $I$ given by 
\begin{equation*}
I(\eta)=\inf_{\psi}\left\{ \frac{1}{2}\Vert\psi\Vert_{2}^{2}\right\} ,
\end{equation*}
where the infimum is taken over all $\psi\in L^{2}(\nu_{T})$ such that $%
(\eta,\psi)$ satisfy $\eqref{nueq_nucl}$.
\end{theorem}

In Section \ref{sec:secexample} we will provide an example taken from \cite%
{KaXi95} where Conditions \ref{bGcond_nucl} and \ref{bGcond_nuclNew} hold.

\section{Proof of Theorem \protect\ref{LDPpoi}}

\label{sec:proofthm1}

The following inequalities will be used several times. Recall the function $%
\ell (r)=r\log r-r+1$.

\begin{lemma}
\label{sineqs}

\begin{enumerate}
\item[(a)] For $a,b\in (0,\infty )$ and $\sigma \in \lbrack 1,\infty )$, $%
ab\leq e^{\sigma a}+\frac{1}{\sigma }\ell (b).$

\item[(b)] For every $\beta >0$, there exist $\kappa _{1}(\beta ),\kappa
_{1}^{\prime }(\beta )\in (0,\infty )$ such that $\kappa _{1}(\beta ),\kappa
_{1}^{\prime }(\beta )\rightarrow 0$ as $\beta \rightarrow \infty $, and%
\begin{equation*}
|x-1|\leq \kappa _{1}(\beta )\ell (x)\mbox{ for }|x-1|\geq \beta ,x\geq 0,%
\mbox{
and }\;x\leq \kappa _{1}^{\prime }(\beta )\ell (x)\mbox{ for 
}x\geq \beta .
\end{equation*}

\item[(c)] For each $\beta >0$, there exists $\kappa _{2}(\beta )\in
(0,\infty )$ such that 
\begin{equation*}
|x-1|^{2}\leq \kappa _{2}(\beta )\ell (x)\quad \mbox{ for }|x-1|\leq \beta
,x\geq 0.
\end{equation*}

\item[(d)] There exists $\kappa _{3}\in (0,\infty )$ such that 
\begin{equation*}
\ell (x)\leq \kappa _{3}|x-1|^{2},\;\;|\ell (x)-(x-1)^{2}/2|\leq \kappa
_{3}|x-1|^{3}\;\mbox{ for all }x\geq 0.
\end{equation*}
\end{enumerate}
\end{lemma}

Note that we can assume without loss that $\kappa_{2}(\beta)$ is
nonincreasing in $\beta$. The following result is immediate from Lemma \ref%
{sineqs}.

\begin{lemma}
\label{sineqs2} Suppose $\varphi \in \mathcal{S}_{+,\varepsilon }^{M}$ for
some $M<\infty $, where $\mathcal{S}_{+,\varepsilon }^{M}$ is defined in (%
\ref{ctrlclass}). Let $\psi =\frac{\varphi -1}{a(\varepsilon )}$. Then for
all $\beta >0$

\begin{enumerate}
\item[(a)] $\displaystyle{\int_{\mathbb{X}\times \lbrack 0,T]}|\psi
|1_{\{|\psi |\geq \beta /a(\varepsilon )\}}d\nu _{T}\leq Ma(\varepsilon
)\kappa _{1}(\beta )}$

\item[(b)] $\displaystyle{\int_{\mathbb{X}\times \lbrack 0,T]}\varphi
1_{\{\varphi \geq \beta \}}d\nu _{T}\leq Ma^{2}(\varepsilon )\kappa
_{1}^{\prime }(\beta )}$

\item[(c)] $\displaystyle{\int_{\mathbb{X}\times \lbrack 0,T]}|\psi
|^{2}1_{\{|\psi |\leq \beta /a(\varepsilon )\}}d\nu _{T}\leq M\kappa
_{2}(\beta ),}$
\end{enumerate}

where $\kappa_{1}, \kappa^{\prime}_{1}$ and $\kappa_{2}$ are as in Lemma \ref%
{sineqs}.
\end{lemma}

The property that $I$ defined in \eqref{eq:first_rate} is a rate function is
immediate on observing that Condition \ref{gencond}(a) says that $\Gamma
_{K}=\{\mathcal{G}_{0}(g):g\in B_{2}(K)\}$ is compact for all $K<\infty $,
and therefore for every $M<\infty $, $\{\eta \in \mathbb{U}:I(\eta )\leq
M\}=\cap _{n\geq 1}\Gamma _{2M+1/n}$ is compact as well.

To prove Theorem \ref{LDPpoi} it suffices to show that the Laplace principle
lower and upper bounds hold for all $F\in C_{b}(\mathbb{U})$. Let ${\mathcal{%
G}}^{\varepsilon }$ be as in the statement of Theorem \ref{LDPpoi}. Then it
follows from Theorem \ref{VR: PRM} with $\theta =\varepsilon ^{-1}$ and $%
F(\cdot )$ there replaced by $F\circ \mathcal{G}^{\varepsilon }(\varepsilon
\cdot )/b(\varepsilon )$ that 
\begin{equation}
-b(\varepsilon )\log \mathbb{\bar{E}}[e^{-F(Y^{\varepsilon })/b(\varepsilon
)}]=\inf_{\varphi \in \bar{\mathcal{A}}_{+}}\mathbb{\bar{E}}\left[
b(\varepsilon )\varepsilon ^{-1}L_{T}(\varphi )+F\circ \mathcal{G}%
^{\varepsilon }(\varepsilon N^{\varepsilon ^{-1}\varphi })\right] .
\label{eq:finite_dim_rep}
\end{equation}%
We first prove the lower bound 
\begin{equation}
\liminf_{\varepsilon \rightarrow 0}-b(\varepsilon )\log \mathbb{\bar{E}}%
[e^{-F(Y^{\varepsilon })/b(\varepsilon )}]\geq \inf_{\eta \in \mathbb{U}}%
\left[ I(\eta )+F(\eta )\right] .  \label{eq:finite_dim_LB}
\end{equation}%
For $\varepsilon \in (0,1)$, choose $\tilde{\varphi}^{\varepsilon }\in \bar{%
\mathcal{A}}_{b}$ such that 
\begin{equation}
-b(\varepsilon )\log \mathbb{\bar{E}}[e^{-F(Y^{\varepsilon })/b(\varepsilon
)}]\geq \mathbb{\bar{E}}\left[ b(\varepsilon )\varepsilon ^{-1}L_{T}(\tilde{%
\varphi}^{\varepsilon })+F\circ \mathcal{G}^{\varepsilon }(\varepsilon
N^{\varepsilon ^{-1}\tilde{\varphi}^{\varepsilon }})\right] -\varepsilon .
\label{eq:finite_dim_nearly_opt}
\end{equation}%
Since $\Vert F\Vert _{\infty }\equiv \sup_{x\in \mathbb{U}}|F(x)|<\infty $,
we have for all $\varepsilon \in (0,1)$ that 
\begin{equation}
\tilde{M}\doteq (2\Vert F\Vert _{\infty }+1)\geq \mathbb{\bar{E}}\left[ 
\frac{b(\varepsilon )}{\varepsilon }L_{T}(\tilde{\varphi}^{\varepsilon })%
\right] .  \label{eq:bound_on_cost}
\end{equation}%
Fix $\delta >0$ and define 
\begin{equation*}
\tau ^{\varepsilon }=\inf \{t\in \lbrack 0,T]:b(\varepsilon )\varepsilon
^{-1}L_{t}(\tilde{\varphi}^{\varepsilon })>2\tilde{M}\Vert F\Vert _{\infty
}/\delta \}\wedge T.
\end{equation*}%
Let 
\begin{equation*}
\varphi ^{\varepsilon }(y,s)=\tilde{\varphi}^{\varepsilon }(y,s)1_{\{s\leq
\tau ^{\varepsilon }\}}+1_{\{s>\tau ^{\varepsilon }\}},\;(y,s)\in \mathbb{X}%
\times \lbrack 0,T].
\end{equation*}%
Observe that $\varphi ^{\varepsilon }\in \bar{\mathcal{A}}_{b}$ and $%
b(\varepsilon )\varepsilon ^{-1}L_{T}(\varphi ^{\varepsilon })\leq M\doteq 2%
\tilde{M}\Vert F\Vert _{\infty }/\delta $. Also, 
\begin{align}
\mathbb{\bar{P}}\left\{ \varphi ^{\varepsilon }\neq \tilde{\varphi}%
^{\varepsilon }\right\} & \leq \mathbb{\bar{P}}\left\{ b(\varepsilon
)\varepsilon ^{-1}L_{T}(\tilde{\varphi}^{\varepsilon })>M\right\}  \notag \\
& \leq \mathbb{\bar{E}}[b(\varepsilon )\varepsilon ^{-1}L_{T}(\tilde{\varphi}%
^{\varepsilon })]/M  \notag \\
& \leq \frac{\delta }{2\Vert F\Vert _{\infty }}  \label{eq:eq332}
\end{align}%
where the last inequality holds by (\ref{eq:bound_on_cost}). For $(y,s)\in 
\mathbb{X}\times \lbrack 0,T]$ define%
\begin{equation*}
\tilde{\psi}^{\varepsilon }(y,s)\equiv \frac{\tilde{\varphi}^{\varepsilon
}(y,s)-1}{a(\varepsilon )},\quad \psi ^{\varepsilon }(y,s)\equiv \frac{%
\varphi ^{\varepsilon }(y,s)-1}{a(\varepsilon )}=\tilde{\psi}^{\varepsilon
}(y,s)1_{\{s\leq \tau ^{\varepsilon }\}}.
\end{equation*}%
Fix $\beta \in (0,1]$ and let $B_{\varepsilon }=\beta /a(\varepsilon )$.
Applying Lemma \ref{sineqs}(d), Lemma \ref{sineqs2}(c), using $\kappa
_{2}(1)\geq \kappa _{2}(\beta )$ and (\ref{eq:finite_dim_nearly_opt}), we
have that 
\begin{align}
-b(\varepsilon )\log \mathbb{\bar{E}}[e^{-F(Y^{\varepsilon })/b(\varepsilon
)}]& \geq \mathbb{\bar{E}}\left[ \frac{b(\varepsilon )}{\varepsilon }\int_{%
\mathbb{X}\times \lbrack 0,T]}\ell (\tilde{\varphi}^{\varepsilon })d\nu
_{T}+F\circ \mathcal{G}^{\varepsilon }(\varepsilon N^{\varepsilon ^{-1}%
\tilde{\varphi}^{\varepsilon }})\right] -\varepsilon  \notag \\
& \geq \mathbb{\bar{E}}\left[ \frac{b(\varepsilon )}{\varepsilon }\int_{%
\mathbb{X}\times \lbrack 0,T]}\ell ({\varphi }^{\varepsilon })1_{\{|{\psi }%
^{\varepsilon }|\leq B_{\varepsilon }\}}d\nu _{T}+F\circ \mathcal{G}%
^{\varepsilon }(\varepsilon N^{\varepsilon ^{-1}\tilde{\varphi}^{\varepsilon
}})\right] -\varepsilon  \notag \\
& \geq \mathbb{\bar{E}}\left[ \frac{1}{2}\int_{\mathbb{X}\times \lbrack
0,T]}\left( ({\psi }^{\varepsilon })^{2}-\kappa _{3}a(\varepsilon )|{\psi }%
^{\varepsilon }|^{3}\right) 1_{\{|{\psi }^{\varepsilon }|\leq B_{\varepsilon
}\}}d\nu _{T}+F\circ \mathcal{G}^{\varepsilon }(\varepsilon N^{\varepsilon
^{-1}\varphi ^{\varepsilon }})\right]  \notag \\
& \hspace{0.4cm}+\mathbb{\bar{E}}\left( F\circ \mathcal{G}^{\varepsilon
}(\varepsilon N^{\varepsilon ^{-1}\tilde{\varphi}^{\varepsilon }})-F\circ 
\mathcal{G}^{\varepsilon }(\varepsilon N^{\varepsilon ^{-1}\varphi
^{\varepsilon }})\right) -\varepsilon  \notag \\
& \geq \mathbb{\bar{E}}\left[ \frac{1}{2}\int_{\mathbb{X}\times \lbrack
0,T]}(\psi ^{\varepsilon })^{2}1_{\{|\psi ^{\varepsilon }|\leq
B_{\varepsilon }\}}d\nu _{T}+F\circ \mathcal{G}^{\varepsilon }(\varepsilon
N^{\varepsilon ^{-1}\varphi ^{\varepsilon }})\right]  \notag \\
& \hspace{0.6cm}-\delta -\varepsilon -\frac{1}{2}\beta \kappa _{3}M\kappa
_{2}(1),  \label{eq:nonasym_LB}
\end{align}%
where the last inequality follows from \eqref{eq:eq332} on noting that 
\begin{equation*}
\left\vert \mathbb{\bar{E}}\left( F\circ \mathcal{G}^{\varepsilon
}(\varepsilon N^{\varepsilon ^{-1}\tilde{\varphi}^{\varepsilon }})-F\circ 
\mathcal{G}^{\varepsilon }(\varepsilon N^{\varepsilon ^{-1}\varphi
^{\varepsilon }})\right) \right\vert \leq 2\Vert F\Vert _{\infty }\mathbb{%
\bar{P}}\left\{ \varphi ^{\varepsilon }\neq \tilde{\varphi}^{\varepsilon
}\right\} \leq \delta .
\end{equation*}%
By weak compactness of $B_{2}(r)$ and again using the monotonicity of $%
\kappa _{2}(\beta )$, $\{\psi ^{\varepsilon }1_{\{|\psi ^{\varepsilon }|\leq
\beta /a(\varepsilon )\}}\}$ is a tight family of $B_{2}((M\kappa
_{2}(1))^{1/2})$-valued random variables. Let $\psi $ be a limit point along
a subsequence which we index once more by $\varepsilon $. By a standard
argument by contradiction it suffices to prove (\ref{eq:finite_dim_LB})
along this subsequence. From Condition \ref{cond1}(b), along this
subsequence $\mathcal{G}^{\varepsilon }(\varepsilon N^{\varepsilon
^{-1}\varphi ^{\varepsilon }})$ converges in distribution to $\eta =\mathcal{%
G}_{0}(\psi )$. Hence taking limits in (\ref{eq:nonasym_LB}) along this
subsequence, we have 
\begin{align*}
\liminf_{\varepsilon \rightarrow 0}-b(\varepsilon )\log \mathbb{\bar{E}}%
[e^{-F(Y^{\varepsilon })/b(\varepsilon )}]& \geq \mathbb{\bar{E}}\left[ 
\frac{1}{2}\int_{\mathbb{X}\times \lbrack 0,T]}\psi ^{2}d\nu _{T}+F(\eta )%
\right] -\delta -\frac{\beta }{2}\kappa _{3}M\kappa _{2}(1) \\
& \geq \mathbb{\bar{E}}\left[ I(\eta )+F(\eta )\right] -\delta -\frac{1}{2}%
\beta \kappa _{3}M\kappa _{2}(1) \\
& \geq \inf_{\eta \in \mathbb{U}}\left[ I(\eta )+F(\eta )\right] -\delta -%
\frac{1}{2}\beta \kappa _{3}M\kappa _{2}(1).
\end{align*}%
where the first line is from Fatou's lemma and the second uses the
definition of $I$ in (\ref{eq:first_rate}). Sending $\delta $ and $\beta $
to $0$ we get (\ref{eq:finite_dim_LB}).

To complete the proof we now show the upper bound 
\begin{equation}
\limsup_{\varepsilon \rightarrow 0}-b(\varepsilon )\log \mathbb{\bar{E}}%
[e^{-F(Y^{\varepsilon })/b(\varepsilon )}]\leq \inf_{\eta \in \mathbb{U}}%
\left[ I(\eta )+F(\eta )\right] .  \label{eq:finite_dim_UB}
\end{equation}%
Fix $\delta >0$. Then there exists $\eta \in \mathbb{U}$ such that 
\begin{equation}
I(\eta )+F(\eta )\leq \inf_{\eta \in \mathbb{U}}[I(\eta )+F(\eta )]+\delta
/2.  \label{eq:finite_dim_delta_opt}
\end{equation}%
Choose $\psi \in L^{2}(\nu _{T})$ such that 
\begin{equation}
\frac{1}{2}\int_{\mathbb{X}\times \lbrack 0,T]}|\psi |^{2}d\nu _{T}\leq
I(\eta )+\delta /2,  \label{eq:finite_dim_delta_opt2}
\end{equation}%
where $\eta =\mathcal{G}_{0}(\psi )$. For $\beta \in (0,1]$ define%
\begin{equation*}
\psi ^{\varepsilon }=\psi 1_{\{|\psi |\leq \frac{\beta }{a(\varepsilon )}%
\}},\quad \varphi ^{\varepsilon }=1+a(\varepsilon )\psi ^{\varepsilon }.
\end{equation*}%
From Lemma \ref{sineqs}(d), for every $\varepsilon >0$, 
\begin{align*}
\int_{\mathbb{X}\times \lbrack 0,T]}\ell (\varphi ^{\varepsilon })d\nu _{T}&
\leq \kappa _{3}\int_{\mathbb{X}\times \lbrack 0,T]}(\varphi ^{\varepsilon
}-1)^{2}d\nu _{T} \\
& =a^{2}(\varepsilon )\kappa _{3}\int_{\mathbb{X}\times \lbrack 0,T]}|\psi
^{\varepsilon }|^{2}d\nu _{T} \\
& \leq a^{2}(\varepsilon )M,
\end{align*}%
where $M=\kappa _{3}\int_{\mathbb{X}\times \lbrack 0,T]}|\psi |^{2}d\nu _{T}$%
. Thus $\varphi ^{\varepsilon }\in \mathcal{U}_{+,\varepsilon }^{M}$ for all 
$\varepsilon >0$. Also 
\begin{equation*}
\psi ^{\varepsilon }1_{\{|\psi ^{\varepsilon }|\leq \frac{\beta }{%
a(\varepsilon )}\}}=\psi 1_{\{|\psi |\leq \frac{\beta }{a(\varepsilon )}\}}
\end{equation*}%
which converges to $\psi $ as $\varepsilon \rightarrow 0$. Thus from
Condition \ref{cond1}(b) 
\begin{equation}
\mathcal{G}^{\varepsilon }(\varepsilon N^{\varepsilon ^{-1}\varphi
^{\varepsilon }})\Rightarrow \mathcal{G}_{0}(\psi ).  \label{eq:eq847}
\end{equation}%
Finally, from (\ref{eq:finite_dim_rep}), Lemma \ref{sineqs}(d) and using $%
b(\varepsilon )\varepsilon ^{-1}=1/a(\varepsilon )^{2}$,%
\begin{align*}
-b(\varepsilon )\log \mathbb{\bar{E}}\left[ e^{-F(Y^{\varepsilon
})/b(\varepsilon )}\right] & \leq b(\varepsilon )\varepsilon
^{-1}L_{T}(\varphi ^{\varepsilon })+F\circ \mathcal{G}^{\varepsilon
}(\varepsilon N^{\varepsilon ^{-1}\varphi ^{\varepsilon }}) \\
& \leq \frac{1}{2}\int_{\mathbb{X}\times \lbrack 0,T]}|\psi ^{\varepsilon
}|^{2}\ d\nu _{T}+\kappa _{3}\int_{\mathbb{X}\times \lbrack
0,T]}a(\varepsilon )|\psi ^{\varepsilon }|^{3}\ d\nu _{T}+F\circ \mathcal{G}%
^{\varepsilon }(\varepsilon N^{\varepsilon ^{-1}\varphi ^{\varepsilon }}) \\
& \leq \frac{1}{2}(1+2\kappa _{3}\beta )\int_{\mathbb{X}\times \lbrack
0,T]}|\psi |^{2}\ d\nu _{T}+F\circ \mathcal{G}^{\varepsilon }(\varepsilon
N^{\varepsilon ^{-1}\varphi ^{\varepsilon }}).
\end{align*}%
Taking limits as $\varepsilon \rightarrow 0$ and using \eqref{eq:eq847} we
have 
\begin{equation*}
\limsup_{\varepsilon \rightarrow 0}-b(\varepsilon )\log \mathbb{\bar{E}}%
\left[ e^{-F(Y^{\varepsilon })/b(\varepsilon )}\right] \leq \frac{1}{2}%
(1+2\kappa _{3}\beta )\int |\psi |^{2}d\nu _{T}+F(\eta ).
\end{equation*}%
Sending $\beta \rightarrow 0$ gives 
\begin{align*}
\limsup_{\varepsilon \rightarrow 0}-b(\varepsilon )\log \mathbb{\bar{E}}%
\left[ e^{-F(Y^{\varepsilon })/b(\varepsilon )}\right] & \leq \frac{1}{2}%
\int |\psi |^{2}d\nu _{T}+F(\eta ) \\
& \leq I(\eta )+F(\eta )+\delta /2 \\
& \leq \inf_{\eta \in \mathbb{U}}[I(\eta )+F(\eta )]+\delta ,
\end{align*}%
where the second inequality is from (\ref{eq:finite_dim_delta_opt2}) and the
last inequality follows from (\ref{eq:finite_dim_delta_opt}). Since $\delta
>0$ is arbitrary, this completes the proof of (\ref{eq:finite_dim_UB}).
\hfill $\Box $

\section{Proofs for the Finite Dimensional Problem (Theorem \protect\ref%
{LDPpoifin})}

\label{sect:pfs_finite_dim}From Theorem \ref{pathuniq} we see that there
exists a measurable map $\mathcal{\bar{G}}^{\varepsilon }:\mathbb{M}%
\rightarrow D([0,T]:\mathbb{R}^{d})$ such that $X^{\varepsilon }\equiv 
\mathcal{\bar{G}}^{\varepsilon }(\varepsilon N^{\varepsilon ^{-1}})$, and
hence there is a map $\mathcal{G}^{\varepsilon }$ such that $Y^{\varepsilon
}\equiv \mathcal{G}^{\varepsilon }(\varepsilon N^{\varepsilon ^{-1}})$.
Define $\mathcal{G}_{0}:L^{2}(\nu _{T})\rightarrow C([0,T]:\mathbb{R}^{d})$
by 
\begin{equation}
\mathcal{G}_{0}(\psi )=\eta \mbox{ if for }\psi \in L^{2}(\nu _{T}),\eta 
\mbox{ solves
(\ref{eq:finite_dim_limit}).}  \label{eq:defgnot}
\end{equation}%
In order to prove the theorem we will verify that Condition \ref{cond1}
holds with these choices of $\mathcal{G}^{\varepsilon }$ and $\mathcal{G}%
_{0} $.

We begin by verifying part (a) of the condition.

\begin{lemma}
\label{lem:lem636} Suppose Conditions \ref{cndns} and \ref{cndnsnew} hold.
Fix $M\in(0,\infty)$ and $g^{\varepsilon},g\in B_{2}(M)$ such that $%
g^{\varepsilon }\rightarrow g$. Let $\mathcal{G}_{0}$ be as defined in %
\eqref{eq:defgnot} Then $\mathcal{G}_{0}(g^{\varepsilon})\rightarrow\mathcal{%
G}_{0}(g)$.
\end{lemma}

\begin{proof}
Note that from Condition \ref{cndns}(c), $(y,s)\mapsto G(X^{0}(s),y)$ is in $%
L^{2}(\nu_{T})$. Thus, since $g^{\varepsilon}\rightarrow g$, we have for
every $t\in\lbrack0,T]$ 
\begin{equation}
\int_{\mathbb{X}\times\lbrack0,t]}g^{\varepsilon}(y,s)G(X^{0}(s),y)\nu
(dy)ds\rightarrow\int_{\mathbb{X}\times\lbrack0,t]}g(y,s)G(X^{0}(s),y)%
\nu(dy)ds.  \label{eq:eq644}
\end{equation}
We argue that the convergence is in fact uniform in $t$. For that note that
for $0\leq s\leq t\leq T$ 
\begin{align*}
\left\vert \int_{\mathbb{X}\times\lbrack
s,t]}g^{\varepsilon}(y,u)G(X^{0}(u),y)\nu(dy)du\right\vert \leq & \left(
1+\sup_{0\leq u\leq T}|X^{0}(u)|\right) \int_{\mathbb{X}\times\lbrack
s,t]}M_{G}(y)|g^{\varepsilon }(y,u)|\nu(dy)du \\
\leq & \left( 1+\sup_{0\leq u\leq T}|X^{0}(u)|\right)
|t-s|^{1/2}M\|M_{G}\|_{2},
\end{align*}
where abusing notation we have denoted the norm in $L^{2}(\nu)$ as $%
\|.\|_{2} $ as well. This implies equicontinuity, and hence the convergence
in \eqref{eq:eq644} is uniform in $t\in\lbrack0,T]$. The conclusion of the
lemma now follows by an application of Gronwall's lemma.
\end{proof}

In order to verify part (b) of Condition \ref{cond1}, we first prove some a
priori estimates. Recall the spaces $\mathcal{H}$ introduced in (\ref%
{eq:defofH}) and $\mathcal{S}_{+,\varepsilon }^{M}$ in (\ref{ctrlclass}).

\begin{lemma}
\label{Ineq1} Let $h\in L^{2}(\nu )\cap \mathcal{H}$ and let $M\in (0,\infty
)$. Then there exist $\varsigma \in (0,\infty )$ such that for any
measurable $I\subset \lbrack 0,T]$ and for all $\varepsilon >0$, 
\begin{equation*}
\sup_{\varphi \in \mathcal{S}_{+,\varepsilon }^{M}}\int_{\mathbb{X}\times
I}h^{2}(y)\varphi (y,s)\nu (dy)ds\leq \varsigma (a^{2}(\varepsilon )+|I|),
\end{equation*}%
where $|I|=\lambda _{T}(I)$.
\end{lemma}

\begin{proof}
Fix $\varphi \in \mathcal{S}_{+,\varepsilon }^{M}$ and let $\Gamma
=\{y:|h(y)|\geq 1\}$. Then 
\begin{equation}
\int_{\mathbb{X}\times I}h^{2}\varphi d\nu _{T}=\int_{\Gamma \times
I}h^{2}\varphi d\nu _{T}+\int_{\Gamma ^{c}\times I}h^{2}\varphi d\nu _{T}.
\label{eq:eq726}
\end{equation}%
The second term on the right side can be estimated as 
\begin{align}
\int_{\Gamma ^{c}\times I}h^{2}\varphi d\nu _{T}& \leq \int_{\mathbb{X}%
\times I}\varphi 1_{\left\{ \varphi \geq 1\right\} }d\nu _{T}+\int_{\mathbb{X%
}\times I}h^{2}d\nu _{T}  \notag \\
& \leq Ma^{2}(\varepsilon )\kappa _{1}^{\prime }(1)+|I\Vert |h\Vert _{2}^{2},
\label{eq:eq839}
\end{align}%
where the second inequality follows from Lemma \ref{sineqs2}(b). Consider
now the first term on the right side of \eqref{eq:eq726}. Note that $\nu
(\Gamma )\leq \Vert h\Vert _{2}^{2}<\infty .$ Let $\delta $ be as in the
definition of $\mathcal{H}$ and let $M(\delta )=\int_{\Gamma }e^{\delta
h^{2}(y)}\nu (dy)$. Then, applying the inequality in Lemma \ref{sineqs}(a)
with $\sigma =1$, $a=\delta h^{2}$ and $b=\varphi /\delta $, we have 
\begin{align}
\int_{\Gamma \times I}h^{2}(y)\varphi (y,s)\nu (dy)ds& \leq |I|\int_{\Gamma
}e^{\delta h^{2}(y)}\nu (dy)+\int_{\Gamma \times I}\ell (\varphi
(y,s)/\delta )\nu (dy)ds  \notag \\
& \leq M(\delta )|I|+\int_{\Gamma \times I}\ell (\varphi (y,s)/\delta )\nu
(dy)ds.  \label{eq:eq1121}
\end{align}%
Note that for $x\geq 0$ 
\begin{align}
\ell \left( \frac{x}{\delta }\right) & =\frac{x}{\delta }\log \left( \frac{x%
}{\delta }\right) -\frac{x}{\delta }+1  \notag \\
& =\frac{1}{\delta }\ell (x)+\frac{\delta -1}{\delta }-\frac{(x-1)}{\delta }%
\log \delta -\frac{\log \delta }{\delta }  \notag \\
& \leq \frac{1}{\delta }\ell (x)+|x-1|\frac{|\log \delta |}{\delta }+\frac{%
\delta +|\log \delta |}{\delta }.  \notag
\end{align}%
Thus 
\begin{equation}
\int_{\Gamma \times I}\ell (\varphi (y,s)/\delta )\nu (dy)ds\leq \frac{M}{%
\delta }a^{2}(\varepsilon )+c_{1}(\delta )a(\varepsilon )\int_{\Gamma \times
I}|\psi (y,s)|\nu (dy)ds+c_{2}(\delta )\nu (\Gamma )|I|,
\label{eq:split_of_cost}
\end{equation}%
where $c_{1}(\delta )=\frac{|\log \delta |}{\delta }$, $c_{2}(\delta )=\frac{%
\delta +|\log \delta |}{\delta }$ and $\psi =(\varphi -1)/a(\varepsilon )$.
Finally 
\begin{align}
a(\varepsilon )\int_{\Gamma \times I}|\psi |d\nu _{T}& =a(\varepsilon
)\int_{\Gamma \times I}|\psi |1_{\left\{ |\psi |\geq 1/a(\varepsilon
)\right\} }d\nu _{T}+a(\varepsilon )\int_{\Gamma \times I}|\psi |1_{\left\{
|\psi |<1/a(\varepsilon )\right\} }d\nu _{T}  \notag \\
& \leq Ma^{2}(\varepsilon )\kappa _{1}(1)+a(\varepsilon )|I|^{1/2}\sqrt{\nu
(\Gamma )}\left( \int |\psi |^{2}1_{\left\{ |\psi |<1/a(\varepsilon
)\right\} }d\nu _{T}\right) ^{1/2}  \notag \\
& \leq Ma^{2}(\varepsilon )\kappa _{1}(1)+a(\varepsilon )|I|^{1/2}\sqrt{\nu
(\Gamma )}(M\kappa _{2}(1))^{1/2},  \label{eq:eq121b}
\end{align}%
where the first inequality follows from Lemma \ref{sineqs2} (a) while the
second uses part (c) of the same lemma. The result now follows by combining %
\eqref{eq:eq1121} and \eqref{eq:eq839} with the last two displays and using $%
a(\varepsilon )|I|^{1/2}\leq (a(\varepsilon )^{2}+|I|)/2$.
\end{proof}

\begin{lemma}
\label{Ineq2} Let $h\in L^{2}(\nu)\cap\mathcal{H}$ and $I$ be a measurable
subset of $[0,T]$. Let $M\in(0,\infty)$. Then there exist maps $\vartheta
,\rho:(0,\infty)\rightarrow(0,\infty)$ such that $\vartheta(u)\downarrow 0$
as $u\uparrow\infty$, and for all $\varepsilon,\beta\in(0,\infty)$, 
\begin{equation*}
\sup_{\psi\in\mathcal{S}_{\varepsilon}^{M}}\int_{\mathbb{X}\times
I}|h(y)\psi(y,s)|1_{\{|\psi|\geq\beta/a(\varepsilon)\}}\nu(dy)ds\leq
\vartheta(\beta)(1+|I|^{1/2}),
\end{equation*}
and 
\begin{equation*}
\sup_{\psi\in\mathcal{S}_{\varepsilon}^{M}}\int_{\mathbb{X}\times
I}|h(y)\psi(y,s)|\nu(dy)ds\leq\rho(\beta)|I|^{1/2}+\vartheta(\beta
)a(\varepsilon).
\end{equation*}
\end{lemma}

\begin{proof}
Let $\psi \in \mathcal{S}_{\varepsilon }^{M}$ and $\beta \in (0,\infty )$.
Then 
\begin{align}
\int_{\mathbb{X}\times I}|h(y)\psi (y,s)|\nu (dy)ds& \leq \int_{\mathbb{X}%
\times I}|h(y)\psi (y,s)|1_{\{|\psi |<\beta /a(\varepsilon )\}}\nu (dy)ds 
\notag \\
& \hspace{0.3cm}+\int_{\mathbb{X}\times I}|h(y)\psi (y,s)|1_{\{|\psi |\geq
\beta /a(\varepsilon )\}}\nu (dy)ds.  \label{split}
\end{align}%
By the Cauchy-Schwarz inequality and Lemma \ref{sineqs2}(c) 
\begin{align}
\int_{\mathbb{X}\times I}|h(y)\psi (y,s)|1_{\{|\psi |<\beta /a(\varepsilon
)\}}\nu (dy)ds& \leq \left( |I|\int_{\mathbb{X}}h^{2}(y)\nu (dy)\int_{%
\mathbb{X}\times I}\psi ^{2}1_{\{|\psi |<\beta /a(\varepsilon )\}}\nu
(dy)ds\right) ^{1/2}  \notag \\
& \leq \Vert h\Vert _{2}(M\kappa _{2}(\beta ))^{1/2}|I|^{1/2}.
\label{split1}
\end{align}%
Let $\varphi =1+a(\varepsilon )\psi $, and note that $\varphi \in \mathcal{S}%
_{+,\varepsilon }^{M}$. For the second term in \eqref{split}, another
application of the Cauchy-Schwarz inequality and Lemma \ref{sineqs2}(a) give 
\begin{align}
\int_{\mathbb{X}\times I}|h(y)\psi (y,s)|1_{\{|\psi |\geq \beta
/a(\varepsilon )\}}\nu (dy)ds& \leq \left( \int_{\mathbb{X}\times
I}h^{2}(y)|\psi (y,s)|\nu (dy)ds\int_{\mathbb{X}\times I}|\psi
(y,s)|1_{\{|\psi |\geq \beta /a(\varepsilon )\}}\nu (dy)ds\right) ^{1/2} 
\notag \\
& \leq \left( Ma(\varepsilon )\kappa _{1}(\beta )\int_{\mathbb{X}\times
I}h^{2}(y)|\psi (y,s)|\nu (dy)ds\right) ^{1/2}  \notag \\
& =\left( M\kappa _{1}(\beta )\int_{\mathbb{X}\times I}h^{2}(y)|\varphi
(y,s)-1|\nu (dy)ds\right) ^{1/2}  \notag \\
& \leq \left( M\kappa _{1}(\beta )\left( |I|\Vert h\Vert _{2}^{2}+\int_{%
\mathbb{X}\times I}h^{2}\varphi d\nu _{T}\right) \right) ^{1/2}
\label{split2} \\
& \leq (M\kappa _{1}(\beta ))^{1/2}\left( \Vert h\Vert _{2}^{2}|I|+\varsigma
(a^{2}(\varepsilon )+|I|)\right) ^{1/2},  \label{split2n}
\end{align}%
where the last inequality is obtained by an application of Lemma \ref{Ineq1}%
. Recall from Lemma \ref{sineqs} that $\kappa _{1}(\beta )\overset{\beta
\rightarrow \infty }{\rightarrow }0$. The first statement in the lemma is
immediate from \eqref{split2n} while the second follows by adding %
\eqref{split1} and \eqref{split2n}.
\end{proof}

Recalling the definition of $\mathcal{U}_{+,\varepsilon }^{M}$ in (\ref%
{eq:def_of_Us}), we note that for every $\varphi \in \mathcal{U}%
_{+,\varepsilon }^{M}$ the integral equation 
\begin{equation}
d\bar{X}^{\varepsilon ,\varphi }(s)=b(\bar{X}^{\varepsilon ,\varphi
}(s))ds+\int_{\mathbb{X}_{T}}\varepsilon G(\bar{X}^{\varepsilon ,\varphi
}(s-),y)\ N^{\varepsilon ^{-1}\varphi }(ds,dy)  \label{eq:eq1152}
\end{equation}%
has a unique pathwise solution. Indeed, let $\tilde{\varphi}=1/\varphi $,
and recall that $\varphi \in \mathcal{U}_{+,\varepsilon }^{M}$ means that $%
\varphi =1$ off some compact set in $y$ and bounded above and below away
from zero on the compact set. Then it is easy to check (see Theorem III.3.24
of \cite{JacShi}, see also Lemma 2.3 of \cite{BDM09}) that 
\begin{equation*}
\mathcal{E}_{t}^{\varepsilon }(\tilde{\varphi})=\exp \left\{ \int_{\mathbb{(}%
0,t]\times \mathbb{X}\times \lbrack 0,\varepsilon ^{-1}\varphi ]}\log (%
\tilde{\varphi})d\bar{N}+\int_{\mathbb{(}0,t]\times \mathbb{X}\times \lbrack
0,\varepsilon ^{-1}\varphi ]}\left( -\tilde{\varphi}+1\right) d\bar{\nu}%
_{T}\right\}
\end{equation*}%
is an $\left\{ \bar{\mathcal{F}}_{t}\right\} $-martingale and consequently 
\begin{equation*}
\mathbb{Q}_{T}^{\varepsilon }(G)=\int_{G}\mathcal{E}_{T}^{\varepsilon }(%
\tilde{\varphi})d\mathbb{\bar{P}},\quad \text{ for }G\in \mathcal{B(}\mathbb{%
\bar{M}}\mathcal{\mathbb{\mathcal{)}}}
\end{equation*}%
defines a probability measure on $\mathbb{\bar{M}}$. Furthermore $\mathbb{%
\bar{P}}$ and $\mathbb{Q}_{T}^{\varepsilon }$ are mutually absolutely
continuous. Also it can be verified that under $\mathbb{Q}_{T}^{\varepsilon
} $, $\varepsilon N^{\varepsilon ^{-1}\varphi }$ has the same law as that of 
$\varepsilon N^{\varepsilon ^{-1}}$ under $\mathbb{\bar{P}}$. Thus it
follows that $\bar{X}^{\varepsilon ,\varphi }=\bar{\mathcal{G}}^{\varepsilon
}(\varepsilon N^{\varepsilon ^{-1}\varphi })$ is $\mathbb{Q}%
_{T}^{\varepsilon }$ a.s. (and hence $\mathbb{\bar{P}}$ a.s.) the unique
solution of \eqref{eq:eq1152}, where $\bar{\mathcal{G}}^{\varepsilon }$ is
as in Theorem \ref{pathuniq}.

Define $\bar{Y}^{\varepsilon ,\varphi }\equiv \mathcal{G}^{\varepsilon
}(\varepsilon N^{\varepsilon ^{-1}\varphi })$, and note that this is
equivalent to 
\begin{equation}
\bar{Y}^{\varepsilon ,\varphi }=\frac{1}{a(\varepsilon )}(\bar{X}%
^{\varepsilon ,\varphi }-X^{0}).  \label{eq:Y_X_relation}
\end{equation}%
The following moment bounds on $\bar{X}^{\varepsilon ,\varphi }$ and $\bar{Y}%
^{\varepsilon ,\varphi }$ will be useful for our analysis. Lemma \ref%
{bdX_poi}, which is needed in the proof of Lemma \ref{bdY_poi}, has a proof
that is very similar to the proof of Proposition 3.13 in \cite{BCD}. Since a
similar result for an infinite dimensional model is proved in detail in
Section \ref{sec:sec748} (see Lemma \ref{bdY_poi_inf} (a)), we omit the
proof of Lemma \ref{bdX_poi}.

\begin{lemma}
\label{bdX_poi} There exists an $\varepsilon_{0}\in(0,\infty)$ such that 
\begin{equation*}
\sup_{\varepsilon\in(0,\varepsilon_{0})}\sup_{\varphi\in\mathcal{U}%
_{+,\varepsilon}^{M}}\mathbb{\bar{E}}\left[ \sup_{0\leq s\leq T}|\bar {X}%
^{\varepsilon,\varphi}(s)|^{2}\right] <\infty.
\end{equation*}
\end{lemma}

\begin{lemma}
\label{bdY_poi} There exists an $\varepsilon_{0}\in(0,\infty)$ such that 
\begin{equation*}
\sup_{\varepsilon\in(0,\varepsilon_{0})}\sup_{\varphi\in\mathcal{U}%
_{+,\varepsilon}^{M}}\mathbb{\bar{E}}\left[ \sup_{0\leq s\leq T}|\bar {Y}%
^{\varepsilon,\varphi}(s)|^{2}\right] <\infty.
\end{equation*}
\end{lemma}

\begin{proof}
Fix $\varphi \in \mathcal{U}_{+,\varepsilon }^{M}$ and let $\psi =(\varphi
-1)/a(\varepsilon )$. Let $\tilde{N}^{\varepsilon ^{-1}\varphi }(dy,ds)={N}%
^{\varepsilon ^{-1}\varphi }(dy,ds)-\varepsilon ^{-1}\varphi (y,s)\nu (dy)ds$%
. Then 
\begin{align*}
\bar{X}^{\varepsilon ,\varphi }(t)-X^{0}(t)& =\int_{0}^{t}\left( b(\bar{X}%
^{\varepsilon ,\varphi }(s))-b(X^{0}(s))\right) ds+\int_{\mathbb{X}\times
\lbrack 0,t]}\varepsilon G(\bar{X}^{\varepsilon ,\varphi }(s-),y)\tilde{N}%
^{\varepsilon ^{-1}\varphi }(dy,ds) \\
& \hspace{0.2cm}+\int_{\mathbb{X}\times \lbrack 0,t]}\left( G(\bar{X}%
^{\varepsilon ,\varphi }(s),y)-G(X^{0}(s),y)\right) \varphi (y,s)\nu (dy)ds
\\
& \hspace{0.2cm}+\int_{\mathbb{X}\times \lbrack 0,t]}G(X^{0}(s),y)(\varphi
(y,s)-1)\nu (dy)ds.
\end{align*}%
Write 
\begin{equation}
\bar{Y}^{\varepsilon ,\varphi }=A^{\varepsilon ,\varphi }+M^{\varepsilon
,\varphi }+B^{\varepsilon ,\varphi }+\mathcal{E}_{1}^{\varepsilon ,\varphi
}+C^{\varepsilon ,\varphi },  \label{eq:eq1109}
\end{equation}%
where 
\begin{align}
M^{\varepsilon ,\varphi }(t)& =\frac{\varepsilon }{a(\varepsilon )}\int_{%
\mathbb{X}\times \lbrack 0,t]}G(\bar{X}^{\varepsilon ,\varphi }(s-),y)\tilde{%
N}^{\varepsilon ^{-1}\varphi }(ds,dy),  \notag \\
A^{\varepsilon ,\varphi }(t)& =\frac{1}{a(\varepsilon )}\int_{0}^{t}\left( b(%
\bar{X}^{\varepsilon ,\varphi }(s))-b(X^{0}(s))\right) ds,  \notag \\
B^{\varepsilon ,\varphi }(t)& =\frac{1}{a(\varepsilon )}\int_{\mathbb{X}%
\times \lbrack 0,t]}\left( G(\bar{X}^{\varepsilon ,\varphi
}(s),y)-G(X^{0}(s),y)\right) \nu (dy)ds,  \notag \\
\mathcal{E}_{1}^{\varepsilon ,\varphi }(t)& =\int_{\mathbb{X}\times \lbrack
0,t]}\left( G(\bar{X}^{\varepsilon ,\varphi }(s),y)-G(X^{0}(s),y)\right)
\psi (y,s)\nu (dy)ds,  \notag \\
C^{\varepsilon ,\varphi }(t)& =\int_{\mathbb{X}\times \lbrack
0,t]}G(X^{0}(s),y)\psi (y,s)\nu (dy)ds.  \label{eq:eq722}
\end{align}%
Noting that $M^{\varepsilon ,\varphi }$ is a martingale, Doob's inequality
gives 
\begin{align*}
\mathbb{\bar{E}}\left[ \sup_{r\leq T}|M^{\varepsilon ,\varphi }(r)|^{2}%
\right] & \leq \left( \frac{\varepsilon }{a(\varepsilon )}\right) ^{2}%
\mathbb{\bar{E}}\left[ \int_{\mathbb{X}\times \lbrack 0,T]}|G(\bar{X}%
^{\varepsilon ,\varphi }(s),y)|^{2}\varepsilon ^{-1}\varphi (y,s)\nu (dy)ds%
\right] \\
& \leq \frac{2\varepsilon }{a^{2}(\varepsilon )}\mathbb{\bar{E}}\left[
\left( 1+\sup_{t\leq T}|\bar{X}^{\varepsilon ,\varphi }(t)|^{2}\right) \int_{%
\mathbb{X}\times \lbrack 0,T]}M_{G}(y)^{2}\varphi (y,s)\nu (dy)ds\right] .
\end{align*}%
From Condition \ref{cndnsnew} $M_{G}\in \mathcal{H}$, and so Lemmas \ref%
{Ineq1} and \ref{bdX_poi} imply that for some $\gamma _{1}\in (0,\infty )$
and $\varepsilon \in (0,\varepsilon _{0})$, where $\varepsilon _{0}$ is as
in Lemma \ref{bdY_poi}. 
\begin{equation}
\sup_{\varphi \in \mathcal{U}_{+,\varepsilon }^{M}}\mathbb{\bar{E}}\left[
\sup_{r\leq T}|M^{\varepsilon ,\varphi }(r)|^{2}\right] \leq \gamma _{1}%
\frac{\varepsilon }{a^{2}(\varepsilon )}.  \label{eq:bound_on_M}
\end{equation}%
Next by the Lipschitz condition on $G$ (Condition \ref{cndns}(b) and
Condition \ref{cndnsnew}(a)) and the Cauchy-Schwarz inequality, there is $%
\gamma _{2}\in (0,\infty )$ such that for all $t\in \lbrack 0,T]$ and $%
\varphi \in \mathcal{U}_{+,\varepsilon }^{M}$ 
\begin{align}
\sup_{r\leq t}|\mathcal{E}_{1}^{\varepsilon ,\varphi }(r)|^{2}& \leq
a^{2}(\varepsilon )\left( \int_{\mathbb{X}\times \lbrack 0,t]}L_{G}(y)|\bar{Y%
}^{\varepsilon ,\varphi }(s)\Vert \psi (y,s)|\nu (dy)ds\right) ^{2}  \notag
\\
& \leq a^{2}(\varepsilon )\sup_{s\leq t}|\bar{Y}^{\varepsilon ,\varphi
}(s)|^{2}\left( \int_{\mathbb{X}\times \lbrack 0,t]}L_{G}(y)|\psi (y,s)|\nu
(dy)ds\right) ^{2}  \notag \\
& \leq \gamma _{2}a^{2}(\varepsilon )\sup_{s\leq t}|\bar{Y}^{\varepsilon
,\varphi }(s)|^{2},  \label{bdE}
\end{align}%
where the last inequality follows from Lemma \ref{Ineq2}.\newline

Again using the Lipschitz condition on $G$ we have, for all $t\in \lbrack
0,T]$ 
\begin{equation*}
\sup_{r\leq t}|B^{\varepsilon ,\varphi }(r)|^{2}\leq T\Vert L_{G}\Vert
_{1}^{2}\int_{0}^{t}|\bar{Y}^{\varepsilon ,\varphi }(s)|^{2}\ ds.
\end{equation*}%
Similarly, the Lipschitz condition on $b$ gives 
\begin{equation*}
\sup_{r\leq t}|A^{\varepsilon ,\varphi }(r)|^{2}\leq TL_{b}^{2}\int_{0}^{t}|%
\bar{Y}^{\varepsilon ,\varphi }(s)|^{2}\ ds.
\end{equation*}%
From Lemma \ref{Ineq2} again we have that, for some $\gamma _{3}\in
(0,\infty )$ and all $\varphi \in \mathcal{U}_{+,\varepsilon }^{M}$ 
\begin{align*}
\sup_{r\leq T}|C^{\varepsilon ,\varphi }(r)|^{2}& \leq \left( \int_{\mathbb{X%
}\times \lbrack 0,T]}\left\vert G(X^{0}(s),y)\psi (y,s)\right\vert \nu
(dy)ds\right) ^{2} \\
& \leq \left( 1+\sup_{s\leq T}|X^{0}(s)|^{2}\right) \left( \int_{\mathbb{X}%
\times \lbrack 0,T]}M_{G}(y)|\psi (y,s)|\nu (dy)ds\right) ^{2} \\
& \leq \gamma _{3},
\end{align*}%
Collecting these estimates we have, for some $\gamma _{4}\in (0,\infty )$
and all $\varphi \in \mathcal{U}_{+,\varepsilon }^{M}$, $t\in \lbrack 0,T]$ 
\begin{equation*}
(1-5a^{2}(\varepsilon )\gamma _{2})\mathbb{\bar{E}}\left[ \sup_{s\leq t}(%
\bar{Y}^{\varepsilon ,\varphi }(s))^{2}\right] \leq \gamma _{4}\left(
1+\int_{0}^{t}\mathbb{\bar{E}}\left[ \sup_{s\leq r}(\bar{Y}^{\varepsilon
,\varphi }(r))^{2}\right] \ dr\right) .
\end{equation*}%
Choose $\varepsilon _{0}$ such that $(1-5a^{2}(\varepsilon )\gamma _{2})>1/2$,%
 for all $\varepsilon \leq
\varepsilon _{0}$. The result now follows from Gronwall's inequality.
\end{proof}

\begin{lemma}
\label{ctrlbound} Let $h\in L^{2}(\nu)\cap\mathcal{H}$. Then for every $%
\delta>0$ there exists a compact set $C\subset\mathbb{X}$ such that 
\begin{equation*}
\sup_{\varepsilon>0}\sup_{\psi\in\mathcal{S}_{\varepsilon}^{M}}\int
_{C^{c}\times\lbrack0,T]}|h(y)\psi(y,s)|\nu(dy)ds<\delta.
\end{equation*}
\end{lemma}

\begin{proof}
By Lemma \ref{Ineq2}, for all $\psi \in \mathcal{S}_{\varepsilon }^{M}$ 
\begin{equation}
\int_{\mathbb{X}\times \lbrack 0,T]}|h(y)\psi (y,s)1_{\{|\psi |\geq \beta
/a(\varepsilon )\}}|\nu (dy)ds\leq \vartheta (\beta )(1+T^{1/2}),  \notag
\label{ineqsplit1}
\end{equation}%
where $\vartheta (\beta )\downarrow 0$ as $\beta \uparrow \infty $. Choose $%
\beta _{0}<\infty $ such that $\vartheta (\beta _{0})(1+T^{1/2})<\delta /2$.
Next, from Lemma \ref{sineqs2}(c), for any compact set $C$ in $\mathbb{X}$ 
\begin{align*}
\int_{C^{c}\times \lbrack 0,T]}|h(y)\psi (y,s)1_{\{|\psi |\leq \beta
_{0}/a(\varepsilon )\}}|\nu (dy)ds& \leq \left( T\int_{C^{c}}h^{2}(y)\nu
(dy)\int_{\mathbb{X}\times \lbrack 0,T]}\psi ^{2}1_{\{|\psi |\leq \beta
_{0}/a(\varepsilon )\}}\right) ^{1/2} \\
& \leq \left( MT\kappa _{2}(\beta _{0})\int_{C^{c}}h^{2}(y)\nu (dy)\right)
^{1/2}.
\end{align*}%
Since $h\in L^{2}(\nu )$, we can find a compact set $C$ such that 
\begin{equation*}
\left( MT\kappa _{2}(\beta _{0})\int_{C^{c}}h^{2}(y)\nu (dy)\right)
^{1/2}<\delta /2,
\end{equation*}%
and the result follows.
\end{proof}

\begin{lemma}
\label{ctrllim0} Let $h\in L^{2}(\nu )\cap \mathcal{H}$ and suppose $h\geq 0$%
. Then for any $\beta <\infty $, 
\begin{equation*}
\sup_{\psi \in \mathcal{S}_{\varepsilon }^{M}}\int_{\mathbb{X}\times \lbrack
0,T]}h(y)|\psi (y,s)|1_{\{|\psi |>\beta /a(\varepsilon )\}}\nu
(dy)ds\rightarrow 0\quad \mbox{as\ }\varepsilon \rightarrow 0.
\end{equation*}
\end{lemma}

\begin{proof}
In view of Lemma \ref{ctrlbound}, it suffices to show that for any compact
subset $C$ of $\mathbb{X}$, 
\begin{equation}
\sup_{\psi \in \mathcal{S}_{\varepsilon }^{M}}\int_{C\times \lbrack
0,T]}h(y)|\psi (y,s)|1_{\{|\psi |>\beta /a(\varepsilon )\}}\nu
(dy)ds\rightarrow 0\quad \mbox{as\ }\varepsilon \rightarrow 0.
\label{eq:L4.7_suff}
\end{equation}%
For $K\in (0,\infty )$, write 
\begin{equation*}
h=h1_{\{h\leq K\}}+h1_{\{h>K\}}.
\end{equation*}%
From Lemma \ref{sineqs2}(a), for any $\psi \in \mathcal{S}_{\varepsilon
}^{M} $, 
\begin{align}
\int_{C\times \lbrack 0,T]}h(y)1_{\{h\leq K\}}|\psi (y,s)|1_{\{|\psi |>\beta
/a(\varepsilon )\}}\nu (dy)ds& \leq K\int_{C\times \lbrack 0,T]}|\psi
(y,s)|1_{\{|\psi |>\beta /a(\varepsilon )\}}\nu (dy)ds  \notag
\label{hsplit1} \\
& \leq KM\kappa _{1}(\beta )a(\varepsilon ).
\end{align}%
The same calculation as in \eqref{split2}, but with $\mathbb{X}$ replaced by 
$C$ and $h$ with $h1_{\{h>K\}}$, gives 
\begin{align}
\int_{C\times \lbrack 0,T]}h1_{\{h>K\}}|\psi |1_{\{|\psi |>\beta
/a(\varepsilon )\}}d\nu _{T}& \leq \Big(M\kappa _{1}(\beta )\Big(%
T\int_{C}h^{2}1_{\{h>K\}}d\nu  \notag  \label{hsplit2} \\
& \hspace{0.6cm}+\int_{C\times \lbrack 0,T]}h^{2}1_{\{h>K\}}\varphi d\nu _{T}%
\Big)\Big)^{1/2},
\end{align}%
where $\varphi =1+a(\varepsilon )\psi $. Now using Lemma \ref{sineqs}(a)
with $\sigma =1$, $a=\delta h^{2}$, $b=\varphi /\delta $, where $\delta $ is
as in (\ref{eq:defofH}), we have 
\begin{equation}
\int_{C\times \lbrack 0,T]}h^{2}1_{\{h>K\}}\varphi d\nu _{T}\leq
T\int_{C}e^{\delta h^{2}}1_{\{h>K\}}\ d\nu +\int_{C\times \lbrack
0,T]}1_{\{h>K\}}\ell (\varphi /\delta )d\nu _{T}.  \label{hsplit3}
\end{equation}%
Next, noting that (\ref{eq:split_of_cost}), \eqref{eq:eq121b} hold for all
measurable sets $\Gamma $ in $\mathbb{X}$, we have on applying these
inequalities with $\Gamma =C\cap \{h>K\}$, 
\begin{align}
\int_{C\times \lbrack 0,T]}1_{\{h>K\}}\ell (\varphi /\delta )d\nu _{T}& \leq 
\frac{M}{\delta }a^{2}(\varepsilon )+c_{1}(\delta )a(\varepsilon
)\int_{C\times \lbrack 0,T]}|\psi |1_{\{h>K\}}d\nu _{T}+c_{2}(\delta )T\nu
(C\cap \lbrack h>K])  \notag \\
& \leq \frac{M}{\delta }a^{2}(\varepsilon )+Mc_{1}(\delta )a^{2}(\varepsilon
)\kappa _{1}(1)+c_{1}(\delta )a(\varepsilon )T^{1/2}\sqrt{\nu (C)}(M\kappa
_{2}(1))^{1/2}  \notag \\
& \quad +c_{2}(\delta )T\nu (C\cap \lbrack h>K]).  \label{eq:eq150}
\end{align}%
Thus from \eqref{hsplit1}, \eqref{hsplit2}, \eqref{hsplit3} and %
\eqref{eq:eq150} we have, for some $M_{1}<\infty $,
\begin{equation*}
\sup_{\psi \in \mathcal{S}_{\varepsilon }^{M}}\int_{C\times \lbrack
0,T]}h|\psi |1_{\{\psi >\beta /a(\varepsilon )\}}d\nu _{T}\leq M_{1}\left(
a^{2}(\varepsilon )+(K+1)a(\varepsilon )+\int_{C}e^{\delta
h^{2}}1_{\{h>K\}}d\nu \right) ^{1/2},
\end{equation*}%
for all $\varepsilon >0$, $K<\infty $. Since $\int_{C}e^{\delta
h^{2}(y)}1_{\{h>K\}}\nu (dy)\rightarrow 0$ as $K\rightarrow \infty $, the
statement in (\ref{eq:L4.7_suff}) follows on sending first $\varepsilon
\rightarrow 0$ and then $K\rightarrow \infty $.
\end{proof}

\begin{lemma}
\label{ctrllim} Let $\{\psi^{\varepsilon}\}_{\varepsilon>0}$ be such that
for some $M<\infty$, $\psi^{\varepsilon}\in\mathcal{S}_{\varepsilon}^{M} $
for all $\varepsilon>0$. Let $f:\mathbb{X}\times\lbrack0,T]\rightarrow%
\mathbb{R}^{d}$ be such that 
\begin{equation*}
|f(y,s)|\leq h(y),y\in\mathbb{X},s\in\lbrack0,T]
\end{equation*}
for some $h$ in $L^{2}(\nu)\cap\mathcal{H}$. Suppose for some $\beta\in(0,1]$
that $\psi^{\varepsilon}1_{\{|\psi^{\varepsilon}|\leq\beta/a(\varepsilon)\}}$
converges in $B_{2}((M\kappa_{2}(1))^{1/2})$ to $\psi$. Then 
\begin{equation*}
\int_{\mathbb{X}\times\lbrack0,t]}f\psi^{\varepsilon}d\nu_{T}\rightarrow
\int_{\mathbb{X}\times\lbrack0,t]}f\psi d\nu_{T},\;\mbox{
for all }t\in\lbrack0,T].
\end{equation*}
\end{lemma}

\begin{proof}
From Lemma \ref{ctrllim0} we have that 
\begin{equation*}
\int_{\mathbb{X}\times \lbrack 0,T]}|f\psi ^{\varepsilon }|1_{\{|\psi
^{\varepsilon }|>\beta /a(\varepsilon )\}}d\nu _{T}\rightarrow 0\;\mbox{ as }%
\varepsilon \rightarrow 0.
\end{equation*}%
Also, since $f1_{[0,t]}\in L^{2}(\nu _{T})$ for all $t\in \lbrack 0,T]$ and $%
\psi ^{\varepsilon }1_{|\psi ^{\varepsilon }|\leq \beta /a(\varepsilon
)}\rightarrow \psi $, we have 
\begin{equation*}
\int_{\mathbb{X}\times \lbrack 0,t]}f\psi ^{\varepsilon }1_{\{|\psi
^{\varepsilon }|\leq \beta /a(\varepsilon )\}}d\nu _{T}\rightarrow \int_{%
\mathbb{X}\times \lbrack 0,t]}f\psi d\nu _{T}.
\end{equation*}%
The result follows on combining the last two displays.
\end{proof}

\subsection{Proof of Theorem \protect\ref{LDPpoifin}}

The following is the key result needed for the proof of the theorem. It
gives tightness of the joint distribution of controls and controlled
processes, and indicates how limits of these two quantities are related.

\begin{lemma}
\label{cont_lim_poi} Let $\{\varphi ^{\varepsilon }\}_{\varepsilon >0}$ be
such that for some $M<\infty $, $\varphi ^{\varepsilon }\in \mathcal{U}%
_{+,\varepsilon }^{M}$ for every $\varepsilon >0$. Let $\psi ^{\varepsilon
}=(\varphi ^{\varepsilon }-1)/a(\varepsilon )$ and $\beta \in (0,1]$.
Suppose that $\bar{Y}^{\varepsilon ,\varphi ^{\varepsilon }}=\mathcal{G}%
^{\varepsilon }(\varepsilon N^{\varepsilon ^{-1}\varphi ^{\varepsilon }})$,
and recall that $\bar{Y}^{\varepsilon ,\varphi ^{\varepsilon }}=(\bar{X}%
^{\varepsilon ,\varphi ^{\varepsilon }}-X^{0})/a(\varepsilon )$, where $\bar{%
X}^{\varepsilon ,\varphi ^{\varepsilon }}=\mathcal{\bar{G}}^{\varepsilon
}(\varepsilon N^{\varepsilon ^{-1}\varphi ^{\varepsilon }})$. Then $\{(\bar{Y%
}^{\varepsilon ,\varphi ^{\varepsilon }},\psi ^{\varepsilon }1_{\{|\psi
^{\varepsilon }|\leq \beta /a(\varepsilon )\}})\}$ is tight in $D([0,T]:%
\mathbb{R}^{d})\times B_{2}((M\kappa _{2}(1))^{1/2})$, and any limit point $(%
\bar{Y},\psi )$ satisfies (\ref{eq:finite_dim_limit}) with $\eta $ replaced
by $\bar{Y}$, w.p.1.
\end{lemma}

\begin{proof}
We will use the notation from the proof of Lemma \ref{bdY_poi}. Assume
without loss of generality that $\varepsilon \leq \varepsilon _{0}$. From %
\eqref{eq:eq1138} and (\ref{eq:bound_on_M}) we have that $\mathbb{\bar{E}}%
[\sup_{r\leq T}|M^{\varepsilon ,\varphi ^{\varepsilon }}(r)|^{2}]\rightarrow
0$ as $\varepsilon \rightarrow 0$. Also, since from Lemma \ref{bdY_poi} $%
\sup_{\varepsilon \leq \varepsilon _{0}}\mathbb{\bar{E}}\left[ \sup_{s\leq
T}(\bar{Y}^{\varepsilon ,\varphi ^{\varepsilon }}(s))^{2}\right] <\infty $, %
\eqref{bdE} implies that $\mathbb{\bar{E}}[\sup_{s\leq T}|\mathcal{E}%
_{1}^{\varepsilon ,\varphi ^{\varepsilon }}(s)|^{2}]\rightarrow 0$.

Noting that $\bar{X}^{\varepsilon ,\varphi ^{\varepsilon
}}(t)=X^{0}(t)+a(\varepsilon )\bar{Y}^{\varepsilon ,\varphi ^{\varepsilon
}}(t)$ we have by Taylor's theorem 
\begin{equation*}
G(\bar{X}^{\varepsilon ,\varphi ^{\varepsilon
}}(s),y)-G(X^{0}(s),y)=a(\varepsilon )D_{x}G(X^{0}(s),y)\bar{Y}^{\varepsilon
,\varphi ^{\varepsilon }}(t)+R^{\varepsilon ,\varphi ^{\varepsilon }}(s,y)
\end{equation*}%
where 
\begin{equation*}
|R^{\varepsilon ,\varphi ^{\varepsilon }}(s,y)|\leq
L_{DG}(y)a^{2}(\varepsilon )|\bar{Y}^{\varepsilon ,\varphi ^{\varepsilon
}}(s)|^{2}
\end{equation*}%
Hence 
\begin{equation*}
B^{\varepsilon ,\varphi ^{\varepsilon }}(t)=\int_{\mathbb{X}\times \lbrack
0,t]}D_{x}G(X^{0}(s),y)\bar{Y}^{\varepsilon ,\varphi ^{\varepsilon }}(s)\nu
(dy)ds+\mathcal{E}_{2}^{\varepsilon ,\varphi ^{\varepsilon }}(t),
\end{equation*}%
where with $K_{DG}=\int_{\mathbb{X}}L_{DG}(y)\nu (dy)$, 
\begin{equation*}
\sup_{r\leq T}|\mathcal{E}_{2}^{\varepsilon ,\varphi ^{\varepsilon
}}(r)|\leq K_{DG}\,a(\varepsilon )\int_{0}^{T}|\bar{Y}^{\varepsilon ,\varphi
^{\varepsilon }}(s)|^{2}ds.
\end{equation*}%
Thus using Lemma \ref{bdY_poi} again, $\mathbb{\bar{E}}[\sup_{r\leq T}|%
\mathcal{E}_{2}^{\varepsilon ,\varphi ^{\varepsilon }}(r)|]\rightarrow 0$.
Similarly, 
\begin{equation*}
A^{\varepsilon ,\varphi ^{\varepsilon }}(t)=\int_{0}^{t}Db(X^{0}(s))\bar{Y}%
^{\varepsilon ,\varphi }(s)ds+\mathcal{E}_{3}^{\varepsilon ,\varphi
^{\varepsilon }}(t)
\end{equation*}%
where $\mathbb{\bar{E}}[\sup_{r\leq T}|\mathcal{E}_{3}^{\varepsilon ,\varphi
^{\varepsilon }}(r)|]\rightarrow 0$.

Putting these estimates together we have from \eqref{eq:eq1109} 
\begin{align}
\bar{Y}^{\varepsilon ,\varphi ^{\varepsilon }}(t)& =\mathcal{E}^{\varepsilon
,\varphi ^{\varepsilon }}(t)+\int_{0}^{t}Db(X^{0}(s))\bar{Y}^{\varepsilon
,\varphi ^{\varepsilon }}(s)ds  \notag  \label{ybarep} \\
& \hspace{0.5cm}+\int_{\mathbb{X}\times \lbrack 0,t]}D_{x}G(X^{0}(s),y)\bar{Y%
}^{\varepsilon ,\varphi ^{\varepsilon }}(s)\nu (dy)ds+\int_{\mathbb{X}\times
\lbrack 0,t]}G(X^{0}(s),y)\psi ^{\varepsilon }(y,s)\nu (dy)ds
\end{align}%
where $\mathcal{E}^{\varepsilon ,\varphi ^{\varepsilon }}=M^{\varepsilon
,\varphi ^{\varepsilon }}+\mathcal{E}_{1}^{\varepsilon ,\varphi
^{\varepsilon }}+\mathcal{E}_{2}^{\varepsilon ,\varphi ^{\varepsilon }}+%
\mathcal{E}_{3}^{\varepsilon ,\varphi ^{\varepsilon }}\Rightarrow 0$.

We now prove tightness of 
\begin{equation*}
\tilde{B}^{\varepsilon,\varphi^{\varepsilon}}(\cdot)=\int_{\mathbb{X}%
\times\lbrack0,\cdot]}D_{x}G(X^{0}(s),y)\bar{Y}^{\varepsilon,\varphi
^{\varepsilon}}(s)\nu(dy)ds,\quad\
C^{\varepsilon,\varphi^{\varepsilon}}(\cdot)=\int_{\mathbb{X}%
\times\lbrack0,\cdot]}G(X^{0}(s),y)\psi^{\varepsilon }(y,s)\nu(dy)ds
\end{equation*}
and 
\begin{equation*}
\tilde{A}^{\varepsilon,\varphi^{\varepsilon}}(\cdot)=\int_{0}^{%
\cdot}Db(X^{0}(s))\bar{Y}^{\varepsilon,\varphi^{\varepsilon}}(s)ds.
\end{equation*}
Applying Lemma \ref{Ineq2} with $h=M_{G}$ 
\begin{align}
|C^{\varepsilon,\varphi^{\varepsilon}}(t+\delta)-C^{\varepsilon,\varphi
^{\varepsilon}}(t)| & =\int_{\mathbb{X}\times\lbrack
t,t+\delta]}|G(X^{0}(s),y)\|\psi^{\varepsilon}(y,s)|\nu(dy)ds  \notag \\
& \leq\left( 1+\sup_{0\leq s\leq T}|X_{s}^{0}|\right) \int_{\mathbb{X}%
\times\lbrack t,t+\delta]}M_{G}(y)|\psi^{\varepsilon}(y,s)|\nu (dy)ds  \notag
\\
& \leq\left( 1+\sup_{0\leq s\leq T}|X_{s}^{0}|\right) (\rho(1)\delta
^{1/2}+\vartheta(1)a(\varepsilon)).  \label{eq:eq724}
\end{align}
Tightness of $\{C^{\varepsilon,\varphi^{\varepsilon}}\}_{\varepsilon>0}$ in $%
C([0,T]:\mathbb{R}^{d})$ is now immediate.

Next we argue the tightness of $\tilde{B}^{\varepsilon,\varphi^{%
\varepsilon}} $. Recall that $m_{T}=\sup_{s\in\lbrack0,T]}|X^{0}(s)|$. Then,
for $0\leq t\leq t+\delta\leq T$ 
\begin{align*}
|\tilde{B}^{\varepsilon,\varphi^{\varepsilon}}(t+\delta)-\tilde{B}%
^{\varepsilon,\varphi^{\varepsilon}}(t)|^{2} & =\left( \int_{\mathbb{X}%
\times\lbrack t,t+\delta]}D_{x}G(X^{0}(s),y)\bar{Y}^{\varepsilon
,\varphi^{\varepsilon}}(s)\nu(dy)ds\right) ^{2} \\
& \leq\left( \left( \sup_{|x|\leq m_{T}}\int_{\mathbb{X}}|D_{x}G(x,y)|_{op}%
\nu(dy)\right) \int_{[t,t+\delta]}|\bar{Y}^{\varepsilon
,\varphi^{\varepsilon}}(s)|ds\right) ^{2} \\
& \leq K_{2}\sup_{t\leq T}|\bar{Y}^{\varepsilon,\varphi^{%
\varepsilon}}(t)|^{2}\delta,
\end{align*}
where $K_{2}=T\sup_{|x|\leq m_{T}}\int_{\mathbb{X}}|D_{x}G(x,y)|_{op}\nu(dy) 
$, which is finite from Condition \ref{cndnsnew} (b). Tightness of $\{\tilde{%
B}^{\varepsilon,\varphi^{\varepsilon}}\}_{\varepsilon>0}$ in $C([0,T]:%
\mathbb{R}^{d})$ now follows as a consequence of Lemma \ref{bdY_poi}.
Similarly it can be seen that $\tilde{A}^{\varepsilon,\varphi^{\varepsilon}}$
is tight in $C([0,T]:\mathbb{R}^{d})$ and consequently, $\bar{Y}%
^{\varepsilon,\varphi^{\varepsilon}}$ is tight in $D([0,T]:\mathbb{R}^{d}) $%
. Also, from Lemma \ref{sineqs2}(c), $\psi^{\varepsilon}1_{\{|\psi^{%
\varepsilon }|\leq\beta/a(\varepsilon)\}}$ takes values in $%
B_{2}((M\kappa_{2}(1))^{1/2})$ for all $\varepsilon>0$ and by the
compactness of the latter space the tightness of $\psi^{\varepsilon}1_{\{|%
\psi^{\varepsilon}|\leq\beta /a(\varepsilon)\}}$ is immediate. This
completes the proof of the first part of the lemma. Suppose now that $(\bar{Y%
}^{\varepsilon,\varphi^{\varepsilon}},\psi^{\varepsilon}1_{\{|\psi^{%
\varepsilon}|\leq\beta/a(\varepsilon)\}}) $ along a subsequence converges in
distribution to $(\bar{Y},\psi)$. From Lemma \ref{ctrllim} and the tightness
of $C^{\varepsilon,\varphi^{\varepsilon}}$ established above 
\begin{equation*}
\left( \bar{Y}^{\varepsilon,\varphi^{\varepsilon}},\int_{\mathbb{X}%
\times\lbrack0,\cdot]}G(X^{0}(s),y)\psi^{\varepsilon}(y,s)\nu(dy)ds\right )
\end{equation*}
converges in distribution, in $D([0,T]:\mathbb{R}^{2d})$, to $(\bar{Y},\int_{%
\mathbb{X}\times\lbrack0,\cdot]}G(X^{0}(s),y)\psi(y,s)\nu(dy)ds)$. The
result now follows on using this in \eqref{ybarep} and recalling that $%
\mathcal{E}^{\varepsilon,\varphi^{\varepsilon}}\Rightarrow0$.
\end{proof}

We now complete the proof of Theorem \ref{LDPpoifin}.\newline

\begin{proof}[Proof of Theorem \protect\ref{LDPpoifin}]
It suffices to show that Condition \ref{cond1} holds with $\mathcal{G}%
^{\varepsilon }$ and $\mathcal{G}_{0}$ defined as at the beginning of the
section. Part (a) of the condition was verified in Lemma \ref{lem:lem636}.
Consider now part (b). Fix $M\in (0,\infty )$ and $\beta \in (0,1]$. Let $%
\{\varphi ^{\varepsilon }\}_{\varepsilon >0}$ be such that for every $%
\varepsilon >0$, $\varphi ^{\varepsilon }\in \mathcal{U}_{+,\varepsilon
}^{M} $ and $\psi ^{\varepsilon }1_{\{|\psi ^{\varepsilon }|<\beta
/a(\varepsilon )\}}\Rightarrow \psi $ in $B_{2}((M\kappa _{2}(1))^{1/2})$,
where $\psi ^{\varepsilon }=(\varphi ^{\varepsilon }-1)/a(\varepsilon )$. To
complete the proof we need to show that 
\begin{equation*}
\mathcal{G}^{\varepsilon }(\varepsilon N^{\varepsilon ^{-1}\varphi
^{\varepsilon }})\Rightarrow \mathcal{G}_{0}(\psi ).  
\end{equation*}%
Recall that $\mathcal{G}^{\varepsilon }(\varepsilon N^{\varepsilon
^{-1}\varphi ^{\varepsilon }})=\bar{Y}^{\varepsilon ,\varphi ^{\varepsilon
}} $. From Lemma \ref{cont_lim_poi} $\{(\bar{Y}^{\varepsilon ,\varphi
^{\varepsilon }},\psi ^{\varepsilon }1_{\{|\psi ^{\varepsilon }|\leq \beta
/a(\varepsilon )\}})\ $ is tight in $D([0,T]:\mathbb{R}^{d})\times
B_{2}((M\kappa _{2}(1))^{1/2})$ and every limit point of $\bar{Y}%
^{\varepsilon ,\varphi ^{\varepsilon }}$ must equal $\mathcal{G}_{0}(\psi )$%
. The result follows.
\end{proof}

\subsection{Proof of Theorem \protect\ref{th:thequivratefn}}

Fix $\eta \in C([0,T]:\mathbb{R}^{d})$ and $\delta >0$. Let $u\in
L^{2}([0,T]:\mathbb{R}^{d})$ be such that 
\begin{equation*}
\frac{1}{2}\int_{0}^{T}\left\vert u(s)\right\vert ^{2}\ ds\leq I(\eta
)+\delta
\end{equation*}%
and $(\eta ,u)$ satisfy (\ref{eq:finite_dim_limit_u}). Define $\psi :\mathbb{%
X}\times \lbrack 0,T]\rightarrow \mathbb{R}$ by 
\begin{equation}
\psi (y,s)\doteq \sum_{i=1}^{d}u_{i}(s)e_{i}(y,s),\;(y,s)\in \mathbb{X}%
\times \lbrack 0,T].  \label{eq:eq628}
\end{equation}%
From the orthonormality of $e_{i}(\cdot ,s)$ it follows that 
\begin{equation}
\frac{1}{2}\int_{\mathbb{X}\times \lbrack 0,T]}|\psi |^{2}d\nu _{T}=\frac{1}{%
2}\int_{0}^{T}|u(s)|^{2}ds.  \label{eq:eq631}
\end{equation}%
Also 
\begin{align*}
\lbrack A(s)u(s)]_{i}& =\sum_{j=1}^{d}\langle G_{i}(X^{0}(s),\cdot
),e_{j}(\cdot ,s)\rangle _{L^{2}(\nu )}u_{j}(s) \\
& =\left\langle G_{i}(X^{0}(s),\cdot ),\sum_{j=1}^{d}e_{j}(\cdot
,s)u_{j}(s)\right\rangle _{L^{2}(\nu )} \\
& =\langle G_{i}(X^{0}(s),\cdot ),\psi (\cdot ,s)\rangle _{L^{2}(\nu )},
\end{align*}%
so that $A(s)u(s)=\int \psi (y,s)G(X^{0}(s),y)\nu (dy)ds$. Consequently $%
\eta $ satisfies (\ref{eq:finite_dim_limit}) with $\psi $ as in %
\eqref{eq:eq628}. Combining this with \eqref{eq:eq631} we have $\bar{I}(\eta
)\leq I(\eta )+\delta $. Since $\delta >0$ is arbitrary we have $\bar{I}%
(\eta )\leq I(\eta )$.

Conversely, suppose $\psi\in L^{2}(\nu_{T})$ is such that 
\begin{equation*}
\frac{1}{2}\int_{\mathbb{X}\times\lbrack0,T]}|\psi|^{2}d\nu_{T}\leq\bar {I}%
(\eta)+\delta
\end{equation*}
and (\ref{eq:finite_dim_limit}) holds. For $i=1,\ldots,d$ define $%
u_{i}:[0,T]\rightarrow\mathbb{R}$ by 
\begin{equation*}
u_{i}(s)=\langle \psi(\cdot,s),e_{i}(\cdot,s)\rangle_{L^{2}(\nu)}.
\end{equation*}
Note that with $u=(u_{1},\ldots,u_{d})$, 
\begin{align}
\frac{1}{2}\int_{0}^{T}|u(s)|^{2}ds & =\frac{1}{2}\int_{0}^{T}\sum_{j=1}^{d}%
\langle \psi(\cdot,s),e_{j}(\cdot,s)\rangle_{L^{2}(\nu)}^{2}  \notag \\
& \leq\frac{1}{2}\int_{0}^{T}\int_{\mathbb{X}}\psi^{2}(y,s)\nu (dy)ds  \notag
\\
& \leq\bar{I}(\eta)+\delta.  \label{eq:eq646}
\end{align}
For $s\in\lbrack0,T]$, let $\{e_{j}(\cdot,s)\}_{j=d+1}^{\infty}$ be defined
in such a manner that $\{e_{j}(\cdot,s)\}_{j=1}^{\infty}$ is a complete
orthonormal system in $L^{2}(\nu)$. Then, for every $s\in\lbrack0,T]$ 
\begin{align*}
\lbrack A(s)u(s)]_{i} & =\sum_{j=1}^{d}\langle G_{i}(X^{0}(s),\cdot
),e_{j}(\cdot,s)\rangle_{L^{2}(\nu)}\langle \psi(\cdot,s),e_{j}(\cdot
,s)\rangle_{L^{2}(\nu)} \\
& =\sum_{j=1}^{\infty}\langle G_{i}(X^{0}(s),\cdot),e_{j}(\cdot
,s)\rangle_{L^{2}(\nu)}\langle
\psi(\cdot,s),e_{j}(\cdot,s)\rangle_{L^{2}(\nu )} \\
& =\langle G_{i}(X^{0}(s),\cdot),\psi(\cdot,s)\rangle_{L^{2}(\nu)}
\end{align*}
where the second equality follows on observing that $G_{i}(X^{0}(s),\cdot)$
is in the linear span of $\{e_{j}(\cdot,s)\}_{j=1}^{d}$ for every $%
i=1,\ldots,d$. So $A(s)u(s)=\int\psi(y,s)G(X^{0}(s),y)\nu(dy)ds$ and
therefore $(\eta,u)$ satisfy (\ref{eq:finite_dim_limit_u}). Combining this
with \eqref{eq:eq646} we get $I(\eta)\leq\bar{I}(\eta)+\delta$. Since $%
\delta>0$ is arbitrary, $I(\eta)\leq\bar{I}(\eta)$ which completes the proof.

\section{Proofs for the Infinite Dimensional Problem (Theorem \protect\ref%
{LDPpoi_inf})}

\label{sec:sec748} From Theorem \ref{Thm: strsol} there is a measurable map $%
\mathcal{G}^{\varepsilon }:\mathbb{M}\rightarrow D([0,T]:\Phi _{-q})$ such
that $Y^{\varepsilon }=\mathcal{G}^{\varepsilon }(\varepsilon N^{\varepsilon
^{-1}})$. Also, in Theorem \ref{etaeqexist} below we will show that for
every $\psi \in L^{2}(\nu _{T})$ there is a unique solution of %
\eqref{nueq_nucl}. We denote this unique solution as $\mathcal{G}_{0}(\psi )$%
. In order to prove the theorem it suffices to show that Condition \ref%
{gencond} holds with $\mathcal{G}^{\varepsilon }$ and $\mathcal{G}_{0}$
defined as above.

Recall that Conditions \ref{bGcond_nucl} and \ref{bGcond_nuclNew} involve
numbers $p<q<q_{1}$. We start by proving the unique solvability of %
\eqref{nueq_nucl}.

\begin{theorem}
\label{etaeqexist} Suppose Conditions \ref{bGcond_nucl} and \ref%
{bGcond_nuclNew} hold. Then for every $\psi\in L^{2}(\nu_{T})$, there exists
a unique $\eta_{\psi}\in C([0,T],\Phi_{-q_{1}})$ that solves %
\eqref{nueq_nucl}. Furthermore, for every $M\in(0,\infty)$, $\sup_{\psi\in
B_{2}(M)}\sup_{0\leq t\leq T}\Vert\eta_{\psi}(t)\Vert_{-q}<\infty$.
\end{theorem}

\begin{proof}
Fix $M\in (0,\infty )$, and for $\psi \in B_{2}(M)$, let $\tilde{\eta}_{\psi
}(\cdot )=\int_{\mathbb{X}\times \lbrack 0,\cdot ]}G(X^{0}(s),y)\psi
(y,s)\nu (dy)ds$. By an application of the Cauchy-Schwarz inequality,
Condition \ref{bGcond_nucl}(b) and the definition of $m_{T}$ in %
\eqref{eq:eq336}, we see that for every such $\psi $ and $0\leq s\leq t\leq
T $ 
\begin{align}
\left\Vert \tilde{\eta}_{\psi }(t)-\tilde{\eta}_{\psi }(s)\right\Vert _{-p}&
=\left\Vert \int_{\mathbb{X}\times \lbrack s,t]}G(X^{0}(r),y)\psi (y,r)\nu
(dy)dr\right\Vert _{-p}  \label{eq:eq624} \\
& \leq (t-s)^{1/2}M(1+m_{T})\left( \int_{\mathbb{X}}M_{G}^{2}(y)\nu
(dy)\right) ^{1/2}  \notag \\
& =(t-s)^{1/2}m_{T}^{1},  \notag
\end{align}%
where%
\begin{equation}
m_{T}^{1}\doteq M(1+m_{T})\left( \int_{\mathbb{X}}M_{G}^{2}(y)\nu
(dy)\right) ^{1/2}.  \label{eq:def_msup1}
\end{equation}%
This shows that $\tilde{\eta}_{\psi }$ is in $C([0,T]:\Phi _{-p})$.
Henceforth we suppress $\psi $ from the notation unless needed. With $A_{v}$
as in Condition \ref{bGcond_nuclNew}, define $\tilde{b}:[0,T]\times \Phi
_{-q}\rightarrow \Phi _{-q_{1}}$ by 
\begin{equation}
\tilde{b}(s,v)\doteq A_{X^{0}(s)}(v+\tilde{\eta}(s))+\int_{\mathbb{X}%
}D_{x}\left( G(X^{0}(s),y)[\cdot ]\right) [v+\tilde{\eta}(s)]\nu
(dy),\;(s,v)\in \lbrack 0,T]\times \Phi _{-q}.  \label{eq:eq630}
\end{equation}%
The right side in \eqref{eq:eq630} indeed defines an element in $\Phi
_{-q_{1}}$ as is seen from the definition of $A_{v}$ and the estimate in %
\eqref{eq:225a}. Note that $\eta $ solves \eqref{nueq_nucl} if and only if $%
\bar{\eta}=\eta -\tilde{\eta}$ solves the equation 
\begin{equation}
\bar{\eta}(t)=\int_{0}^{t}\tilde{b}(s,\bar{\eta}(s))ds.  \label{eq:eq809}
\end{equation}%
We now argue that (\ref{eq:eq809}) has a unique solution, namely there is a
unique $\bar{\eta}\in C([0,T]:\Phi _{-q_{1}})$ such that for all $\phi \in
\Phi $ 
\begin{equation*}
\bar{\eta}(t)[\phi ]=\int_{0}^{t}\tilde{b}(s,\bar{\eta}(s))[\phi ]ds,\;t\in
\lbrack 0,T].
\end{equation*}%
For this, in view of Theorem 3.7 in \cite{BCD}, it suffices to check that
for some $K<\infty $, $\tilde{b}$ satisfies the following properties.

\begin{description}
\item {(a)} For all $t\in\lbrack0,T]$ and $u\in\Phi_{-q}$, $\tilde{b}%
(t,u)\in\Phi_{-q_{1}}$, and the map $u\mapsto\tilde{b}(t,u)$ is continuous.

\item {(b)} For all $t\in\lbrack0,T]$, and $\phi\in\Phi$, $2\tilde
b(t,\phi)[\theta_{q}\phi]\leq K(1+\|\phi\|_{-q}^{2}). $

\item {(c)} For all $t\in\lbrack0,T]$, and $u\in\Phi_{-q}$, $\|\tilde
b(t,u)\|_{-q_{1}}^{2}\leq K(1+\|u\|_{-q}^{2}). $

\item {(d)} For all $t\in\lbrack0,T]$, and $u_{1},u_{2}\in\Phi_{-q}$, 
\begin{equation}  \label{eq:5.4a}
2\langle \tilde b(t,u_{1})-\tilde b(t,u_{2}),u_{1}-u_{2}\rangle_{-q_{1}}
\leq K\|u_{1}-u_{2}\|_{-q_{1}}^{2}.
\end{equation}
\end{description}

Consider first part (a). For $s\in \lbrack 0,T]$ and $\eta \in \Phi _{-q}$,
define $\tilde{A}_{X^{0}(s)}(\eta ):\Phi _{q}\rightarrow \mathbb{R}$ by 
\begin{equation*}
\tilde{A}_{X^{0}(s)}(\eta )[\phi ]=\int_{\mathbb{X}}D_{x}\left(
G(X^{0}(s),y)[\phi ]\right) [\eta ]\nu (dy).
\end{equation*}%
Note that $\tilde{b}(s,v)=A_{X^{0}(s)}(v+\tilde{\eta}(s))+\tilde{A}%
_{X^{0}(s)}(v+\tilde{\eta}(s))$. Let $K_{1}=\max \{\sqrt{T}m_{T}^{1},1\}$,
with $m_{T}^{1}$ defined in (\ref{eq:def_msup1}). Then using Condition \ref%
{bGcond_nuclNew}(c) and (\ref{eq:eq624}) we have, for each fixed $s\in
\lbrack 0,T]$, that 
\begin{equation}
|\tilde{A}_{X^{0}(s)}(v+\tilde{\eta}(s))[\phi ]|\leq K_{1}M_{DG}^{\ast
}(1+\Vert v\Vert _{-q_{1}})\Vert \phi \Vert _{q_{1}},\;\phi \in \Phi
_{q_{1}},v\in \Phi _{-q}.  \label{eq:eq706}
\end{equation}%
Consequently, for each fixed $s$, $v\mapsto \tilde{A}_{X^{0}(s)}(v+\tilde{%
\eta}(s))$ is a map from $\Phi _{-q}$ to $\Phi _{-q_{1}}$. Also, from
Condition \ref{bGcond_nuclNew}(b) for each fixed $s$, $v\mapsto
A_{X^{0}(s)}(v+\tilde{\eta}(s))$ is a map from $\Phi _{-q}$ to $\Phi
_{-q_{1}}$. By the same condition the map $v\mapsto A_{X^{0}(s)}(v+\tilde{%
\eta}(s))$ is continuous for each $s$. Also, from Condition \ref%
{bGcond_nuclNew}(c) we have for each fixed $s\in \lbrack 0,T]$ and $%
v,v^{\prime }\in \Phi _{-q}$ 
\begin{equation}
|\tilde{A}_{X^{0}(s)}(v+\tilde{\eta}(s))[\phi ]-\tilde{A}_{X^{0}(s)}(v^{%
\prime }+\tilde{\eta}(s))[\phi ]|\leq M_{DG}^{\ast }\Vert v-v^{\prime }\Vert
_{-q_{1}}\Vert \phi \Vert _{q_{1}},\;\phi \in \Phi _{q_{1}}.
\label{eq:eq706b}
\end{equation}%
Consequently the map $v\mapsto \tilde{A}_{X^{0}(s)}(v+\tilde{\eta}(s))$ is
continuous as well. This proves (a).

For (b) note that again using Condition \ref{bGcond_nuclNew}(c) and (\ref%
{eq:eq624}), for $\phi\in\Phi$, 
\begin{align}
\tilde{A}_{X^{0}(s)}(\phi+\tilde{\eta}(s))[\theta_{q}\phi]\leq & \int_{%
\mathbb{X}}M_{DG}(X^{0}(s),y)\Vert\theta_q\phi\Vert_{q}\Vert\phi+\tilde{\eta 
}(s)\Vert_{-q}\nu(dy)  \notag \\
\leq & M_{DG}^{\ast}(m_{T}^{1}\sqrt{T}+\Vert\phi\Vert_{-q})\Vert\phi
\Vert_{-q}.  \label{eq:eq231}
\end{align}
Also, from \eqref{eq: eq429c} 
\begin{equation*}
2A_{X^{0}(s)}(\phi+\tilde{\eta}(s))[\theta_{q}\phi]\leq C_{A}(\Vert\phi
\Vert_{-q}+\sqrt{T}m_{T}^{1})\Vert\phi\Vert_{-q}.
\end{equation*}
Combining this estimate with \eqref{eq:eq231} we have (b).

Consider now part (c). Note that, from \eqref{eq:eq706} we have 
\begin{equation*}
\Vert\tilde{A}_{X^{0}(s)}(v+\tilde{\eta}(s))\Vert_{-q_{1}}\leq
K_{1}M_{DG}^{\ast}(1+\Vert v\Vert_{-q_{1}}),\;v\in\Phi_{-q}.
\end{equation*}
Also, from \eqref{eq: eq429a} and (\ref{eq:eq624}) 
\begin{equation*}
\Vert A_{X^{0}(s)}(v+\tilde{\eta}(s))\Vert_{-q_{1}}\leq M_{A}(1+\sqrt{T}%
m_{1}^{T}+\Vert v\Vert_{-q}),\;v\in\Phi_{-q}.
\end{equation*}
Combining the last two estimates we have 
\begin{equation}
\Vert\tilde{b}(t,v)\Vert_{-q_{1}}\leq(K_{1}M_{DG}^{\ast}+M_{A})(1+\sqrt {T}%
m_{1}^{T}+\Vert v\Vert_{-q}),\;v\in\Phi_{-q}  \label{eq:eq320O5}
\end{equation}
which verifies part (c).

Finally, for (d) note that from \eqref{eq: eq429b}, for all $%
u_{1},u_{2}\in\Phi_{-q}$ and $s\in\lbrack0,T]$ 
\begin{equation}  \label{eq:5.8a}
\langle u_{1}-u_{2},A_{X^{0}(s)}(u_{1}+\tilde{\eta}(s))-A_{X^{0}(s)}(u_{2}+%
\tilde{\eta}(s))\rangle_{-q_{1}}\leq L_{A}\Vert
u_{1}-u_{2}\Vert_{-q_{1}}^{2}.
\end{equation}
Also, from \eqref{eq:eq706b}, for all $u_{1},u_{2}\in\Phi_{-q}$ and $%
s\in\lbrack0,T]$ 
\begin{equation}
\langle u_{1}-u_{2},\tilde{A}_{X^{0}(s)}(u_{1}+\tilde{\eta}(s))-\tilde {A}%
_{X^{0}(s)}(u_{2}+\tilde{\eta}(s))\rangle_{-q_{1}}\leq M_{DG}^{\ast}\Vert
u_{1}-u_{2}\Vert_{-q_{1}}^{2}.  \label{eq:eq325}
\end{equation}
Part (d) now follows on combining these two displays.

As noted earlier, we now have from Theorem 3.7 in \cite{BCD} that %
\eqref{eq:eq809} and therefore \eqref{limpt_poi_inf} has a unique solution
in $C([0,T], \Phi_{-q_{1}})$. Also, from the same theorem it follows that 
\begin{equation*}
\sup_{\psi \in B_2(M)} \sup_{0 \leq t \leq T} \|\bar\eta(t)\|_{-q} < \infty.
\end{equation*}
The second part of the theorem is now immediate on noting that 
\begin{equation*}
\sup_{\psi\in B_{2}(M)} \sup_{0 \leq t \leq T}\|\eta_{\psi}(t)\|_{-q} \leq%
\sqrt{T} m_{1}^{T} + \sup_{\psi \in B_2(M)} \sup_{0 \leq t \leq
T}\|\bar\eta(t)\|_{-q}.
\end{equation*}
\end{proof}

The following lemma verifies part (a) of Condition \ref{gencond}.

\begin{lemma}
\label{lem:lem636inf} Suppose that Conditions \ref{bGcond_nucl} and \ref%
{bGcond_nuclNew} hold. Fix $M\in(0,\infty)$ and $g^{\varepsilon},g\in
B_{2}(M)$ such that $g^{\varepsilon}\rightarrow g$. Let $\mathcal{G}_{0}$ be
the mapping that was shown to be well defined in Theorem \ref{etaeqexist}.
Then $\mathcal{G}_0(g^{\varepsilon})\rightarrow\mathcal{G}_{0}(g)$.
\end{lemma}

\begin{proof}
From \eqref{eq:eq624} and the compact embedding of $\Phi _{-p}$ into $\Phi
_{-q}$ we see that the collection 
\begin{equation*}
\left\{ \tilde{\eta}_{\varepsilon }(\cdot )=\int_{\mathbb{X}\times \lbrack
0,\cdot ]}G(X^{0}(r),y)g^{\varepsilon }(y,r)\nu (dy)dr\right\} _{\varepsilon
>0}
\end{equation*}%
is precompact in $C([0,T]:\Phi _{-q})$. Combining this with the convergence $%
g^{\varepsilon }\rightarrow g$ and the fact that $(s,y)\mapsto
G(X^{0}(s),y)[\phi ]$ is in $L^{2}(\nu _{T})$ for every $\phi \in \Phi $, we
see that 
\begin{equation}
\tilde{\eta}_{\varepsilon }\rightarrow \tilde{\eta}\mbox{ as }\varepsilon
\rightarrow 0\mbox{ in }C([0,T]:\Phi _{-q})  \label{eq:eq721}
\end{equation}%
where $\tilde{\eta}=\int_{\mathbb{X}\times \lbrack 0,\cdot
]}G(X^{0}(r),y)g(y,r)\nu (dy)dr$. Next, let $\eta _{\varepsilon }$ denote
the unique solution of \eqref{nueq_nucl} with $\psi $ replaced by $%
g^{\varepsilon }$ and, as in the proof of Theorem \ref{etaeqexist}, define $%
\bar{\eta}_{\varepsilon }=\eta _{\varepsilon }-\tilde{\eta}_{\varepsilon }$.
From Theorem \ref{etaeqexist} 
\begin{equation}
M_{\bar{\eta}}=\sup_{\varepsilon >0}\sup_{0\leq t\leq T}\Vert \bar{\eta}%
_{\varepsilon }(t)\Vert _{-q}<\infty .  \label{eq:eq705}
\end{equation}%
Also, for every fixed $\phi \in \Phi $, $\bar{\eta}_{\varepsilon }$ solves 
\begin{equation}
\bar{\eta}_{\varepsilon }(t)[\phi ]=\int_{0}^{t}\tilde{b}_{\varepsilon }(s,%
\bar{\eta}_{\varepsilon }(s))[\phi ]ds,  \label{eq:eq707}
\end{equation}%
where $\tilde{b}_{\varepsilon }$ is defined by the right side of %
\eqref{eq:eq630} by replacing $\tilde{\eta}$ with $\tilde{\eta}_{\varepsilon
}$.

Next, let $\bar{\eta}\in C([0,T]:\Phi _{-q_{1}})$ be the unique solution of 
\begin{equation*}
\bar{\eta}(t)[\phi ]=\int_{0}^{t}\tilde{b}(s,\bar{\eta}(s))[\phi ]ds,\;\phi
\in \Phi ,
\end{equation*}%
where $\tilde{b}$ is as in \eqref{eq:eq630}. Let $\hat{A}_{v}=A_{v}+\tilde{A}%
_{v}$ and $a_{\varepsilon }(s)=\hat{A}_{X^{0}(s)}(\bar{\eta}_{\varepsilon
}(s)+\tilde{\eta}_{\varepsilon }(s))-\hat{A}_{X^{0}(s)}(\bar{\eta}(s)+\tilde{%
\eta}(s))$. Using the same bounds as those used in (\ref{eq:eq320O5}), %
\eqref{eq:5.8a} and \eqref{eq:eq325}, there is $K<\infty $ such that%
\begin{align*}
\Vert \bar{\eta}_{\varepsilon }(t)-\bar{\eta}(t)\Vert _{-q_{1}}^{2}&
=2\int_{0}^{t}\langle \tilde{b}_{\varepsilon }(s,\bar{\eta}_{\varepsilon
}(s))-\tilde{b}(s,\bar{\eta}(s)),\bar{\eta}_{\varepsilon }(s)-\bar{\eta}%
(s)\rangle _{-q_{1}}ds \\
& =2\int_{0}^{t}\langle a_{\varepsilon }(s),\bar{\eta}_{\varepsilon }(s)-%
\bar{\eta}(s)\rangle _{-q_{1}}ds \\
& =2\int_{0}^{t}\langle a_{\varepsilon }(s),(\bar{\eta}_{\varepsilon }(s)+%
\tilde{\eta}_{\varepsilon }(s))-(\bar{\eta}(s)+\tilde{\eta}(s))\rangle
_{-q_{1}}ds+2\int_{0}^{t}\langle a_{\varepsilon }(s),\tilde{\eta}(s)-\tilde{%
\eta}_{\varepsilon }(s)\rangle _{-q_{1}}ds \\
& \leq K\int_{0}^{t}\Vert (\bar{\eta}_{\varepsilon }(s)+\tilde{\eta}%
_{\varepsilon }(s)-(\bar{\eta}(s)+\tilde{\eta}(s))\Vert
_{-q_{1}}^{2}ds+K_{2}\int_{0}^{t}\Vert \tilde{\eta}(s)-\tilde{\eta}%
_{\varepsilon }(s)\Vert _{-q_{1}}ds,
\end{align*}%
where $K_{2}=2(K_{1}M_{DG}^{\ast }+M_{A})(1+\sqrt{T}m_{T}^{1}+M_{\bar{\eta}%
}) $ and $M_{\bar{\eta}}$ is from (\ref{eq:eq705}). Thus 
\begin{equation*}
\Vert \bar{\eta}_{\varepsilon }(t)-\bar{\eta}(t)\Vert _{-q_{1}}^{2}\leq
2K\int_{0}^{t}\Vert \bar{\eta}_{\varepsilon }(s)-\bar{\eta}(s)\Vert
_{-q_{1}}^{2}ds+K_{3}\int_{0}^{t}(\Vert \tilde{\eta}_{\varepsilon }(s)-%
\tilde{\eta}(s)\Vert _{-q_{1}}^{2}+\Vert \tilde{\eta}_{\varepsilon }(s)-%
\tilde{\eta}(s)\Vert _{-q_{1}})ds,
\end{equation*}%
where $K_{3}=K_{2}+2K$. The result now follows on combining this with %
\eqref{eq:eq721} and $q<q_{1}$, and using Gronwall's lemma.
\end{proof}

We now consider part (b) of Condition \ref{gencond}. For $\varphi \in 
\mathcal{U}_{+,\varepsilon }^{M}$ let $\bar{X}^{\varepsilon ,\varphi }=\bar{%
\mathcal{G}}^{\varepsilon }(\varepsilon N^{\varepsilon ^{-1}\varphi })$. As
in Section \ref{sect:pfs_finite_dim} it follows by an application of
Girsanov's theorem that $\bar{X}^{\varepsilon ,\varphi }$ is the unique
solution of the integral equation 
\begin{equation}
\bar{X}_{t}^{\varepsilon ,\varphi }[\phi ]=x_{0}[\phi ]+\int_{0}^{t}b(\bar{X}%
_{s}^{\varepsilon ,\varphi })[\phi ]ds+\varepsilon \int_{\mathbb{X}\times
\lbrack 0,t]}G(\bar{X}_{s-}^{\varepsilon ,\varphi },y)[\phi ]N^{\varepsilon
^{-1}\varphi }(dy,ds),\,\phi \in \Phi  \label{eq:sdeg2cont}
\end{equation}%
Define $\bar{Y}^{\varepsilon ,\varphi }$ as in (\ref{eq:Y_X_relation}). Then 
$\bar{Y}^{\varepsilon ,\varphi }=\mathcal{G}^{\varepsilon }(\varepsilon
N^{\varepsilon ^{-1}\varphi })$.

The following moment bounds on $\bar{X}^{\varepsilon ,\varphi }$ and $\bar{Y}%
^{\varepsilon ,\varphi }$ will be key. The proof of part (a) is given in
Proposition 3.13 of \cite{BCD}. However, equation (3.33) in \cite{BCD}
contains an error, in view of which we give a corrected proof below. The
idea is to first approximate $\bar{X}^{\varepsilon ,\varphi }$ by a sequence
of finite-dimensional processes $\{\bar{X}^{\varepsilon ,d,\varphi }\}_{d\in 
\mathbb{N}}$ and obtain an analogous equation for the $d$-dimensional
process for every value of $d$. The desired estimate follows by first
obtaining an estimate for the finite dimensional processes that is uniform
in $d$ and then sending $d\rightarrow \infty $.

\begin{lemma}
\label{bdY_poi_inf} Suppose Condition \ref{bGcond_nucl} and \ref%
{bGcond_nuclNew}(d) hold. Fix $M<\infty $. Then there exists an $\varepsilon
_{0}>0$ such that

\begin{enumerate}
\item[(a)] \label{bdX_semi} 
\begin{equation*}
\sup_{\varepsilon \in (0,\varepsilon _{0})}\sup_{\varphi \in \mathcal{U}%
_{+,\varepsilon }^{M}}\mathbb{\bar{E}}\left[ \sup_{0\leq s\leq T}\Vert \bar{X%
}^{\varepsilon ,\varphi }(s)\Vert _{-p}^{2}\right] <\infty .
\end{equation*}

\item[(b)] \label{bdY_semi} 
\begin{equation*}
\sup_{\varepsilon \in (0,\varepsilon _{0})}\sup_{\varphi \in \mathcal{U}%
_{+,\varepsilon }^{M}}\mathbb{\bar{E}}\left[ \sup_{0\leq s\leq T}\Vert \bar{Y%
}^{\varepsilon ,\varphi }(s)\Vert _{-q}^{2}\right] <\infty .
\end{equation*}
\end{enumerate}
\end{lemma}

\begin{proof}
We first prove part \eqref{bdX_semi}. We follow the steps in the proof of
Theorem 6.2.2 of \cite{KaXi95} (see also the proof of Theorem 3.7 in \cite%
{BCD}). Recall that $\{\phi _{j}\}_{j\in \mathbb{N}}$ is a CONS in $\Phi
_{0} $ and a COS in $\Phi _{n},n\in \mathbb{N}$. For $d\in \mathbb{N}$ let $%
\pi _{d}:\Phi _{-p}\rightarrow \mathbb{R}^{d}$ be defined by 
\begin{equation*}
\pi _{d}(u)\doteq (u[\phi _{1}^{p}],\hdots,u[\phi _{d}^{p}])^{\prime
},\;u\in \Phi _{-p}.
\end{equation*}%
Let $x_{0}^{d}=\pi _{d}(x_{0})$. Define $\beta ^{d}:\mathbb{R}%
^{d}\rightarrow \mathbb{R}^{d}$ and $g^{d}:\mathbb{R}^{d}\times \mathbb{X}%
\rightarrow \mathbb{R}^{d}$ by 
\begin{equation*}
\beta ^{d}(x)_{k}=b\left( \sum\nolimits_{j=1}^{d}x_{j}\phi _{j}^{-p}\right)
[\phi _{k}^{p}],\;g^{d}(x,y)_{k}=G\left( \sum\nolimits_{j=1}^{d}x_{j}\phi
_{j}^{-p},y\right) [\phi _{k}^{p}],k=1,\ldots ,d,
\end{equation*}%
where $\eta \lbrack \phi ]$ was defined in (\ref{eqn:defsqbrkt}). Next
define $\gamma ^{d}:\Phi ^{\prime }\rightarrow \Phi ^{\prime }$ by 
\begin{equation*}
\gamma ^{d}u\doteq \sum_{k=1}^{d}u[\phi _{k}^{p}]\phi _{k}^{-p},
\end{equation*}%
and define $b^{d}:\Phi ^{\prime }\rightarrow \Phi ^{\prime }$ and $%
G^{d}:\Phi ^{\prime }\times \mathbb{X}\rightarrow \Phi ^{\prime }$ by 
\begin{equation*}
b^{d}(u)=\gamma ^{d}b(\gamma ^{d}u),\quad G^{d}(u,y)=\gamma ^{d}G(\gamma
^{d}u,y).
\end{equation*}%
It is easy to check that for each $d\in \mathbb{N}$, $b^{d}$ and $G^{d}$
satisfy Condition \ref{bGcond_nucl} [with $%
(M_{b^{d}},M_{G^{d}},C_{b^{d}},L_{b^{d}},L_{G^{d}})$ equal to $%
(M_{b},M_{G},C_{b},L_{b},L_{G})$ for all $d$]; see the proof of Theorem
6.2.2 in \cite{KaXi95}. Also, from Lemma 6.2.2 in \cite{KaXi95} and an
argument based on Girsanov's theorem (as in Section \ref{sect:pfs_finite_dim}%
) it follows that the following integral equation has a unique solution in $%
D([0,T]:\mathbb{R}^{d})$ for all $\varphi \in \mathcal{U}_{+,\varepsilon
}^{M}$: 
\begin{equation*}
\bar{x}^{\varepsilon ,d,\varphi }(s)=x_{0}^{d}+\int_{0}^{t}\beta ^{d}(\bar{x}%
^{\varepsilon ,d,\varphi }(s))ds+\int_{\mathbb{X}\times \lbrack
0,t]}\varepsilon g^{d}(\bar{x}^{\varepsilon ,d,\varphi
}(s-),y)N^{\varepsilon ^{-1}\varphi }(dy,ds).
\end{equation*}%
Let 
\begin{equation*}
\bar{X}^{\varepsilon ,d,\varphi }(t)=\sum_{k=1}^{d}\bar{x}_{k}^{\varepsilon
,d,\varphi }(t)\phi _{k}^{-p},\quad t\in \lbrack 0,T].
\end{equation*}%
Then, with $X_{0}^{d}=\sum_{k=1}^{d}(x_{0}^{d})_{k}\phi _{k}^{-p}$, for all $%
t\in \lbrack 0,T]$ 
\begin{equation*}
\bar{X}^{\varepsilon ,d,\varphi }(t)=X_{0}^{d}+\int_{0}^{t}b^{d}(\bar{X}%
^{\varepsilon ,d,\varphi }(s))ds+\int_{\mathbb{X}\times \lbrack
0,t]}\varepsilon G^{d}(\bar{X}^{\varepsilon ,d,\varphi
}(s-),y)N^{\varepsilon ^{-1}\varphi }(dy,ds).
\end{equation*}

We next prove that there exists $\varepsilon_{0}>0$ such that 
\begin{equation*}
\sup_{d\in\mathbb{N}}\sup_{\varepsilon\in(0,\varepsilon_{0})}\mathbb{\bar {E}%
}\left[ \sup_{0\leq t\leq T}\Vert\bar{X}^{\varepsilon,d,\varphi}(s)\Vert
_{-p}^{2}\right] <\infty.
\end{equation*}
The proof is similar to Lemma 6.2.2 in \cite{KaXi95} (see also the proof of 
\cite[Proposition 3.13]{BCD}), and therefore we just outline the main steps.
By It\^{o}'s lemma 
\begin{align}
& \Vert\bar{X}^{\varepsilon,d,\varphi}(t)\Vert_{-p}^{2}  \notag \\
& =\Vert X^{0}(t)\Vert_{- p}^{2}+2\int_{0}^{t}b^{d}(\bar{X}%
^{\varepsilon,d,\varphi }(s))[\theta_{p}\bar{X}^{\varepsilon,d,\varphi}(s)]ds
\notag \\
& \hspace{0.3cm}+2\int_{\mathbb{X}\times\lbrack0,t]}\left\langle \bar {X}%
^{\varepsilon,d,\varphi}(s),G^{d}(\bar{X}^{\varepsilon,d,\varphi
}(s),y)\right\rangle _{-p}\varphi d\nu_{T}+\int_{\mathbb{X}\times\lbrack
0,t]}\varepsilon\Vert G^{d}(\bar{X}^{\varepsilon,d,\varphi}(s),y)\Vert
_{-p}^{2}\varphi d\nu_{T}  \notag \\
& \hspace{0.3cm}+\int_{\mathbb{X}\times\lbrack0,t]}2\left[ \left\langle \bar{%
X}^{\varepsilon,d,\varphi}(s-),\varepsilon G^{d}(\bar{X}^{\varepsilon
,d,\varphi}(s-),y)\right\rangle _{-p}+\Vert\varepsilon G^{d}(\bar {X}%
^{\varepsilon,d,\varphi}(s-),y)\Vert_{-p}^{2}\right] d\tilde {N}%
^{\varepsilon^{-1}\varphi}.  \label{eq:eq133}
\end{align}
Recalling that Condition \ref{bGcond_nucl}(c) holds with $b=b^{d}$ (with the
same constant $C_{b}$ for all $d$), we have 
\begin{equation*}
2\int_{0}^{t}\left\langle \bar{X}^{\varepsilon,d,\varphi}(s),b^{d}(\bar {X}%
^{\varepsilon,d,\varphi}(s))\right\rangle _{-p}ds\leq
C_{b}\int_{0}^{t}(1+\Vert\bar{X}^{\varepsilon,d,\varphi}(s)\Vert_{-p}^{2})ds.
\end{equation*}
Now exactly as in \cite[Proposition 3.13]{BCD} it follows that there exists $%
L_{1}\in(0,\infty)$ such that for all $d\in\mathbb{N}$, $\varepsilon\in(0,1)$
and $\varphi\in\mathcal{U}_{+,\varepsilon}^{M}$ 
\begin{equation}
\sup_{0\leq s\leq T}\Vert\bar{X}^{\varepsilon,d,\varphi}(s)\Vert_{-
p}^{2}\leq L_{1}\left( 1+\sup_{0\leq s\leq T}|M^{d}(s)|\right) ,
\label{eq:eq139}
\end{equation}
where $M^{d}(t)$ is the last term on the right side of \eqref{eq:eq133} (see
(3.35) in \cite{BCD}). Once more, exactly as in \cite{BCD} (see (3.36) and
(3.37) therein) one has that there is a $L_{2}\in(0,\infty)$ such that for
all $d\in\mathbb{N}$, $\varepsilon\in(0,1)$ and $\varphi\in\mathcal{U}%
_{+,\varepsilon}^{M}$ 
\begin{equation*}
\mathbb{\bar{E}}\sup_{0\leq s\leq T}|M^{d}(s)|\leq\varepsilon L_{2}\left( 1+%
\mathbb{\bar{E}}\sup_{0\leq t\leq T}\Vert\bar{X}^{\varepsilon,d,\varphi
}(t)\Vert_{-p}^{2}\right) +\frac{1}{8}\mathbb{\bar{E}}\sup_{0\leq t\leq
T}\Vert\bar{X}^{\varepsilon,d,\varphi}(t))\Vert_{-p}^{2}.
\end{equation*}
Using the last estimate in \eqref{eq:eq139} we now have that, for some $%
\varepsilon_{0}>0$, 
\begin{equation}
\sup_{d\in\mathbb{N}}\sup_{\varepsilon\in(0,\varepsilon_{0})}\sup_{\varphi
\in\mathcal{U}_{+,\varepsilon}^{M}}\mathbb{\bar{E}}\sup_{0\leq t\leq T}\Vert%
\bar{X}^{\varepsilon,d,\varphi}(t)\Vert_{- p}^{2}<\infty.  \label{eq:eq205}
\end{equation}
Also an application of Girsanov's theorem and Theorem 6.1.2 of \cite{KaXi95}
shows that $\bar{X}^{\varepsilon,d,\varphi}$ converges in distribution, in $%
D([0,T]:\Phi_{-q})$ to the solution of \eqref{eq:sdeg2cont}. The estimate in
part \eqref{bdX_semi} of the lemma now follows from \eqref{eq:eq205} and an
application of Fatou's lemma.

We now prove part \eqref{bdY_semi} of the lemma. By It\^{o}'s formula,

\begin{align*}
\Vert \bar{X}^{\varepsilon ,\varphi }(t)-X^{0}(t)\Vert _{-q}^{2}&
=2\int_{0}^{t}\langle \bar{X}^{\varepsilon ,\varphi }(s)-X^{0}(s),b(\bar{X}%
^{\varepsilon ,\varphi }(s))-b(X^{0}(s))\rangle _{-q}ds \\
& \mbox{}+2\int_{\mathbb{X}\times \lbrack 0,t]}\langle \bar{X}^{\varepsilon
,\varphi }(s)-X^{0}(s),G(\bar{X}^{\varepsilon ,\varphi
}(s),y)-G(X^{0}(s),y))\rangle _{-q}\nu (dy)ds \\
& \mbox{}+2\int_{\mathbb{X}\times \lbrack 0,t]}\langle \bar{X}^{\varepsilon
,\varphi }(s)-X^{0}(s),G(\bar{X}^{\varepsilon ,\varphi }(s),y)\rangle
_{-q}(\varphi -1)\nu (dy)ds \\
& \mbox{}+\int_{\mathbb{X}\times \lbrack 0,t]}\varepsilon \Vert G(\bar{X}%
^{\varepsilon ,\varphi }(s),y)\Vert _{-q}^{2}\varphi \nu (dy)ds \\
& \mbox{}+\int_{\mathbb{X}\times \lbrack 0,t]}\Big(2\langle \bar{X}%
^{\varepsilon ,\varphi }(s-)-X^{0}(s-),\varepsilon G(\bar{X}^{\varepsilon
,\varphi }(s-),y)\rangle _{-q} \\
& \mbox{}\hspace{2cm}+\Vert \varepsilon G(\bar{X}^{\varepsilon ,\varphi
}(s-),y)\Vert _{-q}^{2}\Big)\ \tilde{N}^{\varepsilon ^{-1}\varphi }(dy,ds) \\
& =a^{2}(\varepsilon )\left( A^{\varepsilon ,\varphi }+B^{\varepsilon
,\varphi }+C^{\varepsilon ,\varphi }+\mathcal{E}_{1}^{\varepsilon ,\varphi
}+M_{1}^{\varepsilon ,\varphi }+M_{2}^{\varepsilon ,\varphi }\right) .
\end{align*}%
By Condition \ref{bGcond_nucl}(d), for all $t\in \lbrack 0,T]$ 
\begin{equation*}
\sup_{0\leq r\leq t}A^{\varepsilon ,\varphi }(r)\leq 2L_{b}\int_{0}^{t}\Vert 
\bar{Y}^{\varepsilon ,\varphi }(s)\Vert _{-q}^{2}ds.
\end{equation*}%
Also by Condition \ref{bGcond_nucl}(e) 
\begin{equation*}
\sup_{0\leq r\leq t}|B^{\varepsilon ,\varphi }(r)|\leq \Vert L_{G}\Vert
_{1}\int_{0}^{t}\Vert \bar{Y}^{\varepsilon ,\varphi }(s)\Vert _{-q}^{2}ds.
\end{equation*}%
Next, note that with $\psi =(\varphi -1)/a(\varepsilon )$ 
\begin{align*}
\sup_{0\leq r\leq t}|C^{\varepsilon ,\varphi }(r)|& \leq 2\int_{\mathbb{X}%
\times \lbrack 0,t]}|\langle \bar{Y}^{\varepsilon ,\varphi }(s),G(\bar{X}%
^{\varepsilon ,\varphi }(s),y)\rangle _{-q}|\,|\psi |d\nu _{T} \\
& \leq 2\int_{\mathbb{X}\times \lbrack 0,t]}\Vert \bar{Y}^{\varepsilon
,\varphi }(s)\Vert _{-q}\Vert G(\bar{X}^{\varepsilon ,\varphi }(s),y)\Vert
_{-q}|\psi |d\nu _{T} \\
& \leq 2\int_{\mathbb{X}\times \lbrack 0,t]}\Vert \bar{Y}^{\varepsilon
,\varphi }(s)\Vert _{-q}\left[ \Vert G(\bar{X}^{0}(s),y)\Vert
_{-q}+L_{G}(y)a(\varepsilon )\Vert \bar{Y}^{\varepsilon ,\varphi }(s)\Vert
_{-q}\right] |\psi |d\nu _{T} \\
& \leq 2\int_{\mathbb{X}\times \lbrack 0,t]}\Vert \bar{Y}^{\varepsilon
,\varphi }(s)\Vert _{-q}R_{G}(y)|\psi |\left( 1+a(\varepsilon )\Vert \bar{Y}%
^{\varepsilon ,\varphi }(s)\Vert _{-q}\right) d\nu _{T},
\end{align*}%
where for $y\in \mathbb{X}$, $R_{G}(y)=M_{G}(y)(1+m_{T})+L_{G}(y)$. Thus 
\begin{align}
\sup_{r\leq t}|C^{\varepsilon ,\varphi }(r)|& \leq 2a(\varepsilon
)\sup_{r\leq t}\Vert \bar{Y}^{\varepsilon ,\varphi }(r)\Vert _{-q}^{2}\int_{%
\mathbb{X}\times \lbrack 0,t]}R_{G}(y)|\psi |d\nu _{T}+2\int_{\mathbb{X}%
\times \lbrack 0,t]}\Vert \bar{Y}^{\varepsilon ,\varphi }(s)\Vert
_{-q}R_{G}(y)|\psi |d\nu _{T}  \notag \\
& =T_{1}+T_{2}.  \label{eq:eq311}
\end{align}%
Consider now $T_{2}$. Note that $R_{G}\in L^{2}(\nu )\cap \mathcal{H}$. We
can therefore apply Lemma \ref{Ineq2} with $h$ replaced by $R_{G}$. For any $%
\beta <\infty $ 
\begin{align}
T_{2}& =2\int_{\mathbb{X}\times \lbrack 0,t]}\Vert \bar{Y}^{\varepsilon
,\varphi }(s)\Vert _{-q}R_{G}(y)|\psi |\left[ 1_{\{|\psi |\leq \beta
/a(\varepsilon )\}}+1_{\{|\psi |>\beta /a(\varepsilon )\}}\right] d\nu _{T} 
\notag \\
& \leq 2\sup_{r\leq t}\Vert \bar{Y}^{\varepsilon ,\varphi }(r)\Vert
_{-q}\vartheta (\beta )(1+\sqrt{T})+\int_{\mathbb{X}\times \lbrack 0,t]}%
\left[ \Vert \bar{Y}^{\varepsilon ,\varphi }(r)\Vert
_{-q}^{2}R_{G}^{2}(y)+|\psi |^{2}1_{\{|\psi |\leq \beta /a(\varepsilon )\}}%
\right] d\nu _{T}  \notag \\
& \leq 2\sup_{r\leq t}\Vert \bar{Y}^{\varepsilon ,\varphi }(r)\Vert
_{-q}\vartheta (\beta )(1+\sqrt{T})+L_{1}\int_{[0,t]}\Vert \bar{Y}%
^{\varepsilon ,\varphi }(s)\Vert _{-q}^{2}ds+M\kappa _{2}(\beta ),
\label{eq:eq321}
\end{align}%
where $L_{1}=\int_{\mathbb{X}}R_{G}^{2}(y)\nu (dy)$ and in the last
inequality we have used Lemma \ref{sineqs2}(c). Once again from Lemma \ref%
{Ineq2} 
\begin{equation*}
L_{2}=\sup_{\varepsilon \in (0,1)}\sup_{\psi \in \mathcal{S}_{\varepsilon
}^{M}}2\int_{\mathbb{X}\times \lbrack 0,T]}R_{G}(y)|\psi |d\nu _{T}<\infty .
\end{equation*}%
Using $a\leq 1+a^{2}$ and the last two estimates in \eqref{eq:eq311}, we
have that 
\begin{equation}
\sup_{r\leq t}|C^{\varepsilon ,\varphi }(r)|\leq L(\beta )+\sup_{r\leq
t}\Vert \bar{Y}^{\varepsilon ,\varphi }(r)\Vert _{-q}^{2}(L_{2}a(\varepsilon
)+\tilde{\vartheta}(\beta ))+L_{1}\int_{[0,t]}\Vert \bar{Y}^{\varepsilon
,\varphi }(s)\Vert _{-q}^{2}ds,  \label{eq:eq329}
\end{equation}%
where $\tilde{\vartheta}(\beta )=2\vartheta (\beta )(1+\sqrt{T})$ and $%
L(\beta )=M\kappa _{2}(\beta )+\tilde{\vartheta}(\beta )$.

Next note that 
\begin{align*}
\sup_{r\leq t}\mathcal{E}_{1}^{\varepsilon ,\varphi }(r)& \leq \frac{%
\varepsilon }{a^{2}(\varepsilon )}\int_{\mathbb{X}\times \lbrack
0,t]}(1+\Vert \bar{X}^{\varepsilon ,\varphi }(s)\Vert
_{-p})^{2}M_{G}^{2}(y)\varphi (y,s)d\nu _{T} \\
& \leq \frac{\varepsilon }{a^{2}(\varepsilon )}\left( 1+\sup_{s\leq T}\Vert 
\bar{X}^{\varepsilon ,\varphi }(s)\Vert _{-p}\right) ^{2}\int_{\mathbb{X}%
\times \lbrack 0,t]}M_{G}^{2}(y)\,\varphi (y,s)d\nu _{T}.
\end{align*}%
Since $M_{G}\in L^{2}(\nu )\cap \mathcal{H}$, we have from Lemma \ref{Ineq1}
that 
\begin{equation*}
L_{3}=\sup_{\varepsilon \in (0,1)}\sup_{\varphi \in \mathcal{S}%
_{+,\varepsilon }^{M}}\int_{\mathbb{X}\times \lbrack 0,t]}M_{G}^{2}\,\varphi
d\nu _{T}<\infty ,
\end{equation*}%
and consequently for all $\varepsilon \in (0,\varepsilon _{0})$ 
\begin{equation*}
\mathbb{\bar{E}}\left[ \sup_{r\leq t}\mathcal{E}_{1}^{\varepsilon ,\varphi
}(r)\right] \leq L_{4}\frac{\varepsilon }{a^{2}(\varepsilon )},
\end{equation*}%
where $L_{4}=2L_{3}\sup_{\varphi \in \mathcal{U}_{+,\varepsilon }^{M}}(1+%
\mathbb{\bar{E}}\sup_{0\leq s\leq T}\Vert \bar{X}^{\varepsilon ,\varphi
}(s)\Vert _{-p}^{2})<\infty $ by part (a) of the lemma if $\varepsilon
_{0}>0 $ is small enough. Next, an application of Lenglart-Lepingle-Pratelli
inequality (see Lemma 2.4 in \cite{kurpro}) gives that for some $L_{5}\in
(0,\infty )$ 
\begin{align*}
\mathbb{\bar{E}}\left[ \sup_{0\leq s\leq T}M_{1}^{\varepsilon ,\varphi }(s)%
\right] & \leq \frac{L_{5}}{a^{2}(\varepsilon )}\mathbb{\bar{E}}\left[ \int_{%
\mathbb{X}\times \lbrack 0,T]}\left\langle \bar{X}^{\varepsilon ,\varphi
}(s)-X^{0}(s),\varepsilon G(\bar{X}^{\varepsilon ,\varphi
}(s),y)\right\rangle _{-q}^{2}\varepsilon ^{-1}\varphi \nu (dy)ds\right]
^{1/2} \\
& \leq \frac{L_{5}\sqrt{\varepsilon }}{a^{2}(\varepsilon )}\mathbb{\bar{E}}%
\left[ \int_{\mathbb{X}\times \lbrack 0,t]}\Vert \bar{X}^{\varepsilon
,\varphi }(s)-X^{0}(s)\Vert _{-q}^{2}\Vert G(\bar{X}^{\varepsilon ,\varphi
}(s-),y)\Vert _{-q}^{2}\varphi \nu (dy)ds\right] ^{1/2} \\
& \leq \frac{L_{5}\sqrt{\varepsilon }}{a(\varepsilon )}\mathbb{\bar{E}}\left[
\sup_{s\leq t}\Vert \bar{Y}^{\varepsilon ,\varphi }(s)\Vert _{-q}\left(
\int_{\mathbb{X}\times \lbrack 0,t]}\Vert G(\bar{X}^{\varepsilon ,\varphi
}(s),y)\Vert _{-q}^{2}\varphi \nu (dy)ds\right) ^{1/2}\right] \\
& \leq \frac{L_{5}\sqrt{\varepsilon }}{2a(\varepsilon )}\left[ \mathbb{\bar{E%
}}\sup_{s\leq t}\Vert \bar{Y}^{\varepsilon ,\varphi }(s)\Vert _{-q}^{2}+%
\mathbb{\bar{E}}\left( 1+\sup_{s\leq t}\Vert \bar{X}^{\varepsilon ,\varphi
}(s)\Vert _{-p}^{2}\right) \int_{\mathbb{X}\times \lbrack
0,t]}M_{G}(y)^{2}\varphi \nu (dy)ds\right] \\
& \leq \frac{L_{5}\sqrt{\varepsilon }}{2a(\varepsilon )}\mathbb{\bar{E}}%
\sup_{s\leq t}\Vert \bar{Y}^{\varepsilon ,\varphi }(s)\Vert _{-q}^{2}+\frac{%
L_{5}L_{4}\sqrt{\varepsilon }}{4a(\varepsilon )}.
\end{align*}%
Finally, 
\begin{align*}
\mathbb{\bar{E}}\left[ \sup_{0\leq s\leq t}M_{2}^{\varepsilon ,\varphi }(s)%
\right] & \leq \frac{1}{a^{2}(\varepsilon )}\mathbb{\bar{E}}\int_{\mathbb{X}%
\times \lbrack 0,T]}\Vert \varepsilon G(\bar{X}^{\varepsilon ,\varphi
}(s-),y)\Vert _{-q}^{2}{N}^{\varepsilon ^{-1}\varphi }(dy,ds) \\
& \mbox{}\quad +\frac{1}{a^{2}(\varepsilon )}\mathbb{\bar{E}}\int_{\mathbb{X}%
\times \lbrack 0,T]}\varepsilon \Vert G(\bar{X}^{\varepsilon ,\varphi
}(s),y)\Vert _{-q}^{2}\varphi d\nu _{T} \\
& \leq \frac{2\varepsilon }{a^{2}(\varepsilon )}\mathbb{\bar{E}}\int_{%
\mathbb{\ X}\times \lbrack 0,T]}\Vert G(\bar{X}^{\varepsilon ,\varphi
}(s),y)\Vert _{-q}^{2}\varphi d\nu _{T} \\
& \leq \frac{\varepsilon L_{4}}{a^{2}(\varepsilon )}.
\end{align*}%
Let $\varepsilon _{1}\in (0,\varepsilon _{0})$ be such that for all $%
\varepsilon \in (0,\varepsilon _{1})$, $\max \{\varepsilon ,a(\varepsilon ),%
\frac{\varepsilon }{a^{2}(\varepsilon )}\}<1$. Collecting terms together, we
now have for all $\varepsilon \in (0,\varepsilon _{1})$ 
\begin{align*}
\mathbb{\bar{E}}\left[ \sup_{s\leq t}\Vert \bar{Y}^{\varepsilon ,\varphi
}(s)\Vert _{-q}^{2}\right] & \leq K_{1}\int_{0}^{t}\mathbb{\bar{E}}\left[
\sup_{r\leq s}\Vert \bar{Y}^{\varepsilon ,\varphi }(r)\Vert _{-q}^{2}\right]
ds+\left( L(\beta )+2L_{4}+\frac{L_{5}L_{4}}{4}\right) \\
& \mbox{}+\left[ L_{2}a(\varepsilon )+L_{5}\frac{\sqrt{\varepsilon }}{%
2a(\varepsilon )}+\tilde{\vartheta}(\beta )\right] \mathbb{\bar{E}}\left[
\sup_{s\leq t}\Vert \bar{Y}^{\varepsilon ,\varphi }(s)\Vert _{-q}^{2}\right]
,
\end{align*}%
where $K_{1}=2L_{b}+\Vert L_{G}\Vert _{1}+L_{1}$. Since $\tilde{\vartheta}%
(\beta )\rightarrow 0$ as $\beta \rightarrow \infty $, we can find $\beta
_{0}<\infty $ and $\varepsilon _{2}\in (0,\varepsilon _{1})$ such that for
all $\varepsilon \in (0,\varepsilon _{2})$, $L_{2}a(\varepsilon )+L_{5}\frac{%
\varepsilon }{2a(\varepsilon )}+\tilde{\vartheta}(\beta _{0})\leq 1/2$.
Using this in the above inequality, for all $\varepsilon \in (0,\varepsilon
_{2})$ 
\begin{equation*}
\mathbb{\bar{E}}\left[ \sup_{s\leq t}\Vert \bar{Y}^{\varepsilon ,\varphi
}(s)\Vert _{-q}^{2}\right] \leq K_{2}+\frac{K_{1}}{2}\int_{0}^{t}\mathbb{%
\bar{E}}\left[ \sup_{r\leq s}\Vert \bar{Y}^{\varepsilon ,\varphi }(r)\Vert
_{-q}^{2}\right] ds,
\end{equation*}%
where $K_{2}=\frac{1}{2}(L(\beta _{0})+2L_{4}+\frac{L_{5}L_{4}}{4})$. The
result now follows from Gronwall's inequality.
\end{proof}

The following result will be used in verifying part (b) of Condition \ref%
{gencond}. Recall the integer $q_{1}>q$ introduced in Condition \ref%
{bGcond_nuclNew}.

\begin{lemma}
\label{cont_lim_poi_inf} Suppose Conditions \ref{bGcond_nucl} and \ref%
{bGcond_nuclNew} hold. Let $\varepsilon _{0}>0$ be as in Lemma \ref%
{bdY_poi_inf}, and let $\{\varphi ^{\varepsilon }\}_{\varepsilon \in
(0,\varepsilon _{0})}$ be such that for some $M<\infty $, $\varphi
^{\varepsilon }\in \mathcal{U}_{+,\varepsilon }^{M}$ for every $\varepsilon
\in (0,\varepsilon _{0})$. Let $\psi ^{\varepsilon }=(\varphi ^{\varepsilon
}-1)/a(\varepsilon )$ and fix $\beta \in (0,1]$. Then $\{\left( \bar{Y}%
^{\varepsilon ,\varphi ^{\varepsilon }},\psi ^{\varepsilon }1_{\{|\psi
^{\varepsilon }|\leq \beta /a(\varepsilon )\}}\right) \}_{\varepsilon \in
(0,\varepsilon _{0})}$ is tight in $D([0,T]:\Phi _{-q_{1}})\times
B_{2}((M\kappa _{2}(1))^{1/2}))$ and any limit point $(\eta ,\psi )$ solves %
\eqref{nueq_nucl}.
\end{lemma}

\begin{proof}
In order to prove the tightness of $\{\bar{Y}^{\varepsilon ,\varphi
^{\varepsilon }}\}_{\varepsilon \in (0,\varepsilon _{0})}$ we will apply
Theorem 2.5.2 of \cite{KaXi95}, according to which it suffices to verify
that:\newline
(a) $\{\sup_{0\leq t\leq T}\Vert \bar{Y}^{\varepsilon ,\varphi ^{\varepsilon
}}(t)\Vert _{-q}\}_{\varepsilon \in (0,\varepsilon _{0})}$ is a tight family
of $\mathbb{R}_{+}$-valued random variables,\newline
(b) for every $\phi \in \Phi $, $\{\bar{Y}^{\varepsilon ,\varphi
^{\varepsilon }}[\phi ]\}_{\varepsilon \in (0,\varepsilon _{0})}$ is tight
in $D([0,T]:\mathbb{R})$.

Note that (a) is immediate from Lemma \ref{bdY_poi_inf}(b). Consider now (b).

As in the proof for the finite-dimensional case (see the proof of Lemma \ref%
{bdY_poi}), we write $\bar{Y}^{\varepsilon ,\varphi ^{\varepsilon
}}=M^{\varepsilon ,\varphi ^{\varepsilon }}+A^{\varepsilon ,\varphi
^{\varepsilon }}+B^{\varepsilon ,\varphi ^{\varepsilon }}+\mathcal{E}%
_{1}^{\varepsilon ,\varphi ^{\varepsilon }}+C^{\varepsilon ,\varphi
^{\varepsilon }}$, where the processes on the right side are as defined in %
\eqref{eq:eq722}. Fix $\phi \in \Phi $. Using Condition \ref{bGcond_nucl}
parts (b) and (e), it follows as in the proof of Lemma \ref{bdY_poi} (see (%
\ref{eq:bound_on_M}) and \eqref{bdE}) that 
\begin{equation}
\mathbb{\bar{E}}\left[ \sup_{0\leq s\leq T}|M^{\varepsilon ,\varphi
^{\varepsilon }}(s)[\phi ]|^{2}\right] \rightarrow 0,\;\mathbb{\bar{E}}\left[
\sup_{0\leq s\leq T}|\mathcal{E}_{1}^{\varepsilon ,\varphi ^{\varepsilon
}}(s)[\phi ]|^{2}\right] \rightarrow 0,\;\mbox{ as }\varepsilon \rightarrow
0.  \label{eq:eq706O2}
\end{equation}%
Next, by Taylor's theorem and Condition \ref{bGcond_nuclNew}(c) , 
\begin{equation*}
G(\bar{X}^{\varepsilon ,\varphi ^{\varepsilon }}(s),y)[\phi
]-G(X^{0}(s),y)[\phi ]=a(\varepsilon )D_{x}(G(X^{0}(s),y)[\phi ])\bar{Y}%
^{\varepsilon ,\varphi ^{\varepsilon }}(s)+R^{\varepsilon ,\varphi
^{\varepsilon },\phi }(y,s)
\end{equation*}%
where 
\begin{equation}
|R^{\varepsilon ,\varphi ^{\varepsilon },\phi }(y,t)|\leq L_{DG}(\phi
,y)a^{2}(\varepsilon )\Vert \bar{Y}^{\varepsilon ,\varphi ^{\varepsilon
}}(t)\Vert _{-q}^{2}  \label{eq:Rbound}
\end{equation}%
Hence 
\begin{equation}
B^{\varepsilon ,\varphi ^{\varepsilon }}(t)[\phi ]=\int_{\mathbb{X}\times
\lbrack 0,t]}D_{x}\left( G(X^{0}(s),y)[\phi ]\right) \bar{Y}^{\varepsilon
,\varphi }(s)\nu (dy)ds+\mathcal{E}_{2}^{\varepsilon ,\varphi ^{\varepsilon
},\phi }(t),  \label{eq:Bbound}
\end{equation}%
where, from (\ref{eq:Rbound}), Lemma \ref{bdY_poi_inf}(b) and Condition \ref%
{bGcond_nuclNew}(c), 
\begin{equation}
\mathbb{\bar{E}}\left[ \sup_{0\leq t\leq T}|\mathcal{E}_{2}^{\varepsilon
,\varphi ^{\varepsilon },\phi }(t)|\right] \leq Ta(\varepsilon )\Vert
L_{DG}(\phi ,\cdot )\Vert _{1}\mathbb{\bar{E}}\sup_{0\leq t\leq T}\Vert \bar{%
Y}^{\varepsilon ,\varphi ^{\varepsilon },\phi }(t)\Vert _{-q}^{2}\rightarrow
0\mbox{ as }\varepsilon \rightarrow 0.  \label{eq:Ebound}
\end{equation}%
Similarly, using Condition \ref{bGcond_nuclNew}(a) 
\begin{equation*}
A^{\varepsilon ,\varphi ^{\varepsilon }}(t)[\phi ]=\int_{0}^{t}D\left(
b(X^{0}(s))[\phi ]\right) \bar{Y}^{\varepsilon ,\varphi ^{\varepsilon
}}(s)ds+\mathcal{E}_{3}^{\varepsilon ,\varphi ^{\varepsilon },\phi }(t)
\end{equation*}%
where 
\begin{equation}
\mathbb{\bar{E}}\left[ \sup_{0\leq t\leq T}|\mathcal{E}_{3}^{\varepsilon
,\varphi ^{\varepsilon }}(t)|\right] \rightarrow 0\mbox{ as }\varepsilon
\rightarrow 0.  \label{eq:eq652O2}
\end{equation}%
Combining \eqref{eq:eq706O2}--\eqref{eq:eq652O2}, we have 
\begin{align}
\bar{Y}^{\varepsilon ,\varphi ^{\varepsilon }}(t)[\phi ]& =\mathcal{E}%
^{\varepsilon ,\varphi ^{\varepsilon },\phi
}(t)+\int_{0}^{t}D(b(X^{0}(s))[\phi ])\bar{Y}^{\varepsilon ,\varphi
^{\varepsilon }}(t)ds  \notag \\
& \hspace{0.3cm}+\int_{\mathbb{X}\times \lbrack
0,t]}D_{x}(G(X^{0}(s),y)[\phi ])\bar{Y}^{\varepsilon ,\varphi ^{\varepsilon
}}(s)\ d\nu _{T}+\int_{\mathbb{X}\times \lbrack 0,t]}G(X^{0}(s),y)[\phi
]\psi ^{\varepsilon }\ d\nu _{T}  \notag \\
& \equiv \mathcal{E}^{\varepsilon ,\varphi ^{\varepsilon },\phi
}(t)+A_{1}^{\varepsilon ,\varphi ^{\varepsilon },\phi
}(t)+B_{1}^{\varepsilon ,\varphi ^{\varepsilon },\phi
}(t)+C_{1}^{\varepsilon ,\varphi ^{\varepsilon },\phi }(t)
\label{ybarep_inf}
\end{align}%
where $\mathbb{\bar{E}}[\sup_{0\leq t\leq T}|\mathcal{E}^{\varepsilon
,\varphi ^{\varepsilon },\phi }(t)|]\rightarrow 0$.

Next, from Condition \ref{bGcond_nucl}(b) we have applying Lemma \ref{Ineq2}
with $h=M_{G}$ as in the proof of \eqref{eq:eq724}, that for all $\delta>0$, 
$t\in\lbrack0,T-\delta]$, $\varepsilon>0$, 
\begin{equation*}
|C_{1}^{\varepsilon,\varphi^{\varepsilon}}(t+\delta)[\phi]-C_{1}^{%
\varepsilon,\varphi^{\varepsilon}}(t)[\phi]|\leq(1+m_{T})\Vert\phi\Vert
_{p}(\rho(1)\delta^{1/2}+\vartheta(1)a(\varepsilon)),
\end{equation*}
where $\rho$ and $\vartheta$ are as in Lemma \ref{Ineq2} and $m_{T}$ is as
in \eqref{eq:eq336}. Tightness of $C_{1}^{\varepsilon,\varphi^{%
\varepsilon}}(\cdot)[\phi]$ in $C([0,T]:\mathbb{R})$ is now immediate.

For tightness of ${B}_{1}^{\varepsilon ,\varphi ^{\varepsilon }}(t)[\phi ]$
note that from Condition \ref{bGcond_nuclNew}(c), for all $\delta >0$, $t\in
\lbrack 0,T-\delta ]$, $\varepsilon >0$ 
\begin{align*}
|B_{1}^{\varepsilon ,\varphi ^{\varepsilon }}(t+\delta )[\phi
]-B_{1}^{\varepsilon ,\varphi ^{\varepsilon }}(t)[\phi ]|& \leq \int_{%
\mathbb{X}\times \lbrack t,t+\delta ]}\Vert D_{x}(G(X^{0}(s),y)[\phi ])\Vert
_{op,-q_{1}}\Vert \bar{Y}^{\varepsilon ,\varphi ^{\varepsilon }}(s)\Vert
_{-q}d\nu _{T} \\
& \leq \Vert \phi \Vert _{q_{1}}M_{DG}^{\ast }\delta \sup_{0\leq t\leq
T}\Vert \bar{Y}^{\varepsilon ,\varphi ^{\varepsilon }}(t)\Vert _{-q}.
\end{align*}%
Tightness of ${B}_{1}^{\varepsilon ,\varphi ^{\varepsilon }}(t)[\phi ]$ in $%
C([0,T]:\mathbb{R})$ now follows from Lemma \ref{bdY_poi_inf}(b). A similar
estimate using \eqref{eq: eq429a} shows that ${A}_{1}^{\varepsilon ,\varphi
^{\varepsilon }}(t)[\phi ]$ is tight in $C([0,T]:\mathbb{R})$ as well.
Combining these tightness properties we have from \eqref{ybarep_inf} that $\{%
\bar{Y}^{\varepsilon ,\varphi ^{\varepsilon }}(\cdot )[\phi ]\}_{\varepsilon
>0}$ is tight in $D([0,T]:\mathbb{R})$ for all $\phi \in \Phi $ which proves
part (b) of the tightness criterion stated at the beginning of the proof.
Thus $\{\bar{Y}^{\varepsilon ,\varphi ^{\varepsilon }}(\cdot
)\}_{\varepsilon \in (0,\varepsilon _{0})}$ is tight in $D([0,T]:\Phi
_{-q_{1}})$. Tightness of $\{\psi ^{\varepsilon }1_{\{|\psi ^{\varepsilon
}|\leq \beta /a(\varepsilon )\}}\}_{\varepsilon \in (0,\varepsilon _{0})}$
holds for the same reason as in the proof of Lemma \ref{cont_lim_poi}, i.e.,
because they take values in a compact set.

Suppose now that $(\bar{Y}^{\varepsilon ,\varphi ^{\varepsilon }}(\cdot
),\psi ^{\varepsilon }1_{\{|\psi ^{\varepsilon }|<\beta /a(\varepsilon )\}})$
converges along a subsequence in distribution to $(\eta ,\psi )$. To prove
the result it suffices to show that for all $\phi \in \Phi $ %
\eqref{limpt_poi_inf} is satisfied. From Lemma \ref{ctrllim}, Condition \ref%
{bGcond_nucl}(b) and the tightness of ${C}_{1}^{\varepsilon ,\varphi
^{\varepsilon }}(t)[\phi ]$ shown above it follows that 
\begin{equation*}
\left( \bar{Y}^{\varepsilon ,\varphi ^{\varepsilon }},\int_{\mathbb{X}\times
\lbrack 0,\cdot ]}G(X^{0}(s),y)[\phi ]\psi ^{\varepsilon }d\nu _{T}\right)
\rightarrow \left( \eta ,\int_{\mathbb{X}\times \lbrack 0,\cdot
]}G(X^{0}(s),y)[\phi ]\psi d\nu _{T}\right)
\end{equation*}%
in $D([0,T]:\Phi _{-q_{1}}\times \mathbb{R})$. The result now follows by
using this convergence in \eqref{ybarep_inf}.
\end{proof}

\section{Example}

\label{sec:secexample} The following equation was introduced in \cite{KaXi95}
to model the spread of Poissonian point source chemical agents in the $d$%
-dimensional bounded domain $[0,l]^{d}$. The time instants, sites and
magnitude of chemical injection into the domain are modeled using a Poisson
random measure on $\mathbb{X}\times \lbrack 0,T]$, where $\mathbb{X}%
=[0,l]^{d}\times \mathbb{R}_{+}$, with a finite intensity measure. Formally
the model can be written as follows. Denote by $\tau _{i}^{\varepsilon
}(\omega )$, $i\in \mathbb{N}$, the jump times of the Poisson process with
rate $\varepsilon ^{-1}\nu (\mathbb{X})$, where $\nu $ is a finite measure
on $\mathbb{X}$, and let $(\kappa _{i},A_{i})$ be an iid sequence of $%
\mathbb{X}$-valued random variables with common distribution $\nu
_{0}(dy)=\nu (dy)/\nu (\mathbb{X})$. Let $\zeta >0$ be a small fixed
parameter and let $c_{\zeta }=\int_{\mathbb{R}^{d}}1_{B_{\zeta }(0)}(x)dx$,
where for $y\in \mathbb{R}^{d}$, $B_{\zeta }(y)=\{x\in \mathbb{R}%
^{d}:|y-x|\leq \zeta \}$. Then the model can be described by the following
equation. 
\begin{equation}
\frac{\partial }{\partial t}u(t,x)=D\Delta u(t,x)-V\cdot \nabla
u(t,x)-\alpha u(t,x)+\sum_{i=1}^{\infty }A_{i}(\omega )c_{\zeta
}^{-1}1_{B_{\zeta }(\kappa _{i})}(x)1_{\left\{ t=\tau _{i}(\omega )\right\}
},  \label{eq:model}
\end{equation}%
where for a smooth function $f$ on $\mathbb{R}^{d}$, $\Delta f=\sum_{i=1}^{d}%
\frac{\partial ^{2}f}{\partial x_{i}^{2}}$ and $\nabla f=(\frac{\partial f}{%
\partial x_{1}},\ldots ,\frac{\partial f}{\partial x_{d}})^{\prime }$, $%
\alpha \in (0,\infty )$, $D>0$, $V\in \mathbb{R}^{d}$ and $\varepsilon >0$
is a scaling parameter. The last term on the right side of \eqref{eq:model}
says that at the time instant $\tau _{i}$, $A_{i}$ amount of contaminant is
introduced which is distributed uniformly over a ball of radius $\zeta $ in $%
\mathbb{R}^{d}$ centered at $\kappa _{i}$, where for simplicity we assume
that $\kappa _{i}$ a.s. takes values in the $\zeta $-interior of $[0,l]^{d}$
(see Condition \ref{cond:kicond}).

The equation is considered with a Neumann boundary condition on the boundary
of the box. A precise formulation of equation \eqref{eq:model} is given in
terms of a SPDE driven by a Poisson random measure of the form in %
\eqref{eq:sdeg}. We now introduce a convenient CHNS to describe the solution
space. Let $\rho _{0}(x)=e^{-2\sum_{i=1}^{d}c_{i}x_{i}}$, $x=(x_{1},\ldots
,x_{d})^{\prime }$, where $c_{i}=\frac{V_{i}}{2D}$, $i=1,\ldots ,d$. Let $%
\mathbb{H}=L^{2}([0,l]^{d},\rho _{0}(x)dx)$. It can be checked that the
operator $A=D\Delta -V\cdot \nabla $ with Neumann boundary condition on $%
[0,l]^{d}$ has eigenvalues and eigenfunctions 
\begin{equation*}
\left\{ -\lambda _{\mathbf{j}},\phi _{\mathbf{j}}\right\} _{\mathbf{j}%
=(j_{1},\ldots ,j_{d})\in \mathbb{N}_{0}^{d}},
\end{equation*}%
where $\lambda _{\mathbf{j}}=\sum_{k=1}^{d}\lambda _{j_{k}}^{(k)}$, $\phi _{%
\mathbf{j}}=\prod_{k=1}^{d}\phi _{j_{k}}^{(k)}$, $\mathbf{j}=(j_{1},\ldots
,j_{d})$, 
\begin{equation*}
\phi _{0}^{(i)}(x)=\sqrt{\frac{2c_{i}}{1-e^{-2c_{i}l}}},\ \ \phi
_{j}^{(i)}(x)=\sqrt{\frac{2}{l}}e^{c_{i}x}\sin \left( \frac{j\pi }{l}%
x+\alpha _{j}^{i}\right) ,\;\alpha _{j}^{i}=\tan ^{-1}\left( -\frac{j\pi }{%
lc_{i}}\right)
\end{equation*}%
and 
\begin{equation*}
\lambda _{0}^{(i)}=0,\ \ \lambda _{j}^{(i)}=D\left( c_{i}^{2}+\left( \frac{%
j\pi }{l}\right) ^{2}\right) .
\end{equation*}%
Note that $\{\phi _{\mathbf{j}}\}$ forms a complete orthonormal set in $%
\mathbb{H}$. For $h\in \mathbb{H}$ and $n\in \mathbb{Z}$, define 
\begin{equation*}
\Vert h\Vert _{n}^{2}=\sum_{\mathbf{j}\in \mathbb{N}_{0}^{d}}\left\langle
h,\phi _{\mathbf{j}}\right\rangle _{\mathbb{H}}^{2}(1+\lambda _{\mathbf{j}%
})^{2n}
\end{equation*}%
and let 
\begin{equation*}
\Phi =\{\phi \in \mathbb{H}:\Vert \phi \Vert _{n}<\infty ,\forall n\in 
\mathbb{Z}\}.
\end{equation*}%
Let $\Phi _{n}$ be the completion of $\mathbb{H}$ with respect to the norm $%
\Vert \cdot \Vert _{n}$. In particular $\Phi _{0}=\mathbb{H}$. Then the
sequence $\{\Phi _{n}\}$ has all the properties stated in Section \ref%
{sec:sec2.4} for $\Phi =\cap _{n\in \mathbb{Z}}\Phi _{n}$ to be a CHNS.
Also, for each $n\in \mathbb{Z}$, $\{\Vert \phi _{\mathbf{j}}\Vert
_{n}^{-1}\phi _{\mathbf{j}}\}$ is a complete orthonormal system in $\Phi
_{n} $.

We will make the following assumption on $\nu $. 
\begin{equation}
\mbox{ For some }\delta >0,\int_{\mathbb{X}}e^{\delta a^{2}}\nu (dy)<\infty
,\;y=(x,a)\in \lbrack 0,l]^{d}\times \mathbb{R}_{+}.
\label{eq:MGcondition_example}
\end{equation}%
Here $\nu $ is a joint distribution on the possible locations and amounts of
pollutants. We now describe the precise formulation of equation %
\eqref{eq:model}. In fact we will consider a somewhat more general equation
that permits the magnitude of chemical injection to depend on the
concentration profile and also allows for nonlinear dependence on the field.
Consider the equation 
\begin{equation*}
X^{\varepsilon }(t)=x_{0}+\int_{0}^{t}b(X^{\varepsilon }(s))ds+\varepsilon
\int_{\mathbb{X}\times \lbrack 0,t]}G(X^{\varepsilon }(s-),y)N^{\varepsilon
^{-1}}(dy,ds),\;t\in \lbrack 0,T],
\end{equation*}%
where $N^{\varepsilon ^{-1}}$ is as in Section \ref{Sec:PRM}. The function $%
b:\Phi ^{\prime }\rightarrow \Phi ^{\prime }$ is defined as follows: for $%
v\in \Phi ^{\prime }$ and $\phi \in \Phi $, $b(v)[\phi ]\doteq b_{1}(v)[\phi
]+b_{0}(v)[\phi ]$, where $b_{1}(v)[\phi ]\doteq v[A\phi ]-\alpha v[\phi ]$
and $b_{0}:\Phi ^{\prime }\rightarrow \Phi ^{\prime }$ is defined by 
\begin{equation*}
b_{0}(v)[\phi ]\doteq \sum_{i=1}^{\ell }K_{i}(v[\eta _{1}],\ldots ,v[\eta
_{m}])\zeta _{i}[\phi ],\;v\in \Phi ^{\prime },\phi \in \Phi ,
\end{equation*}%
where $K_{i}:\mathbb{R}^{m}\rightarrow \mathbb{R}$ and $\{\eta
_{j}\}_{j=1}^{m}$, $\{\zeta _{i}\}_{i=1}^{\ell }$ are given elements in $%
\Phi $. Also, $G:\Phi ^{\prime }\times \mathbb{X}\rightarrow \Phi ^{\prime }$
is as follows. For $v\in \Phi ^{\prime }$, $y=(x,a)\in \mathbb{X}$ and $\phi
\in \Phi $ 
\begin{equation*}
G(v,y)[\phi ]\doteq aG_{1}(v)c_{\zeta }^{-1}\int_{B_{\zeta }(x)\cap \lbrack
0,l]^{d}}\phi (z)\rho _{0}(z)dz,
\end{equation*}%
where $G_{1}:\Phi ^{\prime }\rightarrow \mathbb{R}$ is given by 
\begin{equation*}
G_{1}(v)\doteq K_{0}(v[\eta _{1}],\ldots ,v[\eta _{m}]),v\in \Phi ^{\prime },
\end{equation*}%
and $K_{0}:\mathbb{R}^{m}\rightarrow \mathbb{R}$. Equation \eqref{eq:model}
corresponds to the case $b_{0}=0$ and $G_{1}=1$. We will make the following
assumption on $\{K_{i}\}_{i=0}^{p}$.

\begin{condition}
\label{cond:kicond}

\begin{enumerate}
\item[(a)] For some $L_{K}\in (0,\infty )$ 
\begin{equation*}
\max_{i=0,\ldots ,\ell }|K_{i}(x)-K_{i}(x^{\prime })|\leq L_{K}|x-x^{\prime
}|,\mbox{ for all }x,x^{\prime }\in \mathbb{R}^{m}.
\end{equation*}

\item[(b)] For each $i=0,\ldots ,\ell $, $K_{i}$ is differentiable and for
some $L_{DK}\in (0,\infty )$ 
\begin{equation*}
\max_{i=0,\ldots ,\ell }|\nabla K_{i}(x)-\nabla K_{i}(x^{\prime })|\leq
L_{DK}|x-x^{\prime }|,\mbox{ for all }x,x^{\prime }\in \mathbb{R}^{m}.
\end{equation*}

\item[(c)] $\nu _{0}\{(x,a):B_{\zeta }(x)\subset \lbrack 0,l]^{d}\}=1$.
\end{enumerate}
\end{condition}

Suppose that $x_{0}\in \Phi _{-p}$. We next verify that the functions $b$
and $G$ satisfy Conditions \ref{bGcond_nucl} and \ref{bGcond_nuclNew}.
Choose $q=p+r$ and $q_{1}=p+2r$ where $r>0$ is such that $\sum_{\mathbf{j}%
\in \mathbb{N}_{0}^{d}}\lambda _{\mathbf{j}}^{2}(1+\lambda _{\mathbf{j}%
})^{-2r}<\infty $. Then the embeddings $\Phi _{-p}\subset \Phi _{-q}$ and $%
\Phi _{-q}\subset \Phi _{-q_{1}}$ are Hilbert-Schmidt.

We first verify that $b$ satisfies the required conditions. Clearly $b$ is a
continuous function from $\Phi _{-p}$ to $\Phi _{-q}$. Also, for $v\in \Phi
_{-p}$ 
\begin{equation*}
\Vert b_{1}(v)\Vert _{-q}^{2}=\sum_{\mathbf{j}\in \mathbb{N}%
_{0}^{d}}(v[A\phi _{\mathbf{j}}^{q}]-\alpha v[\phi _{\mathbf{j}%
}^{q}])^{2}=\sum_{\mathbf{j}\in \mathbb{N}_{0}^{d}}(\lambda _{\mathbf{j}%
}+\alpha )^{2}(v[\phi _{\mathbf{j}}^{q}])^{2}\leq c_{\lambda }\Vert v\Vert
_{-p}^{2},
\end{equation*}%
where $c_{\lambda }=\sup_{\mathbf{j}\in \mathbb{N}_{0}^{d}}\{(\lambda _{%
\mathbf{j}}+\alpha )^{2}(1+\lambda _{\mathbf{j}})^{-2r}\}$ and the last
inequality follows on noting that, for $n\in \mathbb{Z}$, $\Vert \phi _{%
\mathbf{j}}\Vert _{n}=(1+\lambda _{\mathbf{j}})^{n}$.

Also, using Condition \ref{cond:kicond}(a) it is easily verified that for
some $C_{1}\in (0,\infty )$ 
\begin{equation}
\Vert b_{0}(v)\Vert _{-p}^{2}\leq C_{1}(1+\Vert v\Vert _{-p})^{2},\;%
\mbox{
for all }v\in \Phi _{-p}.  \label{eq:eq751O3}
\end{equation}%
Combining the above two estimates we see that $b$ satisfies Condition \ref%
{bGcond_nucl}(b). Next, using the observation that $\alpha \geq 0$ and $%
\lambda _{\mathbf{j}}\geq 0$ for all $\mathbf{j}$, we see that $2b_{1}(\phi
)[\theta _{p}\phi ]\leq 0$ for all $\phi \in \Phi $. Also, using %
\eqref{eq:eq751O3} it is immediate that 
\begin{equation*}
2b_{0}(\phi )[\theta _{p}\phi ]\leq C_{1}\Vert \phi \Vert _{-p}(1+\Vert \phi
\Vert _{-p}),\;\mbox{ for all }\phi \in \Phi .
\end{equation*}%
This shows that $b$ satisfies Condition \ref{bGcond_nucl}(c).

Once again using the nonnegativity of $-A$ and $\alpha $ we see that 
\begin{equation*}
\langle u-u^{\prime },b_{1}(u)-b_{1}(u^{\prime })\rangle _{-q}\leq 0,%
\mbox{
for all }u,u^{\prime }\in \Phi _{-p}.
\end{equation*}%
Also, by the Lipschitz property of $K_{i}$ (Condition \ref{cond:kicond}(a))
we see that 
\begin{equation*}
\Vert b_{0}(u)-b_{0}(u^{\prime })\Vert _{-q}\leq C_{2}\Vert u-u^{\prime
}\Vert _{-q},\mbox{
for all }u,u^{\prime }\in \Phi _{-p}.
\end{equation*}%
Combining the two inequalities shows that $b$ satisfies Condition \ref%
{bGcond_nucl}(d).

Next we verify that $b$ satisfies Condition \ref{bGcond_nuclNew}. Note that
for $\phi \in \Phi $ the map $\Phi _{-q}\ni v\mapsto b_{1}(v)[\phi ]\in 
\mathbb{R}$ is Fr\'{e}chet differentiable and 
\begin{equation*}
D(b_{1}(v)[\phi ])[\eta ]=\eta \lbrack A\phi ]-\alpha \eta \lbrack \phi ],%
\mbox{ for 
all }\eta \in \Phi _{-q}.
\end{equation*}%
Thus Condition \ref{bGcond_nuclNew}(a) holds trivially for $b_{1}$. Also,
from differentiability of $K_{i}$ it follows that $b_{0}(v)[\phi ]$ is Fr%
\'{e}chet differentiable and for $\eta \in \Phi _{-q}$ 
\begin{equation*}
D(b_{0}(v)[\phi ])[\eta ]=\sum_{i=1}^{\ell }\sum_{j=1}^{m}\frac{\partial }{%
\partial x_{j}}K_{i}(v[\eta _{1}],\ldots ,v[\eta _{m}])\eta \lbrack \eta
_{j}]\zeta _{i}[\phi ].
\end{equation*}%
Using the Lipschitz property of $\nabla K_{i}$ (Condition \ref{cond:kicond}%
(b)) it is now easy to check that $b_{0}$ and consequently $b$ satisfies
Condition \ref{bGcond_nuclNew}(a) as well.

Next, for $v\in \Phi _{-p}$ and $\eta \in \Phi _{-q}$ 
\begin{equation}
\sum_{\mathbf{j}\in \mathbb{N}_{0}^{d}}|D(b_{1}(v)[\phi _{\mathbf{j}%
}^{q_{1}}])[\eta ]|^{2}=\sum_{\mathbf{j}\in \mathbb{N}_{0}^{d}}|\eta \lbrack
(A-\alpha )\phi _{\mathbf{j}}^{q_{1}}]|^{2}=\sum_{\mathbf{j}\in \mathbb{N}%
_{0}^{d}}(\alpha +\lambda _{\mathbf{j}})^{2}(1+\lambda _{\mathbf{j}%
})^{-2r}|\eta \lbrack \phi _{\mathbf{j}}^{q_{1}}]|^{2}\leq c_{\lambda }\Vert
\eta \Vert _{-q}^{2}.  \label{eq:eq1006O3}
\end{equation}%
Also using the linear growth of $\nabla K_{i}$ (which follows from Condition %
\ref{cond:kicond}(b)), there is a $C_{3}\in (0,\infty )$ such that 
\begin{equation}
\sup_{\{v\in \Phi _{-p}:\Vert v\Vert _{-p}\leq m_{T}\}}\sum_{\mathbf{j}\in 
\mathbb{N}_{0}^{d}}|D(b_{0}(v)[\phi _{\mathbf{j}}^{q_{1}}])[\eta ]|^{2}\
\leq C_{3}\Vert \eta \Vert _{-q}^{2}\sum_{i=1}^{\ell }\sum_{\mathbf{j}\in 
\mathbb{N}_{0}^{d}}|\zeta _{i}[\phi _{\mathbf{j}}^{q_{1}}]|^{2}.
\label{eq:eq1015}
\end{equation}%
Combining \eqref{eq:eq1006O3} and \eqref{eq:eq1015} we get that for $\eta
\in \Phi _{-q}$ and $\phi \in \Phi _{q_{1}}$ 
\begin{equation*}
\sup_{\{v\in \Phi _{-p}:\Vert v\Vert _{-p}\leq m_{T}\}}|A_{v}(\eta )[\phi
]|^{2}\leq \left( \sum_{\mathbf{j}\in \mathbb{N}_{0}^{d}}|\langle \phi ,\phi
_{\mathbf{j}}^{q_{1}}\rangle _{q_{1}}|A_{v}(\eta )[\phi _{\mathbf{j}%
}^{q_{1}}]|\right) ^{2}\leq 2(c_{\lambda }+C_{4})\Vert \phi \Vert
_{q_{1}}^{2}\Vert \eta \Vert _{-q}^{2},
\end{equation*}%
where $C_{4}=C_{3}\sum_{i=1}^{\ell }\sum_{\mathbf{j}\in \mathbb{N}%
_{0}^{d}}|\zeta _{i}[\phi _{\mathbf{j}}^{q_{1}}]|^{2}$. This shows that $%
\eta \mapsto A_{v}(\eta )$ is continuous and that $b$ satisfies 
\eqref{eq:
eq429a}.

Using the non-negativity of $-A$ and $\alpha $ it follows that for all $\eta
\in \Phi _{-q}$, $A_{v}^{1}(\eta )\doteq D(b_{1}(v)[\cdot ])[\eta ]\in \Phi
_{-q_{1}}$ satisfies 
\begin{equation*}
\langle \eta ,A_{v}^{1}(\eta )\rangle _{-q_{1}}\leq 0.
\end{equation*}%
Also, from linear growth of $DK_{i}$, $A_{v}^{0}(\eta )\doteq
D(b_{0}(v)[\cdot ])[\eta ]\in \Phi _{-q_{1}}$ satisfies, for some $C_{5}\in
(0,\infty )$ 
\begin{equation*}
\sup_{\{v\in \Phi _{-p}:\Vert v\Vert _{-p}\leq m_{T}\}}\langle \eta
,A_{v}^{0}(\eta )\rangle _{-q_{1}}\leq C_{5}\Vert \eta \Vert _{-q_{1}}^{2},%
\mbox{ for all }\eta \in \Phi _{-q}.
\end{equation*}%
Combining the last two estimates, for all $\eta _{1},\eta _{2}\in \Phi _{-q}$
\begin{align*}
\sup_{\{v\in \Phi _{-p}:\Vert v\Vert _{-p}\leq m_{T}\}}\langle \eta
_{1}-\eta _{2},A_{v}(\eta _{1})-A_{v}(\eta _{2})\rangle _{-q_{1}}& \leq
\sup_{\{v\in \Phi _{-p}:\Vert v\Vert _{-p}\leq m_{T}\}}\langle \eta
_{1}-\eta _{2},A_{v}^{0}(\eta _{1}-\eta _{2})\rangle _{-q_{1}} \\
& \leq C_{5}\Vert \eta _{1}-\eta _{2}\Vert _{-q_{1}}^{2}.
\end{align*}%
This verifies \eqref{eq: eq429b}.

Next noting for all $u\in \Phi _{-p}$ and $\phi \in \Phi $ that $u[(A-\alpha
)(\theta _{q}\phi )]\leq C_{6}\Vert u\Vert _{-p}\Vert \phi \Vert _{-q}$ and $%
\phi \lbrack (A-\alpha )(\theta _{q}\phi )]\leq 0$, we see that for all $%
u\in \Phi _{-p}$ 
\begin{equation*}
\sup_{\{v\in \Phi _{-p}:\Vert v\Vert _{-p}\leq m_{T}\}}2A_{v}^{1}(\phi
+u)[\theta _{q}\phi ]\leq 2C_{6}\Vert u\Vert _{-p}\Vert \phi \Vert _{-q}.
\end{equation*}%
Also, using the linear growth of $\nabla K_{i}$ we see that, for some $%
C_{7}\in (0,\infty )$ 
\begin{equation*}
\sup_{\{v\in \Phi _{-p}:\Vert v\Vert _{-p}\leq m_{T}\}}2A_{v}^{0}(\phi
+u)[\theta _{q}\phi ]\leq C_{7}(\Vert \phi \Vert _{-q}+\Vert u\Vert
_{-p})\Vert \phi \Vert _{-q},\mbox{
for all }u\in \Phi _{-p},\phi \in \Phi .
\end{equation*}%
From the last two inequalities we have \eqref{eq: eq429c}.

Conditions for $G$ are verified in a similar fashion. In particular note
that for $\phi \in \Phi _{0}$ and $x\in \mathbb{R}^{d}$ such that $B_{\zeta
}(x)\subset \lbrack 0,l]^{d}$ 
\begin{equation*}
ac_{\zeta }^{-1}\int_{B_{\zeta }(x)}|\phi (z)|\rho _{0}(z)dz\leq aC_{8}\Vert
\phi \Vert _{0},
\end{equation*}%
where $C_{8}=c_{\zeta }^{-1/2}\sup_{z\in \lbrack 0,l]^{d}}\rho _{0}(z)$.
From this it is immediate that for some $C_{9}<\infty $, $G$ satisfies
Condition \ref{bGcond_nucl} with $M_{G}(y)=L_{G}(y)=aC_{9}$, $y=(x,a)\in 
\mathbb{X}$. Note that in view of (\ref{eq:MGcondition_example}) $%
M_{G},L_{G} $ satisfy part (d) of Condition \ref{bGcond_nuclNew}. Remaining
parts of this condition are verified similarly and we omit the details.

\setcounter{section}{0}

\setcounter{equation}{0} \renewcommand{\theequation}{\thesection.%
\arabic{equation}}

\end{document}